\documentclass[microtype,utf8]{gtpart}

\RequirePackage{booktabs}
\RequirePackage{multirow}
\RequirePackage{xspace}
\RequirePackage{graphics}
\RequirePackage{graphicx}
\RequirePackage[utf8]{inputenc}
\RequirePackage[all]{xy}
\usepackage{enumitem}

\title[Asymptotically cylindrical Calabi--Yau 3--folds]{Asymptotically
cylindrical Calabi--Yau 3--folds from weak Fano 3--folds}

\author[A~Corti]{Alessio~Corti}
\givenname{Alessio}
\surname{Corti}
\address{Department of Mathematics\\
Imperial College London\\\newline
London SW7 2AZ\\
UK}
\email{a.corti@imperial.ac.uk}
\urladdr{}

\author[M~Haskins]{Mark~Haskins}
\givenname{Mark}
\surname{Haskins}
\email{m.haskins@imperial.ac.uk}
\urladdr{}

\author[J~Nordström]{Johannes~Nordstr\"om}
\givenname{Johannes}
\surname{Nordstr\"om}
\email{j.nordstrom@imperial.ac.uk}
\urladdr{}

\author[T~Pacini]{Tommaso~Pacini}
\givenname{Tommaso}
\surname{Pacini}
\address{Scuola Normale Superiore\\\newline
Piazza dei Cavalieri 7\\
56126 Pisa\\
Italy}
\email{tommaso.pacini@sns.it}
\urladdr{}

\volumenumber{17}
\issuenumber{4}
\publicationyear{2013}
\papernumber{44}
\startpage{1955}
\endpage{2059}

\doi{}
\MR{}
\Zbl{}



\accepted{4 March 2013}
\published{15 July 2013}
\publishedonline{15 July 2013}
\proposed{Richard Thomas}
\seconded{Ronald Fintushel, Yasha Eliashberg}
\corresponding{}
\editor{Nicholas}
\version{2}

\let\xysavmatrix\xymatrix
\def\xymatrix{\disablesubscriptcorrection\xysavmatrix}
\AtBeginDocument{\let\bar\wbar\let\tilde\wtilde\let\hat\what}

\usepackage{slashed}

\numberwithin{equation}{section}
\numberwithin{table}{section}

\newtheorem{theorem}{Theorem}[section]
\newtheorem{lemma}{Lemma}[section]
\newtheorem{prop}{Proposition}[section]
\newtheorem{corollary}{Corollary}[section]

\theoremstyle{definition}
\newtheorem{definition}{Definition}[section]
\newtheorem{example}{Example}[section]

\theoremstyle{remark}
\newtheorem{remark}{Remark}[section]
\newtheorem*{remark*}{Remark}
\newtheorem*{assumption}{Assumption}
\newtheorem{convention}{Convention}[section]

\makeatletter
\let\c@lemma\c@theorem
\let\c@prop\c@theorem
\let\c@corollary\c@theorem
\let\c@definition\c@theorem
\let\c@example\c@theorem
\let\c@remark\c@theorem
\let\c@convention\c@theorem
\makeatother

\makeatletter \let\c@equation\c@theorem \makeatother

\newcommand{\abs}[1]{\lvert#1\rvert}
\newcommand{\sunitary}[1]{\textup{SU$(#1)$}}
\newcommand{\sunitaryn}{\textup{SU$(n)$}}
\newcommand{\Proj}{\operatorname{Proj}}
\newcommand{\Pic}{\operatorname{Pic}}
\newcommand{\Hol}{\operatorname{Hol}}
\newcommand{\rank}{\operatorname{rk}}
\newcommand{\weil}{\operatorname{Cl}}
\newcommand{\ex}{\operatorname{Ex}}
\newcommand{\gdiv}{\operatorname{div}}

\newcommand{\N}{\mathbb{N}}
\newcommand{\PP}{\mathbb{P}}
\newcommand{\CP}{\mathbb{P}}
\newcommand{\Sph}{\mathbb{S}}
\newcommand{\ra}{\rightarrow}
\newcommand{\Hom}{\operatorname{Hom}}
\newcommand{\Imag}{\operatorname{Im}}
\newcommand{\hk}{hyper-Kähler\xspace}

\newcommand{\gtwo}{\ensuremath{\textup{G}_2}\xspace}

\newcommand{\nef}{\mathit{NE}}
\newcommand{\nefb}{\overline{\nef}}
\newcommand{\cont}{\operatorname{cont}}

\DeclareMathOperator{\codim}{codim}
\DeclareMathOperator{\Aut}{Aut}
\DeclareMathOperator{\Cl}{Cl}
\DeclareMathOperator{\uPic}{\underline{Pic}}
\DeclareMathOperator{\Amp}{Amp}

\newcommand{\oo}{\mathcal{O}}
\newcommand{\bfa}{\mathbf{a}}
\newcommand{\QQ}{\mathbb{Q}}
\newcommand{\RR}{\mathbb{R}}
\newcommand{\CC}{\mathbb{C}}
\newcommand{\TT}{\mathbb{T}}
\newcommand{\bbS}{\mathbb{S}}
\newcommand{\ZZ}{\mathbb{Z}}
\newcommand{\FF}{\mathbb{F}}

\newcommand{\tsfrac}[2]{\textstyle\frac{#1}{#2}}

\newcommand{\bbz}{\mathbb{Z}}
\newcommand{\bbr}{\mathbb{R}}
\newcommand{\bbc}{\mathbb{C}}
\newcommand{\bbp}{\mathbb{P}}
\newcommand{\bbrp}{\mathbb{R}^{+}}

\newcommand{\into}{\hookrightarrow}
\newcommand{\gtmfd}{\gtwo--manifold}

\DeclareMathAlphabet{\df}{U}{eus}{m}{n}
\newcommand{\hdg}{h}

\newcommand{\acls}[1]{\abs{-K_{#1}}}
\newcommand{\anglen}{\vartheta}
\newcommand{\blow}{\pi}
\newcommand{\cyl}{\infty}
\renewcommand{\mod}{\textup{ mod }}

\newcommand{\gen}[1]{\langle#1\rangle}

\renewcommand{\theenumi}{\textup{(\roman{enumi})}}

\begin{document}

\begin{asciiabstract}
We prove the existence of asymptotically cylindrical (ACyl) Calabi-Yau
3-folds starting with (almost) any deformation family of smooth weak
Fano 3-folds.  This allow us to exhibit hundreds of thousands of new
ACyl Calabi-Yau 3-folds; previously only a few hundred ACyl Calabi-Yau
3-folds were known.  We pay particular attention to a subclass of weak
Fano 3-folds that we call semi-Fano 3-folds.  Semi-Fano 3-folds
satisfy stronger cohomology vanishing theorems and enjoy certain
topological properties not satisfied by general weak Fano 3-folds, but
are far more numerous than genuine Fano 3-folds.  Also, unlike Fanos
they often contain P^1s with normal bundle O(-1) + O(-1), giving rise
to compact rigid holomorphic curves in the associated ACyl Calabi-Yau
3-folds.
We introduce some general methods to compute the basic topological
invariants of ACyl Calabi-Yau 3-folds constructed from semi-Fano
3-folds, and study a small number of representative examples in
detail.  Similar methods allow the computation of the topology in many
other examples.
All the features of the ACyl Calabi-Yau 3-folds studied here find
application in [arXiv:1207.4470] where we construct many new compact
G_2-manifolds using Kovalev's twisted connected sum construction.
ACyl Calabi-Yau 3-folds constructed from semi-Fano 3-folds are
particularly well-adapted for this purpose.
\end{asciiabstract}

\begin{webabstract}
  We prove the existence of asymptotically cylindrical (ACyl)
  Calabi--Yau 3--folds starting with (almost) any deformation family
  of smooth \emph{weak Fano} 3--folds.  This allow us to exhibit
  hundreds of thousands of new ACyl
  Calabi--Yau 3--folds; previously only a few hundred ACyl
  Calabi--Yau 3--folds were known.  We pay particular attention to a
  subclass of weak Fano 3--folds that we call \emph{semi-Fano}
  3--folds.  Semi-Fano 3--folds satisfy stronger cohomology
  vanishing theorems and enjoy certain topological properties not
  satisfied by general weak Fano 3--folds, but are far more numerous
  than genuine Fano 3--folds.  Also, unlike Fanos they often contain
  $\mathbb{P}^1$s with normal bundle $\mathcal{O}(-1) \oplus
  \mathcal{O}(-1)$, giving rise to compact rigid holomorphic curves in
  the associated ACyl Calabi--Yau 3--folds.

  We introduce some general methods to compute the basic topological
  invariants of ACyl Calabi--Yau 3--folds constructed from semi-Fano
  3--folds, and study a small number of representative examples in
  detail.  Similar methods allow the computation of the topology in
  many other examples.

  All the features of the ACyl Calabi--Yau 3--folds studied here find
  application in [arXiv:1207.4470]
  where we construct many new compact $\mathrm{G}_2$--manifolds using
  Kovalev's twisted connected sum construction.  ACyl Calabi--Yau 3--folds
  constructed from semi-Fano 3--folds are particularly well-adapted for
  this purpose.
\end{webabstract}

\begin{abstract}
  We prove the existence of asymptotically cylindrical (ACyl)
  Calabi--Yau 3--folds starting with (almost) any deformation family
  of smooth \emph{weak Fano} 3--folds.  This allow us to exhibit
  hundreds of thousands of new ACyl
  Calabi--Yau 3--folds; previously only a few hundred ACyl
  Calabi--Yau 3--folds were known.  We pay particular attention to a
  subclass of weak Fano 3--folds that we call \emph{semi-Fano}
  3--folds.  Semi-Fano 3--folds satisfy stronger cohomology
  vanishing theorems and enjoy certain topological properties not
  satisfied by general weak Fano 3--folds, but are far more numerous
  than genuine Fano 3--folds.  Also, unlike Fanos they often contain
  $\mathbb{P}^1$s with normal bundle $\mathcal{O}(-1) \oplus
  \mathcal{O}(-1)$, giving rise to compact rigid holomorphic curves in
  the associated ACyl Calabi--Yau 3--folds.

  We introduce some general methods to compute the basic topological
  invariants of ACyl Calabi--Yau 3--folds constructed from semi-Fano
  3--folds, and study a small number of representative examples in
  detail.  Similar methods allow the computation of the topology in
  many other examples.

  All the features of the ACyl Calabi--Yau 3--folds studied here find
  application in~\cite{chnp2} where we construct many new compact
  $\mathrm{G}_2$--manifolds using Kovalev's twisted connected sum
  construction.  ACyl Calabi--Yau 3--folds constructed from semi-Fano
  3--folds are particularly well-adapted for this purpose.
\end{abstract}

\maketitle

\section{Introduction}

Compact Calabi--Yau manifolds have been studied intensively ever since Yau's resolution of the 
Calabi conjecture~\cite{yau:CY} allowed algebraic geometers to produce them in abundance. 
Nevertheless, some fundamental questions about compact Calabi--Yau manifolds even in dimension 
three remain open. For example, are there finitely many or infinitely many topological types 
of nonsingular Calabi--Yau 3--fold?

There has also been important work on complete noncompact Kähler Ricci-flat (KRF) metrics by many authors:
Calabi, Yau, Eguchi--Hansen, Gibbons--Hawking, Hitchin, Kronheimer, Anderson--Kronheimer--LeBrun, 
Atiyah--Hitchin, Tian--Yau, Joyce, Nakajima,  Biquard and Carron to name only a small selection.
Nevertheless, compared to the compact nonsingular case, current understanding of noncompact KRF metrics is 
much less complete and demands further study; several open questions in this area go back as far as Yau's 1978 ICM address. 

The simplest classes of noncompact KRF metrics are:
\begin{enumerate}
\item[(a)]
those of maximal volume growth, that is, Euclidean volume growth;
\item[(b)]
those of minimal volume growth, that is, linear volume growth.
\end{enumerate}
The maximal volume growth case -- especially the class of so-called ALE
metrics -- has already attracted considerable attention,
for example,  Kronheimer's classification results for ALE \hk 4--manifolds~\cite{kron:ale:classify} and Joyce's higher dimensional 
existence results~\cite[Section~8]{joyce:holonomybook};
part of the reason for the focus on the ALE case has been the intimate link to
the theory of (noncollapsed) metric degenerations of compact Einstein manifolds with bounded diameter. 
Another obvious model for noncollapsed metric degenerations of compact Einstein manifolds is provided by the development of 
long ``almost cylindrical necks''. For this reason it is important to understand asymptotically cylindrical (ACyl) Einstein metrics. 
The simplest class of such ACyl Einstein metrics are the asymptotically cylindrical 
Calabi--Yau metrics studied in the present paper; see also Haskins--Hein--Nordstr\"om~\cite{hhn}.

ACyl Calabi--Yau \emph{3--folds} play a distinguished role because they can also be used as 
building blocks in Kovalev's twisted connected sum construction of compact 
manifolds with holonomy \gtwo: see Kovalev~\cite{kovalev:connectsums}, Kovalev--Lee~\cite{kovalev:lee} 
and the more recent developments in Corti--Haskins--Nordstr\"om--Pacini~\cite{chnp2}.
The twisted connected sum construction -- first developed in~\cite{kovalev:connectsums} -- constituted 
a major advance in the understanding of compact \mbox{$\gtwo$--manifolds}; 
along with Joyce's original  orbifold resolution
construction~\cite[Sections~11 and~12]{joyce:holonomybook}
it remains one of only two methods available to produce compact \gtwo--manifolds.

Given a pair of ACyl Calabi--Yau 3--folds $V_{\pm}$ the twisted connected sum construction 
gives a way to combine the pair of noncompact ACyl $7$--manifolds $\Sph^{1} \times V_{\pm}$ -- both 
of which have holonomy $\sunitary{3} \subset \gtwo$ -- to construct a compact $7$--manifold 
with holonomy the full group \gtwo. The twisted connected sum construction is possible only when 
a certain compatibility between the cylindrical ends of $V_{\pm}$ can be arranged; studying this 
``matching'' problem for pairs of ACyl Calabi--Yau 3--folds is therefore very important 
for our applications to \gtwo--geometry in~\cite{chnp2}.

While we know the existence of huge numbers of deformation classes of compact Calabi--Yau 3--folds, until the 
present paper only a couple of hundred families of ACyl Calabi--Yau 3--folds were known. 
In the present paper we prove that it is possible to construct deformation families of ACyl Calabi--Yau 3--folds from 
(almost) any deformation family of smooth \emph{weak Fano} 3--folds. 
As a consequence we prove that there are at least several hundred thousand deformation classes of ACyl Calabi--Yau 3--folds.

A \emph{Fano} 3--fold $Y$ is a smooth projective variety for which $-K_{Y}$ is ample or positive: 
complex projective space $\CP^{3}$, smooth quadrics, cubics and quartics in $\CP^{4}$ 
being the simplest examples.  
Fano 3--folds have been important objects in algebraic geometry since Fano's work in the 1930s 
and are still very much an active research area in contemporary algebraic geometry.  
A \emph{weak Fano} 3--fold\footnote{some authors call this an \emph{almost Fano} 3--fold.}
is a smooth projective 3--fold for which $-K_{Y}$ is big and nef (but not ample). 
Differential geometers are encouraged to think of a line bundle being big and nef as the algebraic--geometric 
formulation of admitting a hermitian metric whose curvature is  sufficiently semi-positive. 
All weak Fano 3--folds can be obtained by choosing suitable resolutions of mildly singular Fano 3--folds.

A number of properties of Fano manifolds generalise without too much difficulty to weak Fanos; 
we replace applications of the Kodaira vanishing theorem with its generalisation the Kawamata--Viehweg 
vanishing theorem.
Kovalev~\cite{kovalev:connectsums} used Fano 3--folds to construct ACyl
Calabi--Yau 3--folds with ends asymptotic to $\C^{*} \times S$ where $S$ is a
smooth K3 surface and suggested that other constructions of suitable ACyl Calabi--Yau 3--folds 
might be possible~\cite[page~148]{kovalev:connectsums};  
we prove that starting only with a \emph{weak Fano} 3--fold 
(satisfying one further very mild restriction which is also needed even in the Fano case)
we can still construct ACyl Calabi--Yau 3--folds with ends asymptotic
to $\C^{*} \times S$. 
However, in order to solve the ``matching'' problem for pairs of ACyl
Calabi--Yau 3--folds constructed from weak Fano 3--folds it turns out
to be important to distinguish the subclass of \emph{semi-Fano}
3--folds, that is, weak Fano 3--folds whose anticanonical morphism is a
semi-small map.\footnote{There seems to be no established terminology
  for this particular subclass of weak Fano 3--folds, so the term
  \emph{semi-Fano} is our invention; it is intended to suggest that a
  \emph{semi}-Fano 3--fold has \emph{semi}-small anticanonical
  morphism.  Warning: Chan et al~\cite{chan} used the term semi-Fano
  manifold to mean something even weaker than weak Fano, that is, a complex
  manifold for which $-K_{Y}$ is nef (but not necessarily big).}

There are two principal advantages in generalising from Fano to weak Fano or semi-Fano 3--folds. 
It is well-known that there are exactly
$105$ deformation families of smooth Fano \mbox{3--folds}
(see Iskovskih~\cite{isk:fano1,isk:fano2},
Mori--Mukai~\cite{MoMuk,MoMua,MoMub}, Mukai--Umemura~\cite{MuUm} and
Takeuchi~\cite{Take}): in the paper, we will
refer to this result as the ``Iskovskih--Mori--Mukai classification''.
On the other hand, there are at least hundreds of thousands of
deformation families of smooth weak Fano or semi-Fano 3--folds and
their topology is less restrictive than for Fano 3--folds;
unlike the Fano case there is at present no classification theory for weak Fano or semi-Fano 3--folds except 
under very special geometric assumptions. Thus generalising from Fano to weak Fano or semi-Fano 3--folds 
allows us to construct a significantly larger number of ACyl Calabi--Yau 3--folds. 

For applications to the twisted connected sum construction of compact \mbox{\gtwo--manifolds}
the following feature is also important; 
whereas on any Fano 3--fold %
the anticanonical class satisfies $-K_{Y}\cdot C> 0$ for any complex curve $C$, 
weak Fano 3--folds can contain special complex curves $C$ for which $K_{Y}\cdot C =0$ 
(the weakening of $-K_{Y}$ being positive to sufficiently semi-positive is crucial here).  
Moreover, in many cases $C$ is a smooth rational curve with normal bundle 
$\mathcal{O}(-1)\oplus \mathcal{O}(-1)$ 
(where $\mathcal{O}(d)$ denotes $\mathcal{O}_{\CP^{1}}(d)$).
 In particular, $C$ is rigid, 
that is, it  has no infinitesimal (holomorphic) deformations.
These special $K$--trivial curves $C$ in weak Fanos allow us to construct compact rigid curves in 
the associated (noncompact) ACyl Calabi--Yau 3--folds. 
The fact that we can construct compact holomorphic curves 
in our ACyl Calabi--Yau 3--folds and that these curves have no infinitesimal deformations will be key to our 
construction of rigid associative 3--folds in compact $\gtwo$--manifolds~\cite{chnp2}.

We also discuss the following topics in some detail 
(keeping in mind applications of ACyl Calabi--Yau 3--folds to the twisted connected sum construction of
\gtwo--manifolds):
\begin{enumerate}
\item
the topology of ACyl Calabi--Yau 3--folds: see \fullref{sec:blocks};
\item
which \hk K3 surfaces can appear as the ACyl limits of our ACyl Calabi--Yau 3--folds: see \fullref{sec:div};
\item
some representative ACyl Calabi--Yau 3--folds obtained from semi-Fano \mbox{3--folds} -- including 
computations of the topology of these examples and the number of rigid holomorphic curves they contain: see 
\fullref{sec:examples};
\item
some general methods available for constructing (and in some cases classifying) 
weak Fano  and semi-Fano 3--folds and some indication how the methods used in (iii) can be deployed 
in this more general context: see \fullref{S:wk:fano:examples}.
\end{enumerate}

We now describe the structure of the rest of the paper.

\fullref{sec:ac_CY} introduces (exponentially) ACyl Calabi--Yau manifolds and explains how to construct 
ACyl Calabi--Yau structures on certain types of quasiprojective manifold: see \fullref{thm:acyl_calabi}.
Underpinning \fullref{thm:acyl_calabi} is an analytic existence theorem for ACyl Calabi--Yau manifolds 
recently proven by Haskins--Hein--Nordstr\"om~\cite[Theorem~D]{hhn}; this result is related to previous work 
of Tian--Yau~\cite{tian:yau} and Kovalev~\cite{kovalev:connectsums}. Building on the previous work of Tian--Yau,
Kovalev claimed to prove the existence of exponentially asymptotically cylindrical Calabi--Yau manifolds, 
improving substantially the asymptotics previously established by Tian--Yau.
Unfortunately Kovalev's proof of the improved asymptotics contains an error 
(see the discussion following the statement of \fullref{thm:acyl_calabi}
and also~\cite{hhn} for further details). Other errors in Kovalev~\cite{kovalev:connectsums} occur in the construction of 
\hk rotations (especially Lemma 6.47 which is used in the proof of the main Theorem 6.44) 
while several other points are unclear. For this reason, in both
this paper and in~\cite{chnp2} we chose not to rely on arguments
from~\cite{kovalev:connectsums}, and to give proofs or alternative
references for
the main results we need.
To this end~\cite{hhn} gives a short self-contained proof of the existence of exponentially
asymptotically cylindrical Calabi--Yau metrics that also bypasses the difficult
existence theory of Tian--Yau~\cite{tian:yau}.

A significant fraction of this paper then concerns trying to find a large number of quasiprojective 3--folds satisfying the 
hypotheses of \fullref{thm:acyl_calabi}. 
In \fullref{prop:onestage} we show that if we can find a closed Kähler
3--fold $Y$ with an anticanonical pencil that has some smooth member and whose
base locus is a smooth
curve, then blowing up that curve gives a 3--fold satisfying the hypotheses of
\fullref{thm:acyl_calabi}, and hence an ACyl Calabi--Yau 3--fold.
In turn, almost any weak Fano 3--fold satisfies the hypotheses of \fullref{prop:onestage}.
To prove this and to show the relative abundance of weak Fano 3--folds
requires a certain amount of algebro-geometric background; this background is
developed in Sections~\ref{sec:alg-geom} and~\ref{sec:weak-fano-3}.

\fullref{sec:alg-geom} contains some material from algebraic
geometry needed for our discussion of weak Fano 3--folds.  We have
included this algebro--geometric material in an attempt to make the
paper accessible to a wide readership.  The first part of the section
deals with various notions of weak positivity for line bundles on
projective manifolds and related vanishing theorems; these vanishing
theorems generalise the classical Kodaira vanishing theorem (and its
extension due to Akizuki--Nakano) for ample line bundles.  The key
results from this section are the Kawamata--Viehweg vanishing theorem
for big and nef line bundles and the Sommese--Esnault--Viehweg
vanishing result for $l$--ample line bundles. Also important for us is
the Lefschetz theorem for semi-small morphisms; this is a special case
of Goresky--MacPherson's vast generalization of the classical
Lefschetz hyperplane theorem allowing a weaker positivity assumption
on the line bundle than ampleness.

The second part of \fullref{sec:alg-geom} contains material on
mildly singular 3--folds and their crepant and small resolutions.  We
are interested in Gorenstein terminal and canonical 3--fold
singularities; the anticanonical model of a smooth weak Fano 3--fold
is a Fano \mbox{3--fold} with Gorenstein canonical singularities: see
\fullref{R:weak:fano:ac:model}.  The simplest terminal 3--fold
singularity, the ordinary double point (ODP for short), or ordinary
node, plays a particularly important role throughout the paper.
Conversely, given a mildly singular Fano 3--fold we can often
construct smooth weak Fano 3--folds by finding appropriate
resolutions. In the terminal singularities case any crepant resolution
is a so-called small resolution, that is, the exceptional set contains no
divisors. The existence of a small resolution of a singular variety
$X$ forces it to be non--$\Q$--factorial, that is, there are Weil divisors on
$X$ no multiple of which are Cartier.  We explain the intimate link
between small birational morphisms with target $X$ and such Weil
divisors on $X$.  An important role is played by the defect of a
Gorenstein canonical 3--fold $X$; the defect quantifies the failure of
$X$ to be $\Q$--factorial.  We also recall some basic properties of
flops in dimension three; for many weak Fano 3--folds we can use flops
to produce many non-isomorphic weak Fano 3--folds from a single weak
Fano 3--fold.  The final part of the section recalls some basic
terminology and facts from Mori theory for 3--folds; this is used only
in \fullref{S:wk:fano:examples} in our discussion of the
classification scheme for weak Fano 3--folds with Picard rank
$\rho=2$.

\fullref{sec:weak-fano-3} defines weak Fano 3--folds and recalls a number of their basic properties. 
Foremost among these properties is \fullref{thm:reid} (due to Reid and Paoletti): 
a general anticanonical divisor in a nonsingular 
weak Fano 3--fold is a nonsingular K3 surface; this is the fundamental property that allows us to 
construct ACyl Calabi--Yau 3--folds out of weak Fano 3--folds.
Propositions~\ref{prop:onestage} and~\ref{prop:sequence} show how one can obtain 
quasiprojective 3--folds on which we can construct ACyl Calabi--Yau structures by blowing up suitable 
curves in suitable K\"ahler 3--folds; the earlier material shows that suitable 3--folds include almost any weak Fano 3--fold. 
These results are central to the paper.

As mentioned above, we also introduce an important subclass of weak Fano
3--folds which we call \emph{semi-Fano} 3--folds:
the anticanonical morphism of a semi-Fano 3--fold is a \emph{semi-small}
birational morphism, that is, it contracts no divisor to a point.
Although weak Fano 3--folds suffice to construct ACyl Calabi--Yau 3--folds, for 
applications to the construction of compact \gtwo--manifolds using the twisted connected sum construction,  
we will often need to restrict to ACyl Calabi--Yau 3--folds obtained from semi-Fano 3--folds. 
The basic advantage is the stronger  cohomology vanishing theorems available for semi-Fano 3--folds.

\fullref{sec:blocks} is concerned with computing the topology of ACyl Calabi--Yau 3--folds 
and in particular the topology of the ACyl Calabi--Yau 3--folds we construct out of semi-Fano 3--folds. 
We compute the full integral cohomology groups of our ACyl Calabi--Yau 3--folds
and note in particular that the only potential source of torsion comes from
$H^3$ of the semi-Fano. 
We do not know any semi-Fano 3--folds for which $H^{3}$ has torsion but we have no general 
proof of its absence. We also establish simply-connectedness of our ACyl 
Calabi--Yau 3--folds and study the second Chern class $c_{2}$, particularly properties related to its divisibility.
These results on the primary topological invariants of ACyl Calabi--Yau 3--folds play an important role in~\cite{chnp2}; 
there they are used to identify for the first time the diffeomorphism type of many compact 
$\gtwo$--manifolds.

\fullref{sec:div} studies anticanonical divisors in semi-Fano 3--folds  in detail. 
By \fullref{thm:reid} any general anticanonical divisor in a weak Fano 3--fold is a smooth K3 surface. 
A natural geometric question about ACyl Calabi--Yau 3--folds constructed from a weak Fano 3--fold 
is the following:  which K3 surfaces  can appear as asymptotic limits of our ACyl Calabi--Yau 3--folds 
as we vary both the weak Fano 3--fold in its deformation 
class and the chosen smooth anticanonical divisor?
Addressing this question turns out to be crucial to the construction of so-called \hk rotations between pairs of 
ACyl Calabi--Yau 3--folds and therefore to the construction of 
compact \gtwo--manifolds via the twisted connected sum construction.

To answer this question we need to develop some appropriate
moduli/deformation theory.  On the K3 side this requires recalling
basic facts about lattice polarised K3 surfaces and versions of the
Torelli theorem in this setting.  We also need to extend Beauville's
results~\cite{beauville:fano} about the moduli stack parameterising pairs
$(Y,S)$ where $Y$ belongs to a given deformation family of smooth Fano
3--folds and $S \in \abs{-K_{Y}}$ is a smooth K3 section.
The key observation -- see \fullref{thm:fano_stack_smooth} -- is
that the appropriate moduli
stack is still smooth when $Y$ is a semi-Fano 3--fold; here we use
the stronger cohomology vanishing theorems available for semi-Fano
3--folds.  The immediate payoff is \fullref{thm:fanok3} which
gives us a good understanding of which K3 surfaces appear as smooth
anti\-canonical divisors in a deformation class of semi-Fano
3--folds. It is likely that most of these facts hold, with
appropriate modification, for more general weak Fano \mbox{3--folds}
but we do not pursue this here; however see for instance the recent paper
by Sano~\cite{sano13}.

\fullref{sec:examples} constructs a handful of ACyl Calabi--Yau 3--folds from a carefully chosen selection 
of Fano and semi-Fano 3--folds and computes the topology of these ACyl Calabi--Yau 3--folds in detail using 
the results from \fullref{sec:blocks}. In this section we only construct a very small number of typical examples 
making no attempt to be systematic. Similar methods can be used to produce many more ACyl Calabi--Yau 3--folds 
and to compute their topology.

\fullref{S:wk:fano:examples} gives many further examples of semi-Fano 3--folds from which one can 
construct many more ACyl Calabi--Yau 3--folds. Our basic aim is to back up our assertion that there are 
\emph{many} more weak Fano or semi-Fano 3--folds than Fano 3--folds. Unlike smooth Fano 3--folds, 
smooth weak Fano 3--folds are far from being classified and even in the longer-term such a classification 
may in practice be out of reach.
Various classes of weak Fano 3--folds with special geometric or topological properties are much closer to being 
classified. We consider in some detail several such special classes: (a) weak Fano 3--folds with Picard rank $\rho=2$, 
(b) toric weak Fano 3--folds and (c) weak Fano 3--folds obtained by small resolutions of nodal cubics.

Thanks to recent work of various authors -- including Arap--Cutrone--Marshburn~\cite{arap:c:m}, 
Blanc--Lamy~\cite{blanc:lamy},
Cutrone--Marshburn~\cite{cutrone:marshburn},  Jahnke--Peternell--Radloff~\cite{peternell1,peternell2}, Kalo\-ghiros \cite{kaloghiros:thesis} and 
Takeuchi~\cite{takeuchi} -- class (a) is known to 
consist of  over 150 distinct deform\-ation classes of semi-Fano 3--folds; 
many of these can be obtained by blowing up an appropriate smooth irreducible curve in an appropriate 
smooth rank one Fano \mbox{3--fold}. This makes it relatively straightforward to determine many of the 
basic topological properties of such weak Fano 3--folds.

Class (b) gives rise to hundreds of thousands of distinct deformation
classes of semi-Fano 3--folds (discussed in a forthcoming paper by Coates,
Haskins, Kasprzyk and Nordstr\"om~\cite{toric:g2}).
Toric semi-Fano 3--folds can be understood completely in terms of the geometry of so-called reflexive polytopes 
of dimension three; such reflexive polytopes were completely classified by Kreuzer--Skarke~\cite{Skarke} and there are 
over four thousand such reflexive polytopes. 
Moreover, the topology of toric semi-Fano 3--folds is relatively simple and easily computed in terms of the reflexive polytope. 
This makes toric semi-Fano 3--folds a very convenient class for producing large numbers of ACyl Calabi--Yau 3--folds 
and computing their topology. 

Class~(c) all consist of so-called weak del Pezzo\footnote{some authors
use almost del Pezzo} 3--folds, that is, weak Fano 3--folds 
for which $-K_{Y}\in H^{2}(Y;\Z)$ is divisible by $2$.
There are very few smooth del Pezzo \mbox{3--folds}, 
of which smooth cubics in $\CP^{4}$ form one deformation family. 
Degenerating a smooth cubic 3--fold to a cubic 3--fold with only ordinary nodes and seeking projective small resolutions 
of these singular del Pezzo 3--folds yields a method to produce numerous weak del Pezzo 3--folds -- all of 
the same anticanonical degree but with increasing Picard rank -- from a single deformation family of smooth del Pezzo 3--folds. 
This particular family of examples -- studied in detail by Finkelnberg~\cite{finkelnberg1987small},  
Finkelnberg--Werner~\cite{finkelnberg1989small} and 
Werner~\cite{werner:thesis} -- illustrates a general principle that a single deformation family of smooth Fano 3--folds can spawn 
many different deformation families of smooth weak Fano 3--folds; this helps to explain why weak Fano 3--folds 
can be expected to be so numerous.

\subsection*{Acknowledgements}

The authors would like to thank  Kevin Buzzard, Paolo Cascini, Tom Coates, Igor  Dolgachev, Simon Donaldson, 
 Bert van Geemen,  Anne-Sophie Kaloghiros, Al Kasprzyk and Vyacheslav Nikulin.
Computations related to toric semi-Fanos were performed in collaboration with Tom Coates and Al Kasprzyk
and were carried out on the Imperial College mathematics cluster and the Imperial College High Performance
Computing Service;  we thank Simon Burbidge, Matt Harvey, and Andy Thomas for technical assistance. 
Part of these computations were performed on hardware supported by AC's EPSRC grant EP/I008128/1.
MH would like to thank the EPSRC for their continuing support of his research under Leadership Fellowship 
EP/G007241/1, which also provided postdoctoral support for JN.
JN also thanks the ERC for postdoctoral support provided by Grant 247331. 
TP gratefully acknowledges the financial support provided by a Marie Curie European Reintegration Grant.
The authors would like to thank the referee for their careful reading of the paper and their detailed comments.

\section{Asymptotically cylindrical Calabi--Yau 3--folds}
\label{sec:ac_CY}

By a \emph{Calabi--Yau manifold} we mean a Kähler manifold
$(M^{2n}, I, g, \omega)$ with a parallel complex $n$--form $\Omega$.
Then the Riemannian holonomy of $(M,g)$ is contained in $\sunitaryn$.
(At this stage we do not insist that $\Hol(g) = \sunitaryn$; however, 
this will be the case for all the noncompact Calabi--Yau 3--folds constructed
later in this paper.)
We further impose a normalisation condition that
\begin{equation}
\label{eq:cynormal}
\frac{\omega^n}{n!} = i^{n^2}2^{-n}\Omega \wedge \overline \Omega
\end{equation}
(equivalently $\Omega$ has constant norm $2^n$). The complex structure and
metric can be recovered from the pair $(\omega, \Omega)$, and we refer to this
as a \emph{Calabi--Yau structure}.
$\Omega$~is holomorphic, so the canonical bundle of $(M,I)$ is trivial.
The well-known relation between the curvature of the canonical bundle and the
Ricci curvature of a Kähler metric implies that $\omega$ is Ricci-flat.

This relation implies also that if $(M,I,\omega)$ is a Ricci-flat Kähler
manifold, then the restricted holonomy (that is, the group generated by parallel
transport around contractible closed curves in~$M$, or equivalently the
identity component of $\Hol(M)$) is contained in \sunitaryn, but if $M$ is not
simply connected then there need not be any global holomorphic section of $K_M$.
In other words, the canonical bundle need not be trivial, though the real
first Chern class $c_1(M) \in H^2(M;\bbr)$ must vanish.
Conversely, Yau's proof of the Calabi conjecture~\cite{yau:CY} shows that
any compact Kähler manifold $M$ with $c_1(M) = 0 \in H^2(M; \bbr)$ admits
Ricci-flat Kähler metrics.
More precisely, every K\"ahler class on $M$ contains a unique K\"ahler Ricci-flat metric.

We now turn our attention to a special type of non-compact complete manifold
called asymptotically cylindrical.

\begin{definition}
We say that $V^{2n}_\cyl$ is a \emph{Calabi--Yau (half)cylinder} if
$V_\cyl \cong \bbrp \times X^{2n-1}$ is equipped with an $\bbrp$--translation
invariant Calabi--Yau structure $(I_\cyl, g_\cyl, \omega_\cyl, \Omega_\cyl)$,
such that $g_\cyl$ is a product metric $dt^2 + g_X^2$ and $X$ is a smooth
closed manifold called the \emph{cross-section} of $V_\cyl$.
\end{definition}

The only Calabi--Yau cylinders that will play any significant role in this
paper have cross-section $X = \Sph^1 \times S$ for a Calabi--Yau
$(n{-}1)$--fold $(S^{2n-2},I_S,g_S,\omega_S,\Omega_S)$, and 
$V_\cyl := \R^+ \times \Sph^1 \times S$ (biholomorphic to $\Delta^*\times S$ where $\Delta^* \subset \C$ denotes 
the unit disc in $\C$ with the origin removed)
has product structure
\begin{equation}
\label{eq:cycyl}
\begin{aligned}
I_\cyl &:= I_{\C} + I_S, &
g_\cyl &:= dt^2 + d\anglen^2 + g_S, \\ 
\omega_\cyl &:= dt \wedge d\anglen + \omega_S,\quad &
\Omega_\cyl &:= (d\anglen - i dt) \wedge \Omega_S ,
\end{aligned}
\end{equation}
where $t$ and $\anglen$ denote the standard variables on $\R^+$ and $\Sph^1$.
(The choice of phase for the $d\anglen - i dt$ factor makes no material
difference, but helps some equations in~\cite{chnp2} take a more pleasant form.)

\begin{definition} \label{def:ACCY}
Let $(V,g,I,\omega,\Omega)$ be a complete Calabi--Yau manifold. We say that $V$
is an \textit{asymptotically cylindrical} (or \emph{ACyl} for short)
Calabi--Yau manifold if there exist (i) a compact set $K\subset V$, (ii) a
Calabi--Yau cylinder $V_\cyl$ and (iii) a diffeomorphism
$\eta: V_\cyl \rightarrow V {\setminus} K$ such that for
all $k\geq 0$, for some $\lambda>0$ and as $t\rightarrow \cyl$,
$$\eta^*\omega-\omega_\cyl=d\varrho,\ \mbox{ for some $\varrho$ such that }
|\nabla^k\varrho|=O(e^{-\lambda t})$$
$$\eta^*\Omega-\Omega_\cyl=d\varsigma,\mbox{ for some $\varsigma$ such that }
|\nabla^k\varsigma|=O(e^{-\lambda t})$$
where $\nabla$ and $|\cdot|$ are defined using the metric $g_\cyl$ on
$V_\cyl$. We will refer to $V_\cyl$ as the \textit{asymptotic end} of~$V$.
\end{definition}

\begin{remark}
Our definition demands that $\eta^*\omega$ be cohomologous to $\omega_\cyl$ on
the end of $V$. However, as long as
$|\eta^*\omega-\omega_\cyl| \rightarrow 0$, this is automatic. The main point
of the definition is thus to impose the existence of specific $\varrho$ and
$\varsigma$ with the stated rate of decay.
\end{remark}

Since the complex structures on both $\R^+\times\Sph^1\times S$ and $V$ are
determined by the corresponding complex volume forms, similar estimates
also hold for $|\nabla^k(\eta^*I-I_\cyl)|$. The same is true for the
metrics.

For the examples of this paper, we will be concerned with the case of complex
dimension $n = 3$. Let us remark briefly on the relation between the holonomy
of an ACyl Calabi--Yau manifold $V$ and its topology in this case.
\begin{itemize}
\item
$\Hol(V)$ is exactly $\sunitary3$ if and only if $\pi_1(V)$ is finite.
\item
If $\Hol(V) = \sunitary3$ and the asymptotic end is a product
$\R^+\times\Sph^1\times S$, then $S$ is a projective K3 surface.
\item If the asymptotic end is a product $\R^+\times\Sph^1\times S$ and
$S$ is a K3 surface, then $\Hol(V) = \sunitary3$ unless $V$ is a quotient
of $\R \times\Sph^1\times S$ by an involution; up to deformation there is
a unique $V$ of the latter kind.
\end{itemize}
For the proofs of these claims, and more general considerations of holonomy
of ACyl Calabi--Yau manifolds, see Haskins--Hein--Nordstr\"om~\cite[Section~2]{hhn}.

We now want to review a method for constructing ACyl Calabi--Yau manifolds.
It is based on the following ACyl version of the Calabi--Yau theorem.
Note that if $S$ is a smooth anticanonical divisor in a closed K\"ahler
manifold $Z$, then the canonical bundle $K_S$ is trivial, so each Kähler class
on $S$ contains a Ricci-flat metric by Yau's proof of the Calabi conjecture.
\begin{theorem}
\label{thm:acyl_calabi}
Let $Z$ be a closed Kähler manifold with a morphism $f \co Z \to \bbp^1$,
with a smooth connected reduced fibre $S \in \acls{Z}$, and
let $V = Z \setminus S$. If $\Omega_S$ is a non-vanishing holomorphic
$(n{-}1)$--form on $S$, $\omega_S$ a Ricci-flat Kähler metric on $S$ 
satisfying the normalisation condition \eqref{eq:cynormal}, and
$[\omega_S] \in H^{1,1}(S)$ is the restriction of a Kähler class on~$Z$,
then there is an ACyl Calabi--Yau structure $(\omega,\Omega)$ on $V$ whose
asymptotic limit has the complex product form \eqref{eq:cycyl}.
\end{theorem}

Closely related statements were made first by Tian--Yau
in~\cite[Theorem~5.2]{tian:yau} and later by Kovalev
in~\cite[Theorem~2.4]{kovalev:connectsums}.  Tian--Yau establish the
existence of a
Calabi--Yau structure on $V$ by solving a complex Monge--Amp\`ere
equation, but not that this structure is asymptotically cylindrical in
the sense defined in \fullref{def:ACCY}, that is, they do not prove that the
metric they construct decays exponentially to the complex product form
\eqref{eq:cycyl}.  The exponential decay is crucial for the gluing
argument used to construct compact \gtwo--manifolds from a pair of
ACyl Calabi--Yau 3--folds via the twisted connected sum construction.
Kovalev used Tian--Yau's work as a starting point and then attempted
to prove the exponential decay as a separate step.  Unfortunately,
Kovalev's exponential decay argument~\cite[page~132]{kovalev:connectsums} relies on an estimate
established by Tian--Yau in their work on complete K\"ahler-Ricci-flat
metrics with maximal volume growth~\cite[page~52]{tian:yau2} -- but the
estimate from~\cite{tian:yau2} crucially relies on a Euclidean type
Sobolev inequality that definitely fails for any volume growth rate
less than the maximal one.  Thus until very recently no complete proof
of the existence of ACyl Calabi--Yau manifolds existed in the
literature.  Haskins--Hein--Nordstr\"om~\cite{hhn} recently filled
this gap by giving a short, direct self-contained proof of an ACyl
version of the Calabi conjecture~\cite[Theorem~4.1]{hhn}; this
proof avoids appealing to the more general (but technically more
formidable and less precise) existence theory of Tian--Yau~\cite{tian:yau}.
Proving the existence of ACyl Calabi--Yau metrics is relatively straightforward given the ACyl Calabi conjecture: 
see~\cite[Theorem~D]{hhn}. Below we explain how to deduce
\fullref{thm:acyl_calabi} from~\cite[Theorem~D]{hhn}.

\begin{proof}
By assumption there is a meromorphic $n$--form $\Omega$ on $Z$ with a simple
pole along~$S$. Its residue is a non-vanishing holomorphic $(n{-}1)$--form
on $S$. Since this is unique up to multiplication by a complex scalar, we can
choose $\Omega$ so that its residue is~$\Omega_S$. 

The restriction of $\Omega$ to $V$ is a holomorphic volume form.
Together with the exponential map $\R^+ \times \Sph^1 \cong \Delta^*$, a smooth
local trivialisation $\Delta \times S \into Z$ for $f$ yields a smooth map
$\eta \co \R^+\times\Sph^1\times S \to V$ that is a diffeomorphism onto the
complement of a compact subset, and $\eta^*\Omega$ has the asymptotic behaviour
required in \fullref{def:ACCY}.

Now~\cite[Theorem~D]{hhn} shows that, in the restriction to $V$ of any Kähler
class on $Z$, there is a unique Ricci-flat Kähler metric $\omega$ such that
\eqref{eq:cynormal} holds (implying that $(\omega,\Omega)$ is a Calabi--Yau
structure), and that $\omega$ is ACyl with respect to $\eta$.
The asymptotic limit of $\omega$ has the form
$\mu dt \wedge d\anglen + \omega_S$, where $\omega_S$ is necessarily the
unique Ricci-flat Kähler metric in the restriction of the Kähler class from $Z$
to $S$. Because $(\omega_S, \Omega_S)$ satisfies the normalisation condition
for Calabi--Yau structures we must have $\mu = 1$. Thus the asymptotic limit of
$(\omega,\Omega)$ is precisely the product \eqref{eq:cycyl}.
\end{proof}

\begin{remark*}
Haskins--Hein--Nordstr\"om~\cite{hhn} also shows that the above construction is reversible in the
following sense: If, as assumed in \eqref{eq:cycyl}, $V$ is an ACyl Calabi--Yau manifold whose cross-section $X$ 
splits as a Riemannian product $\Sph^{1} \times S$ 
for some smooth compact Calabi--Yau \mbox{$(n{-}1)$--fold} $S$, then 
if $V$ is simply-connected one can prove that there is a
smooth closed Kähler (in fact projective) manifold $Z$ with an anticanonical fibration over $\bbp^1$
such that applying \fullref{thm:acyl_calabi} recovers $V$.
It is not always the case that the asymptotic end of an ACyl Calabi--Yau
manifold splits in this way: \linebreak see \mbox{\cite[Example~1.5]{hhn}}
for such a manifold.
Provided that a simply-connected ACyl Calabi--Yau $V$ has irreducible holonomy and $\dim_\bbc V > 2$, 
then one can prove that a projective compactification $Z$ still exists 
even when the asymptotic end is not such a Calabi--Yau product; 
in this case $Z$ may have orbifold singularities,
but $V$ can still be recovered from $Z$ by a generalisation of
\fullref{thm:acyl_calabi}: see~\cite[Theorems~B and~C]{hhn}.
\end{remark*}

Kovalev~\cite{kovalev:connectsums} applies \fullref{thm:acyl_calabi} to
certain blow-ups $Z$ of Fano 3--folds. Since there are $105$ deformation classes
of smooth Fano 3--folds this yields a similar number of deformation classes of
ACyl Calabi--Yau 3--folds. 
(Some Fano 3--folds $Y$ can be blown up in several different ways to give
different admissible $Z$, see, for example, Examples~\ref{exa:P3_deg} and~\ref{ex:2conics}. 
This has not studied systematically, so
it is difficult to be more precise with the enumeration here.)
Kovalev--Lee~\cite{kovalev:lee} have also applied \fullref{thm:acyl_calabi} to 3--folds $Z$ of a different kind, obtained from
K3 surfaces with non-symplectic involution. There are 75 deformation classes 
of K3 surfaces with non-symplectic involution to which their result applies; 
this gives another 75 deformation families of ACyl Calabi--Yau 3--folds.
Together these existing constructions yield at most a few hundred ACyl 
Calabi--Yau 3--folds.

In \fullref{sec:weak-fano-3} (for example, see \fullref{prop:onestage} and the paragraph preceding it)
we show that the same procedure used by 
Kovalev in the case of Fano 3--folds can be applied to the much larger class of weak Fano  3--folds: 
see \fullref{dfn:weak_fano}.
Since, as we will explain in detail later, 
there are hundreds of thousands of deformation classes of weak Fano 3--folds this expands the number of 
known ACyl Calabi--Yau 3--folds from a few hundred to at least several hundred thousand.
The topology of these ACyl manifolds is discussed
in \fullref{sec:blocks}. In particular we find that they are simply connected,
so their holonomy is exactly $\sunitary3$.

\section{Algebro--geometric preliminaries}
\label{sec:alg-geom}
We review briefly some definitions and results from algebraic geometry
needed for our later discussion of weak Fano 3--folds; although these
notions are well known to algebraic geometers they seem to be
unfamiliar to many differential geometers interested in manifolds with
special or exceptional holonomy.  The reader should feel free to
proceed to the section on weak Fano 3--folds, returning to this section
as needed.

\begin{convention}
We always assume our varieties to be complex projective varieties and morphisms to be
projective unless specifically stated otherwise.
\end{convention}

\subsection*{Line bundles, weak positivity and vanishing theorems}
We will need generalisations of the Kodaira--Akizuki--Nakano vanishing theorem 
and the Lefschetz theorem for sections of ample line bundles; the generalisations we need 
replace the ampleness/positivity of the line bundle with some condition of sufficient 
semi-positivity of the line bundle. Depending on what semi-positivity assumption we make on $L$ 
we recover more or less of the cohomology vanishing results implied by the Kodaira--Akizuki--Nakano vanishing theorem.
We refer the reader to Lazarsfeld's book~\cite{lazarsfeld:positivity} for a comprehensive treatment of positivity 
for line bundles. 

\begin{definition}
\label{def:line:positivity}
 Let $L$ be a line bundle on a projective algebraic variety $Y$; we say that:
 \begin{enumerate}
\item $L$ is \emph{very ample} if the sections in 
   $H^0(Y,L)$ define an embedding into projective space;
 \item $L$ is \emph{ample} if for some integer $m>0$ $L^{\otimes m}$
   is very ample;
   \item $L$ is \emph{semi-ample}, or \emph{eventually free}, if for
  some integer $m>0$ the sections in $H^0(Y,L^{\otimes m})$ define a
  morphism to projective space; equivalently, the linear system
  $|L^{\otimes m}|$ is base point free; 
 \item $L$ is \emph{nef} if for every compact algebraic curve $C\subset
   Y$,  $\deg L_{|C}=c_1(L)\cap C\geq 0$;
 \item $L$ is \emph{big} if for some integer $m>0$ the sections in 
   $H^0(Y, L^{\otimes m})$ define a rational map to projective space which is
   birational on its image. 
 \end{enumerate}
\end{definition}
By replacing ample in the definition of a Fano manifold with the weaker condition big and nef we will obtain 
the definition of a weak Fano manifold: see \fullref{dfn:weak_fano}.

See also \fullref{D:l-ample} for the notion of an \emph{$l$--ample} line bundle; 
this is intermediate between semi-ample and ample.
\begin{remark*}
  It is well known that, if $L$ is nef, then $L$ is big if and only if 
\[
L^{\dim Y}:= \int_Y c_1(L)^{\dim Y} >0.
\]
\end{remark*}

Suppose that $L$ is a semi-ample line bundle on a normal projective
variety $Y$.  We denote by $M(Y,L)$ the sub-semigroup
$M(Y,L) = \{m \in \N \,|\, L^{\otimes m} \ \text{is base point free} \}$.
We write $e$ for the ``exponent'' of $M(Y, L)$, that is, the largest
natural number dividing every element of $M(Y, L)$;  
in particular $L^{\otimes ke}$ is free for $k\gg 0$.   
Given $m \in M(Y, L)$, write $X_m = \varphi_{m}(Y)$ for the image of the morphism
$\varphi_m = \varphi_{L^{\otimes m}}\co Y \to \PP H^0(Y,L^{\otimes m})^\vee$. 

The following is a well-known result of Zariski: see Lazarsfeld~\cite[Theorem~2.1.27]{lazarsfeld:positivity}.

\begin{theorem}[Semi-ample fibrations] 
\label{t:semiample}
Let $L$ be a semi-ample bundle on a normal projective variety $Y$.
Then there is an algebraic fibre space $\varphi \co Y \longrightarrow
X$ having the property that for any sufficiently large integer $k\in M(Y,L)$:
\[
X_k = X \quad \text{and} \quad \varphi_k = \varphi.
\]
Furthermore there is an ample line bundle $A$ on $X$ such that
$\varphi^*A = L^{\otimes e}$, where $e$ is the exponent of $M(Y,L)$.
\end{theorem}

In other words, for $m \gg 0$ the mappings $\varphi_m$ stabilise to
define a fibre space structure on $Y$ (essentially characterised by
the fact that $L^{\otimes e}$ is trivial on the fibres).
\begin{remark}
\label{r:semiample:fg}
A corollary of the previous theorem is the following fact:
if $L$ is a semiample line bundle then  $L$ is finitely generated, that
is, $R(Y,L) := \bigoplus_{m \ge 0}H^{0}(Y,mL)$ is 
a finitely generated $\C$--algebra: see~\cite[2.1.30]{lazarsfeld:positivity}.
\end{remark}

If $L$ is ample (or positive) we have the famous cohomology vanishing theorem due to 
Kodaira~\cite{kodaira} and extended by Akizuki--Nakano~\cite{akizuki:nakano}. 
If $L$ is sufficiently semi-positive then we 
can also obtain similar cohomology vanishing theorems as we now describe.

We begin with the Kawamata--Viehweg vanishing theorem; this requires the weakest positivity assumption:

\begin{theorem}[Kawamata--Viehweg vanishing] 
\label{T:KV:vanish}
Let $L$ be a nef and big line
  bundle on a non-singular projective variety $Y$.
Then $H^i(Y, K_Y\otimes L)=(0)$ for $i>0$.
  Equivalently, by Serre duality, $H^i(Y, L^\vee)=(0)$ for $0\leq
  i<\dim Y$. 
\end{theorem}

\begin{remark*}
  We have stated a simplified form of the vanishing theorem of Kawamata
  and Viehweg that suffices for our purpose. The general statement -- for
example, see Koll\'ar--Mori~\cite[Theorem~2.64]{KM} -- and the proof of even the simplified form, require the use of
  fractional divisors.  
\end{remark*}

In general the Akizuki--Nakano generalisation of Kodaira vanishing fails 
for big and nef line bundles: see Lazarsfeld~\cite[Example~4.3.4]{lazarsfeld:positivity} 
for a big and nef line bundle $L$ on $Y$, the one point blowup of $\CP^{3}$, 
for which $H^{1}(Y, \, \Omega^{1}_Y \otimes L^{\vee}) \neq 0$.
However, we do have the following generalisation of the Akizuki--Nakano
vanishing theorem, due to Sommese and improved by Esnault--Viehweg~\cite[6.6]{ev:vanishing}.
\begin{definition}
\label{D:l-ample}
  A semi-ample line bundle $L$ on a non-singular projective variety $Y$ is
  \emph{$l$--ample} for some integer $l\ge 0$ if the maximum dimension of any
  fibre of the semi-ample fibration $\varphi\co Y \to X$ 
  is $\leq l$.
\end{definition}
An ample line bundle is $0$--ample. 
For $l$--ample line bundles we get Akizuki--Nakano-type vanishing results 
but for a restricted range of cohomology groups that depends on~$l$. 
\begin{theorem}[Sommese--Esnault--Viehweg vanishing]
\label{thm:aknava}
  Let $L$ be an $l$--ample line bundle on a non-singular projective
  variety $Y$ with semi-ample fibration $\varphi\co Y \to X$. Then
\[
H^p(Y, \, \Omega_Y^q \otimes L^{\vee}) = 0, \quad \text{for} \ p+q < 
\min{\{\dim X, \, \dim{Y}-l+1\}}.
\]
In particular, if $L$ is also big  then $\dim Y=\dim X$ and so if $l \ge 1$
vanishing holds when $p+q <\dim{Y}-l+1$. 
\end{theorem}

For ample line bundles $L$ we  have the Lefschetz hyperplane theorem
that relates the topology of sections of $L$ to the topology of $Y$.
For general big and nef line bundles the Lefschetz hyperplane theorem is false. 
However, there is a good generalisation of the Lefschetz hyperplane
theorem to the case of a line bundle that defines a semi-small
morphism, due -- in its strongest and most general form -- to Goresky and MacPherson.

We begin with the following, which we take from Goresky--MacPherson~\cite[page~151]{goresky:macpherson}.

\begin{definition}
\label{D:semismall:small}
Let  $f\co Y\to X$ be a projective morphism of projective varieties (not necessarily of the same dimension) 
and for any non-negative integer $k$ write 
\[X^{k} = \{ x\in X \mid \dim f^{-1}(x)=k\}.
\]
We say that $f$ is \emph{semi-small} if 
\[\dim{Y} - \dim{X^{k}}  \ge 2k  \quad \text{for every }k \ge 0.\]
Equivalently, $f$ is semi-small if and only if there is
    no irreducible subvariety $E\subset Y$ such that $2\dim E -\dim
    f (E)>\dim Y$.
\end{definition}

\begin{remark}
\label{r:lef:1ample}
If $L$ is a semi-ample line bundle on a non-singular projective 3--fold $Y$ and the semi-ample fibration 
$\varphi \co Y \to X$ is birational then $L$ is semi-small if and only
if $L$ is $1$--ample. 
\end{remark}

\noindent

The following Lefschetz theorem for semi-small morphisms is a
more-or-less immediate consequence of Goresky--MacPherson's ``relative
Lefschetz hyperplane theorem with large fibers''~\cite[Theorem~1.1, page~150]{goresky:macpherson}.

\begin{prop}
\label{prop:lef}
Let $Y$ be a non-singular projective variety of complex dimension
$n=\dim_\CC Y$, $f\co Y \to \PP^N$ a semi-small morphism, and $S\in
|f^\star \mathcal{O}_{\PP^N}(1)|$ a non-singular member. Then, the restriction map
\[
H^m(Y;\ZZ) \to H^m(S;\ZZ)
\]
 is an isomorphism for $m<n-1$ and is primitive injective for $m=n-1$.  
\end{prop}

\begin{proof}
  All statements follow from the fact that $H_m(Y,S;\ZZ)=(0)$ for
  $m\leq n-1$. This fact is an immediate consequence of~\cite[Theorem~1.1,
  page~150]{goresky:macpherson}. Here we are applying
  the statement with their $(X,H)$ being our $(Y,S)$. The assumptions
  are satisfied because the morphism $Y\to X$ is semi-small, see
  \emph{loc.\ cit.}\ Remark~(2), page 151. Note that by \emph{loc.\ cit.}\ 
  Remark~(1), page 151, we are allowed to replace $H_\delta$ with~$H$.
  In summary the conclusion is that the usual statement of the
  Lefschetz theorem holds in our case for the pair $(Y,S)$.

We discuss in some further detail the statement of primitivity of the inclusion.
Consider the long exact sequence of cohomology of the pair $(Y,S)$:
\[
\cdots \to H^{n-1}(Y;\ZZ) \stackrel{\rho}{\rightarrow} H^{n-1}(S;\ZZ)
\stackrel{\delta}{\rightarrow} H^{n}(Y,S;\ZZ)\to H^{n}(Y;\ZZ) \to \cdots \,.
\]
Notice that $\Imag(\rho)$ is primitive iff $\mbox{Coker}(\rho)$ is
torsion-free, which is equivalent to $\Imag(\delta)$ being
torsion-free.  It is thus enough to show that $H^n(Y,S)$ is
torsion-free. By the universal coefficient theorem, the torsion of
this group is isomorphic to the torsion of $H_{n-1}(Y,S)$, which is
trivial by what we said. 
\end{proof}

\begin{remark*}
We will use \fullref{prop:lef} in the proof of \fullref{prop:block_from_weak}(iii) 
(see also \fullref{lem:lef}) 
to show that anticanonical sections of a semi-Fano 3--fold $Y$ are
$\Pic{Y}$--polarised K3 surfaces.
\end{remark*}

\subsection*{Weak Fano 3--folds via resolutions of singularities}
We will see shortly that every smooth weak Fano 3--fold $Y$ -- one of the main objects 
of interest in this paper -- can be obtained as a special type of resolution 
of a mildly singular Fano 3--fold: see \fullref{R:weak:fano:ac:model}.
For this reason even though 
we are interested in constructing smooth weak Fano 3--folds we will 
need to deal with certain mildly singular 3--folds. 
This forces us to address several issues that arise only on singular varieties, 
for example, the fact that on a singular complex variety not every Weil divisor need be Cartier 
plays an important role in this paper. 

Moreover, while resolutions of singularities exist very generally, the 
special sort of resolutions required to produce smooth weak Fanos from 
singular Fanos impose severe restrictions on the type of singularities we should consider. 
This leads us to consider in detail Gorenstein canonical and terminal singularities 
and special types of resolution of such singularities: so-called crepant and small resolutions. 
The existence of crepant and small resolutions is a delicate issue in general, 
as we will try to explain, but it is central to the construction 
of smooth weak Fano 3--folds from singular Fano 3--folds.

\subsection*{Divisors on singular varieties}
We begin with some generalities about divisors on singular varieties; 
this issue comes up because we are forced to work with singular varieties.

We denote by $\Cl X$, the \emph{class group} of Weil divisors on $X$
modulo linear equivalence, and by $\Pic X$ the Picard group of Cartier
divisors on $X$ modulo linear equivalence.  A variety is
\emph{factorial} if every Weil divisor is Cartier or
\emph{$\Q$--factorial} if some integer multiple of every Weil divisor
is Cartier.  Being $\Q$--factorial is a local property in the Zariski
topology of $X$, not the analytic topology. On any normal complex
variety we can define the \emph{canonical divisor} $K_{X}$ (by
extension from the regular part using the normality assumption) as a
Weil divisor (unique up to linear equivalence).  In general $K_{X}$ is
not a Cartier divisor; we say that $X$ is \emph{Gorenstein}
  (respectively $\Q$--Gorenstein) if $K_{X}$ is
Cartier (respectively there exists some $j\in \N$ so that $jK_{X}$ is
Cartier) and $X$ is Cohen--Macaulay.

\begin{convention}
In the rest of the paper we assume all our varieties to be normal  and Gorenstein, 
but many of the varieties we encounter will be neither factorial nor $\Q$--factorial. 
\end{convention}

\subsection*{Small projective birational morphisms and resolution of singularities}
Let $X$ and $Y$ be normal complex algebraic varieties both of
dimension $n$.  Given a projective birational morphism $f\co Y \ra X$, define
the $f$--exceptional set $E:=\ex(f)$ to be the closed subset where $f$
is not a local isomorphism. $f$ is surjective, $E=f^{-1}(f(E))$ and
$\codim_{X}f(E)\ge 2$.

\begin{definition}
\label{D:small}
  We call a projective birational morphism $f\co Y \ra X$ $\emph{small}$ if
  the exceptional set $E=\ex(f)$ is of (complex) codimension at least
  $2$.
\end{definition}
Small projective birational morphisms and particularly projective small resolutions 
(that is, when $Y$ is non-singular: see \fullref{d:crep:small:res}) 
play important roles in this paper. \fullref{E:ODP} gives the simplest -- and for this paper the most important -- example 
of a small resolution.
\begin{remark}
\label{R:q:fact:not:small}
If $X$ is $\Q$--factorial every irreducible component of $E$ has
codimension $1$ \cite[Section~1.40]{debarre:book}; in particular, if
$X$ is $\Q$--factorial then a projective birational morphism $f\co Y \ra X$ is
never small.
In other words, if there exists \emph{any} projective birational morphism
$f \co Y \to X$ which is small 
(but not an isomorphism) then $X$ cannot be $\Q$--factorial; 
this forces us to deal with singular varieties that are not $\Q$--factorial.
\end{remark}

By \fullref{R:q:fact:not:small} if we are interested in small projective
birational morphisms $f  \co Y \to X$ 
then $X$ is forced to be non $\Q$--factorial. We now want to explain in
detail the intimate link between non--$\Q$--Cartier divisors 
$D' \in \Cl(X)$ and small projective birational morphisms to $X$. 

The following elementary lemma makes this connection precise:
see, for example, Kawamata~\cite[Lemma~3.1]{kawamata1988crepant} or Koll\'ar's survey~\cite[Proposition~6.1.2]{kollar:flips:jdg}.

\begin{lemma}
\label{l:small:morphism}
Let $f \co Y \to X$ be a small projective birational morphism which is not an isomorphism 
and let $D$ be an $f$--ample (see \fullref{r:f:ample}) Cartier divisor on $Y$. 
Then the following hold:
\begin{enumerate}
\item $mf_{*}D$ is not Cartier if $m>0$;
\item $f_{*}\mathcal{O}_{Y}(mD) = \mathcal{O}_{X}(mf_{*}D)$ for $m\ge 0$, and 
\item
$R(X,f_{*}D):= \bigoplus_{m\ge 0}{\oo}_{X}(mf_{*}D)$ is a finitely generated
$\mathcal{O}_{X}$--algebra 
and $Y$ is recovered from $R(X,f_{*}D)$ by taking $\Proj R(X,f_{*}D)$.
\end{enumerate}
Conversely, let $D'$ be a Weil divisor on $X$ which is not $\Q$--Cartier and for which 
\[R(X,D'):=\bigoplus_{m\ge 0}{\mathcal{O}_{X}(mD')}\]
is a finitely generated $\mathcal{O}_{X}$--algebra. 
Then $Y := \Proj R(X,D')$ is a normal projective variety, 
the projection map $f \co Y \to X$ is a small projective birational
morphism and $f_{*}^{-1}D'$ is $f$--ample (and $\Q$--Cartier).
\end{lemma}

\begin{remark}
\label{r:f:ample}
Given a projective morphism $f\co Y \ra X$ of varieties there is a general
notion of $f$--ampleness or ampleness  of a divisor relative to the
morphism $f$: see, for example, Lazarsfeld~\cite[Section~1.7]{lazarsfeld:positivity}.
Rather than give the general definition we recall the relative version of
the Nakai criterion: a divisor $D$ is \emph{$f$--ample} if and only if
$D^{\dim{V}}\cdot V >0 $ for every irreducible subvariety $V \subset Y$
of positive dimension which maps to a point under~$f$.
\end{remark}

\begin{remark*}
  Suppose that $-D'=Z\geq 0$ is effective; then
  $\oo_X(D^\prime)=I_Z\subset \oo_X$ is the ideal sheaf of $Z\subset
  X$, and then the $m$\textsuperscript{th} symbolic power of $I_Z$ is
  $I_Z^{(m)} \cong \oo(mD')$.  (For this reason the algebra $R(X,D')$ is
  called the symbolic power algebra of $D'$.)  If, in addition, we
  assume that the sheaf of algebras $\bigoplus_{m \ge 0}\oo_{X}(mD')$ is
  generated by $\oo_{X}(D') = I_{Z} \subset \oo_{X}$, that is,
  generated in degree $1$, then $Y=Bl_{I_{Z}}X$ is the
  blow up of the ideal of $Z$.  In other words, in this case we can
  get $Y$ by blowing up the non--$\Q$--Cartier divisor $Z\subset X$. In
  fact, this will be the case in all the examples considered in this
  paper: see \fullref{sec:examples}.

  To motivate the construction, recall first the universal property of
  blowups.  The blowup of $X$ in a subvariety (or, more generally,
  closed subscheme) $S\subset X$ with ideal $I_S\subset \oo_X$ is a
  morphism $\pi \co X' \to X$ such that the ideal
  $\pi^{-1}(I_S)\cdot \oo_{X'}\subset \oo_{X'}$ is a Cartier divisor
  on $X'$ and such that for any morphism $f \co Y \to X$ with
  $f^{-1}(I_S)\cdot \oo_Y$ a Cartier divisor on $Y$ there exists a
  unique morphism $\rho \co Y \to X'$ such that $f = \pi \circ
  \rho$. Informally: the blowup of $S \subset X$ is the ``smallest'' morphism to $X$ that
  turns $S$ into a Cartier divisor. In particular blowing up a Cartier
  divisor $D \subset X$ can only induce an isomorphism of~$X$.
  However, if $X$ is not $\Q$--factorial then blowing up a Weil divisor
  $Z$ in $X$ that is not $\Q$--Cartier must induce a non-trivial
  birational morphism to $X$ -- because it converts the Weil divisor
  $Z$ into a Cartier divisor in the blowup.  Since $Z$ is of
  codimension 1 in $X$ we expect that this birational morphism will
  not alter $X$ too much; in fact, when the symbolic algebra of $I_Z$
  is generated by $I_Z$, \fullref{l:small:morphism} states that the
  induced birational morphism is \emph{small} in the sense of
  \fullref{D:small}.
\end{remark*}

\begin{remark*}
To obtain a small projective birational morphism from a non--$\Q$--Cartier divisor $D' \in \Cl(X)$ as above we need to 
know that the symbolic power algebra $R(X,D')$ is a finitely generated
$\oo_{X}$--algebra. This is not true in general. 
However, for 3--folds with mild singularities Kawamata has shown that this is always true; 
we discuss this in more detail below in \fullref{t:kawamata:fg}, after we have introduced 
appropriate classes of mildly singular 3--folds.
\end{remark*}

\begin{remark}
In general, blowing up different non--$\Q$--Cartier divisors $D' \in \Cl(X)$ as in \fullref{l:small:morphism}
may give rise to the same small projective birational morphism $f \co Y \to X$. 
\end{remark}

In this paper given some mildly singular variety $X$ we will be particularly interested in constructing (special kinds of) 
projective birational morphisms $f \co Y \to X$ where $Y$ is \emph{non-singular}.

\begin{definition}
  A \emph{resolution of singularities} or \emph{desingularisation} of
  $X$ is a projective birational morphism $f\co Y \ra X$ where $Y$ is non-singular.
\end{definition}
The so-called \emph{ramification formula} for a resolution of
singularities compares $K_{Y}$ to the pullback $f^{*}K_{X}$ and states
that there exist unique integers $a_{i}\in \Z$ (recall that we always
assume $K_X$ Cartier) so that
\begin{equation}
\label{E:ramify}
K_{Y}-f^{*}K_{X}= \sum_{i}a_{i}E_{i},
\end{equation}
where $E_{i}$ are the \emph{exceptional divisors} of $f$, that is, the
irreducible components of the exceptional set $E$ of codimension
$1$. The divisor $\sum_{i}{a_{i}E_{i}}$ is sometimes called the
\emph{discrepancy} of $f$.
\noindent
We say:
\begin{definition}
\label{d:crep:small:res}
A resolution $f\co Y \ra X$ is \emph{crepant} if $K_{Y}=f^{*}K_{X}$, that
is, if all
the coefficients $a_{i}\in \Z$ in \eqref{E:ramify} vanish. 

A \emph{small resolution} $f\co Y \to X$ is a resolution of singularities in which 
the projective birational morphism $f$ is small in the sense of \fullref{D:small}.
\end{definition}
\begin{remark*}\hfill
\begin{enumerate}
\item
  If $f\co Y \ra X$ is a small resolution then the exceptional set
  $E$ contains no divisors and hence $K_{Y}=f^{*}K_{X}$; so any
  small resolution is crepant. 
\item
In general a crepant resolution  need not be small, however, see \fullref{R:terminal:crepant}(ii).
\item
A resolution of singularities always exists (at least in characteristic $0$) by Hironaka, 
whereas crepant or small resolutions exist only in very special circumstances. 
We will see below that the existence of a crepant or small resolution of $X$ imposes 
strong constraints on the singularities $X$ may have (for example, see \fullref{R:terminal:crepant}). 
Moreover, even when $X$ satisfies these constraints 
determining whether a given (mildly) singular variety $X$ admits 
a projective crepant or small resolution can be very delicate.
\end{enumerate}
\end{remark*}

\subsection*{Terminal and canonical 3--fold Gorenstein singularities}
One of the standard ways to define various classes of singularities is by assumptions on the coefficients $a_{i}$ 
that appear in the ramification formula above. In this spirit we say:  
\begin{definition}
\label{D:terminal}
A normal Gorenstein variety $X$ has \emph{terminal} (respectively
\emph{canonical}) singularities if for a given resolution of
singularities $f\co Y \ra X$ all the coefficients $a_{i} \in \Z$ in
\eqref{E:ramify} are positive (respectively non-negative).
\end{definition}
One can show that this definition does not depend on the resolution $f$ we chose.
Numerous other equivalent definitions of terminal and canonical
singularities can be found in Reid~\cite{YPG}, which we recommend for an
introduction to 3--fold terminal or canonical singularities.
\begin{remark}\hfill{}\nolinebreak
  \label{R:terminal:crepant}
\begin{enumerate}
\item If $X$ admits a crepant resolution $f$ then \eqref{E:ramify}
  holds with all coefficients $a_{i}=0$ and hence $X$ has canonical
  singularities.  If $X$ admits a small resolution $f$ then since the
  exceptional set contains no divisors, $f$ vacuously satisfies the
  condition in \fullref{D:terminal} and so $X$ has terminal
  singularities.
\item If $X$ has terminal singularities then it follows immediately
  from \eqref{E:ramify} and the definitions that any crepant
  resolution must be small. In other words, if $X$ has terminal 
  singularities then a resolution of $X$ is crepant if and only if it is small.
\item If $X$ is $\Q$--factorial with terminal singularities then $X$
  admits no crepant (necessarily small by the previous remark)
  resolutions by \fullref{R:q:fact:not:small}.
\item In dimension two terminal points are non-singular.  Any canonical
  singularity of a normal surface is (locally analytically) equivalent
  to a Du Val singularity, that is, to a hypersurface singularity in
  $\C^{3}$ of type $A_{n}$, $D_{n}$ ($n\ge 4$), $E_{6}, E_{7}$ or
  $E_{8}$ (see Reid~\cite[Table~0.2]{reid:canonical} for a list of defining
  polynomials); Du Val singularities are the same as rational double
  points.
\item A terminal normal 3--fold (respectively $k$--fold with $k\ge 3$)
  has only isolated singularities (respectively at worst codimension
  $3$ singularities); canonical 3--folds have in general
  $1$--dimensional singular loci.
\item One can prove that the notions of terminal and canonical
  singularities are (algebraically or analytically) local (see
Matsuki~\cite[4.1.2(iii)]{matsuki}); see also \fullref{P:cdv:terminal} for a concrete local characterisation of
  Gorenstein terminal 3--fold singularities.

\end{enumerate}
\end{remark}
The simplest example of a 3--fold Gorenstein terminal singularity is the
ordinary double point or ordinary node; this singularity will play a
crucial role in the paper.
\begin{example}[The 3--fold ordinary double point]\
\label{E:ODP}
Define a hypersurface $X \subset \C^{4}$ by 
\[
X:= \big\{ (z_{1},z_{2},z_{3},z_{4}) \in \C^{4} \mid z_{1} z_{2} = z_{3} z_{4}\,\big\}.
\]
$X$ is the affine cone over the quadric $Q \simeq \CP^{1} \times
\CP^{1} \subset \CP^{3}$.  $X$ is non-singular away from the origin $0$
where it has an isolated singular point, called the \emph{ordinary
  double point} (ODP for short) or \emph{ordinary node}. Blowing up the origin yields
a non-singular variety $\wtilde{X}$ and a resolution of
singularities \mbox{$\pi\co \wtilde{X} \ra X$} whose exceptional
set $E$ is isomorphic to $\CP^{1}\times \CP^{1}$, but $\pi$ is not a
crepant resolution. $\CP^{1}\times \CP^{1}$ has two rulings and we can
contract the fibres of either ruling; this yields two other
resolutions \mbox{$\pi_{\pm}\co \hat{X}^{\pm} \ra X$} whose exceptional
set $E^{\pm}$ is isomorphic to $\CP^{1}$ with normal bundle
\mbox{$\mathcal{O}(-1)\oplus \mathcal{O}(-1)$}.  These two small
resolutions of $X$ can be realised concretely as complete
intersections in $\C^{4}\times \CP^{1}$ as follows
\begin{subequations}
\begin{align}
  \hat{X}^{+} &=  \big\{(z,[x_{1},x_{2}]) \in \C^{4}\times \CP^{1}\mid z_{1}x_{2}= z_{4}x_{1}, z_{3}x_{2} = z_{2}x_{1}\big\},\\
  \hat{X}^{-} &= \big\{(z,[y_{1},y_{2}]) \in \C^{4}\times
  \CP^{1}\mid z_{1}y_{2}= z_{3}y_{1},  z_{4}y_{2} =
  z_{2}y_{1}\big\},
\end{align}
\end{subequations}
where $\pi^{\pm}\co \hat{X}^{\pm} \ra X$  are the restrictions of the obvious projection 
$\C^{4}\times \CP^{1} \ra \C^{4}$.
Since $X$ admits small resolutions the origin is a (Gorenstein) terminal singularity.
\end{example}
\begin{remark*}
Affine cones over other del Pezzo surfaces give rise to canonical (non-terminal) 3--fold singularities.
For example, take a non-singular del Pezzo surface $S$ in $\CP^{8}$ isomorphic to the $1$--point blowup of $\CP^{2}$. 
Then the vertex of the affine cone over $S$ in $\C^{9}$ is an isolated canonical 3--fold singularity. 
This example illustrates that, 
unlike the ordinary double point (and Gorenstein terminal 3--fold singularities more generally see below),  
isolated Gorenstein canonical 3--fold singularities need not be of hypersurface type.
\end{remark*}

\begin{definition}
\label{D:nodal:3fold}
We call a (normal Gorenstein) terminal projective 3--fold $X$ a \emph{nodal 3--fold} 
if each of its singular points $P\in X$ is (locally analytically) equivalent to the ordinary double point of \fullref{E:ODP}.
\end{definition}
The Gorenstein terminal 3--fold singularities were classified by Reid~\cite{reid:canonical}.
To state Reid's classification we recall the notion of a cDV singularity.
\begin{definition}
\label{def:cDV}
 A \mbox{3--fold} singularity $P\in X$ is \emph{cDV} -- \emph{compound
Du Val} -- if a general  local analytic surface section $P\in S \subset X$ has Du Val
singularities. Equivalently, $P\in X$ is cDV if it is (locally analytically) equivalent to a hypersurface singularity given by 
\[
f + tg =0,
\]
where $f \in \C[x,y,z]$ defines a Du Val singularity and $g \in
\C[x,y,z,t]$ is arbitrary. Note that a general cDV singularity need
not be isolated.

A cDV singularity $P\in X$ is said to be of type $cA_{n}$, $cD_{n}$,
$cE_{n}$ according to the Du Val singularity type of a sufficiently
general surface section $S$ through $P$.
\end{definition}
\begin{prop}
\label{P:cdv:terminal}
Every cDV singularity of a Gorenstein 3--fold is canonical and the
Gorenstein terminal 3--fold singularities are precisely the isolated
cDV singularities; in particular Gorenstein terminal 3--fold
singularities are all hypersurface double point singularities.
\end{prop}

\begin{remark}
\label{R:terminal:small:resn}
Let $(X,0)$ (or just $X$) be (the analytic germ of) an isolated Gorenstein
3--fold singularity, and $\pi\co Y \ra X$ a small resolution. Write the
exceptional set as $E= \pi^{-1}(0) = C = \bigcup_{i} C_{i}$ where each
$C_{i}$ is irreducible. It is known that the $C_{i}$ are non-singular
rational curves meeting transversally and the normal bundle of $C_{i}$
in $Y$ is of type $(-1,-1)$ or $(0,-2)$ or $(1,-3)$:
see Laufer~\cite[Theorem~4.1]{laufer:exceptional}.  By the assumption that $X$
admits a small resolution $X$ has terminal singularities and therefore
by \fullref{P:cdv:terminal} has isolated cDV
singularities. The converse question of which isolated cDV
singularities admit small resolutions is addressed by 
Katz--Morrison~\cite{katz:morrison}, Kawamata~\cite{kawamata:small} and Pinkham~\cite{pinkham}.  
For example, the $cA_{2}$ singularity defined by
$\big\{x^{2}+y^{2}+z^{2}+t^{2n+1}=0\big\}$ is
factorial and hence admits no small resolutions, but this is not the
case for the $cA_{2}$ singularity
$\big\{x^{2}+y^{2}+z^{2}+t^{2n}=0\big\}$; see Friedman
\mbox{\cite[Remark~1.7]{friedman:double:pts}}.
\end{remark}

\subsection*{3--fold flops}
Flops will play an important role later in the paper; they allow one to produce new (possibly many) weak Fano 3--folds 
from an existing one. This is a phenomenon which does not occur for smooth Fano 3--folds.
We give only the basic definitions and state some of the foundational results about 3--fold flops; 
we refer the reader to Koll\'ar's survey article~\cite{kollar:flips:jdg} or to the book by Koll\'ar--Mori~\cite{KM} 
for much more comprehensive treatments.

\begin{definition}
\label{D:flopping}
  Let $Y$ be a normal variety. A \emph{flopping contraction} is a
  projective birational morphism $f\co Y \ra X$ to a normal variety $X$ that is small
and such that $K_{Y}$  is numerically $f$--trivial, 
that is, $K_{Y} \cdot C=0$ for any curve $C$ contracted by $f$.

  Let $f\co Y \ra X$ be a flopping contraction, and $D$ be a Cartier
  divisor on $Y$ that is \mbox{$f$--antiample}, that is, $-D$ is $f$--ample
(recall \fullref{r:f:ample}).
  A projective birational small morphism 
  $f^+\co Y^+ \ra X$ is called the \emph{$D$--flop} of $f$ if $D^+$, the proper
  transform of $D$ on $Y^+$, is $f^+$--ample.
\end{definition}

\begin{remark}
  The $D$--flop of $f$ is unique if it exists.
\end{remark}

\begin{theorem}[Koll\'ar--Mori~{{\cite[Theorems 6.14--15]{KM}}}]
\label{T:flops}
  Let $D$ be a Cartier divisor on a projective threefold $Y$ with
  terminal singularities.  Then $D$--flops exist and preserve the
  analytic singularity type of $Y$. In particular, if $Y$ is non-singular then so
  is any flop~$Y^{+}$.
\end{theorem}

There are always many similiarites between the 3--folds $Y$ and $Y^+$,
for example, see \cite[Theorem~3.2.2]{kollar:flips:jdg}, 
but typically are not isomorphic or even diffeomorphic varieties.

\begin{remark}
\label{R:flop:rk2}
  If $Y$ is non-singular and has Picard rank $\rho(Y)=2$, then the $D$--flop of $Y$ does
  not depend on the choice of divisor $D$.
\end{remark}

\begin{remark}
We explain very briefly how flops occur in the context of weak Fano 3--folds; 
we will return to this point after we have developed the basic properties of weak Fano 3--folds
in \fullref{sec:weak-fano-3}. 

For many weak Fano 3--folds $Y$ the anticanonical morphism 
$\varphi_{AC} \co Y \to X$ (see \fullref{d:ac:morphism}) 
will be a flopping contraction in the sense of \fullref{D:flopping}. 
By choosing any \mbox{$\varphi_{AC}$--antiample} divisor $D$ on $Y$ we can
perform the $D$--flop of $\varphi_{AC}$. 
This yields another weak Fano 3--fold $Y^+$ and another projective small birational morphism 
$\varphi^+ \co Y^+ \to X$; $X$ is again the anticanonical model of $Y^+$ and 
$Y^+$ is also smooth by \fullref{T:flops}.
In general $Y$ and $Y^+$ are not isomorphic because the ring structure on cohomology is usually changed.
Thus $Y$ and $Y^+$ are usually different projective small resolutions of the same singular variety $X$; 
$X$ itself will turn out to be a mildly singular but non--$\Q$--factorial
Fano variety: see \fullref{R:weak:fano:ac:model}.

In general $Y^+$ depends on the choice of the $\varphi_{AC}$--antiample divisor $D$. 
By \fullref{R:flop:rk2} the $D$--flop of $Y$ does not depend on $D$ when $\rho(Y)=2$; 
that is, in this case there will be a unique flop of the rank two weak Fano 3--fold $Y$.
However when $\rho(Y)\ge 3$ then $Y$ may admit many different flops depending on the choice of $D$
and all of them are smooth weak Fano 3--folds sharing many properties of $Y$.
\fullref{R:picard:rank:toric:small}(iii) exhibits a smooth weak Fano 3--fold with $\rho(Y)=10$ 
which admits over 80 non-isomorphic projective small resolutions.
\end{remark}

\subsection*{Defect, small $\Q$--factorialisations and small resolutions}
Recall from \fullref{R:q:fact:not:small} that if a singular variety $X$ admits a small birational morphism 
$f \co Y \to X$ then $X$ cannot be $\Q$--factorial. We now introduce a non-negative integer $\sigma(X)$, 
the defect of $X$,  
which quantifies the failure of a singular 3--fold $X$ to be $\Q$--factorial; 
the defect also measures the failure of Poincar\'e duality on the singular variety~$X$.
\begin{definition}
\label{D:defect}
The \emph{defect} of a Gorenstein canonical projective 3--fold is
\[
\sigma(X) = \rank{ \weil(X)/\Pic(X)} = \rank H_{4}(X) - \rank H^{2}(X)\]
where $\weil(X)$ denotes the class group of Weil divisors of $X$ and $\Pic(X)$ the Picard group of $X$.
\end{definition}
\begin{remark}
\label{R:defect:factorial}\hfill
\begin{enumerate}
\item
By definition the defect of $X$ is zero if and only if $X$ is $\Q$--factorial. 
Hence any Gorenstein terminal 3--fold that admits a crepant (and hence small) resolution must have positive defect.
\item
In fact the divisor class group $\weil(X)$ of a terminal Gorenstein 3--fold
$X$ is torsion-free -- see Kawamata~\cite[Lemma~5.1]{kawamata1988crepant} 
who attributes the proof to Reid and Ue -- so that $X$ is $\Q$--factorial if and only if it is factorial. 
In particular, 
the defect $\sigma(X)$ of a terminal Gorenstein 3--fold is zero if and only if $X$ is factorial, 
that is, every Weil divisor is Cartier.
\end{enumerate}
\end{remark}

If $X$ admits a projective small resolution $f \co Y \to X$ then by the previous remark $X$ must have 
defect $\sigma(X)>0$. We can therefore attempt to use the blowup construction from 
\fullref{l:small:morphism} to construct some small projective birational morphism to $X$ by 
choosing a non--$\Q$--Cartier divisor $D' \in \Cl(X)$ and considering $\Proj R(X,D')$. 
However, we must verify that $D'$ satisfies 
the condition that $R(X,D')$ is a finitely generated $\oo_X$--algebra. 
This is not true in general; however for 3--folds Kawamata~\cite[Theorem~6.1]{kawamata1988crepant}
proved the following result.
\begin{theorem}
\label{t:kawamata:fg}
Let $X$ be a Gorenstein canonical 3--fold and $D' \in \Cl(X)$. Then $R(X,D')$ is finitely generated.
\end{theorem}
Kawamata's proof uses the classification of Gorenstein terminal 3--fold singularities described earlier 
in a fundamental way. An easy corollary of \fullref{t:kawamata:fg}, 
also due to Kawamata~\cite[Corollary~4.5]{kawamata1988crepant} is the following.
\begin{corollary}
For any projective 3--fold $X$ with canonical (respectively terminal) singularities 
there exists a small projective birational morphism \mbox{$f\co Y
  \ra X$} such that $Y$ is $\Q$--factorial with at most canonical
(respectively terminal) singularities.  The morphism $f\co Y \ra X$
is said to be a (small) \emph{\mbox{$\Q$--factorialisation}} of $X$.
\end{corollary}
\begin{proof}
The proof is simple given \fullref{t:kawamata:fg}. 
Set $X=X_0$ and choose an arbitrary non--$\Q$--Cartier divisor $D_0 \in \Cl(X_0)$. 
Then by \fullref{t:kawamata:fg} and \fullref{l:small:morphism} \mbox{$X_1:= \Proj R(X_0,D_0)$}
is a normal projective variety and the natural projection to $X$ 
gives a small projective birational morphism $f_1 \co X_1 \to X_0$.
Clearly $\sigma(X_1)< \sigma(X_0)<\infty$. If $X_1$ fails to be $\Q$--factorial we repeat the process 
setting $X_2 = \Proj R(X_1,D_1)$ etc. This process terminates after at most $\sigma(X_0)$ steps 
and yields $Y=X_k$ a $\Q$--factorial variety with canonical (respectively terminal) singularities 
and a small projective birational morphism $f\co Y \to X$. 
\end{proof}
\begin{remark*}
If $Y$ is any $\Q$--factorialisation of $X$ then the defect $\sigma(X)$ can also be calculated as 
\[
\sigma(X) = \rank \Pic(Y) - \rank \Pic(X).
\]
\end{remark*}

\begin{remark*}
The existence of small $\Q$--factorialisations  is now known in all dimensions as a consequence of the 
work of Birkar--Cascini--Hacon--McKernan~\cite{bchm}. One chooses an initial resolution of singularities 
and then runs an appropriate well-directed relative minimal model program
which contracts the exceptional divisors of the resolution and whose
output is the desired small $\Q$--factorialisation. However, Kawamata's result for canonical 
3--folds suffices for our purposes. 
\end{remark*}
\begin{remark}\hfill
\label{r:q:fact:flop}
\begin{enumerate}
\item
Small $\Q$--factorialisations are not unique but one can prove that any two
differ by a sequence of finitely many flops (see
Koll\'ar~\cite[6.38]{kollar:flips:jdg}). 
\item
Suppose that $X$ has only terminal
singularities; if one small $\Q$--factorialisation of $X$ is singular,
then because terminal flops preserve singularities (recall \fullref{T:flops}) 
all small $\QQ$--factorialisations are singular, and then $X$ has no crepant resolutions. 
\end{enumerate}
\end{remark}

A natural question is whether there are only finitely many distinct small
$\Q$--fac\-tor\-i\-al\-i\-sa\-tions 
of a given Gorenstein terminal 3--fold $X$. This follows from the following more general finiteness 
result of Kawamata--Matsuki~\cite[Main~Theorem]{kawamata:matsuki}.
\begin{theorem}
\label{t:crep:finite}
Let $X$ be a projective 3--fold with canonical singularities. 
Then there exist only finitely many projective birational crepant
morphisms $f \co Y\to X$
such that $Y$ is a 3--fold with only canonical singularities. 
\end{theorem}
\begin{remark*}
Generalising \fullref{d:crep:small:res}, 
we say that a projective birational morphism $f \co Y \to X$ is crepant if $K_{Y}=f^{*}K_{X}$.
\end{remark*}
In particular as an immediate corollary of \fullref{t:crep:finite} there
are only finitely many different small $\Q$--factorialisations 
of a given terminal 3--fold $X$. 

We summarise our discussion above.
Given any terminal 3--fold $X$ we can always find small $\Q$--factorialisations $Y$ of $X$, 
but there are only finitely many of them and any two of them are related by a sequence of flops.
For a general $X$,  all $\Q$--factorialisations of $X$ will still be singular; in this case $X$ admits no projective small resolutions. 
In other words, for a general terminal 3--fold $X$ it is quite rare that $X$ admits a projective small resolution.
For many purposes in algebraic geometry the existence of a small $\Q$--factorialisation  of $X$ often suffices;  
however for our later purposes in constructing smooth weak Fano 3--folds as projective small resolutions
of terminal Fano 3--folds, it is crucial that the terminal Fano 3--fold
admit a smooth small $\Q$--factorialisation, that is, 
a projective small resolution.

It can therefore be very subtle to determine whether a 
3--fold $X$ with Gorenstein terminal (respectively canonical) singularities
admits a projective small (respectively crepant) resolution.
Even if we suppose $X$ has only terminal and therefore isolated cDV
singularities and that locally (the analytic germ of) each singularity
admits a small resolution then there are global reasons why $X$
may admit no projective small resolutions. This occurs even in the simplest case where $X$ is nodal, 
that is, has only ordinary double points. 

\subsection*{Projective small resolutions of nodal 3--folds}

We now consider in more detail the projective small
resolution problem for the special case of \emph{nodal} 3--folds: recall \fullref{D:nodal:3fold}; projective small resolutions of nodal Fano 3--folds 
will give rise to smooth weak Fano 3--folds containing special rigid holomorphic curves. 
These special rigid curves will play a crucial role in~\cite{chnp2} because 
they give rise to rigid associative 3--folds in twisted connected sum $\gtwo$--manifolds.
  
As we now explain, it is not a problem to find small resolutions of
$X$ if we are prepared to leave the projective world and work in the
\emph{complex analytic category}; the difficulty is to find
\emph{projective} (or K\"ahler) small resolutions of $X$.  Suppose the
\mbox{3--fold} $X$ has $k$ ordinary nodes $P_{1}, \ldots ,P_{k}$ as its only
singular points.  Let $\wtilde{X}$ denote the blowup of $X$ in all
its singular points; $\pi\co \wtilde{X} \ra X$ is a non-singular
projective 3--fold with $k$ exceptional divisors $E_{1}, \ldots ,
E_{k}$ isomorphic to $\CP^{1}\times \CP^{1}$. There are two natural
projections $\pi_{i}^{\pm}\co E_{i} \ra \CP^{1}$, (rulings of
$\CP^{1} \times \CP^{1}$) corresponding to a choice of $\CP^{1}$
factor.  For each exceptional divisor we make a choice of one of these
two rulings; by Nakano~\cite{nakano1,nakano2} the fibres of every
$\pi_{i}^{\pm}$ can be blown down to yield a non-singular Moishezon
\mbox{3--fold} $\what{X}$, that is, a compact complex 3--fold with three
algebraically independent meromorphic functions.  Thus we obtain
$2^{k}$ Moishezon small resolutions $\what{X}$ of the nodal
3--fold $X$ in which each singular point $P_{i}$ has been replaced by
a non-singular rational curve $C_{i}$ with normal bundle
$\mathcal{O}(-1)\oplus\mathcal{O}(-1)$.  In general some of these
$2^{k}$ small resolutions may be isomorphic.  This happens when the
nodal 3--fold $X$ admits automorphisms permuting the nodes; such
an automorphism will lift to an action on the set of all small
resolutions of $X$ and thereby give rise to isomorphic small
resolutions.

Since all the small resolutions are Moishezon,  a small resolution
$\what{X}$ of $X$ is projective if and only if it is K\"ahler. 
(Recall that Moishezon~\cite{moishezon} proved that a Moishezon manifold is projective if and only if it is K\"ahler.)
A natural but delicate question is therefore:
given a nodal projective 3--fold $X$ with $k$ nodes how many of its $2^{k}$ Moishezon small resolutions are 
projective (K\"ahler)?

In general, even though our initial nodal 3--fold $X$ is projective \emph{none} of its
$2^{k}$ small resolutions need be projective. 
In fact, from our previous results we have the following: 
all $2^{k}$ Moishezon small resolutions of $X$ are non-projective if and only if 
any (and therefore all) projective small $\Q$--factorialisation of $X$ is singular. 
Thus answering the question above about how many of the small resolutions are projective is rather subtle and
it is equivalent to the following two questions
\begin{enumerate}
\item
Does $X$ admit a \emph{smooth} projective small $\Q$--factorialisation? 
\item
If so, how many different projective small $\Q$--factorialisations does $X$ admit?
\end{enumerate}

The existence of projective small resolutions of nodal projective 3--folds 
has been considered by various authors. 
To illustrate some of the issues in concrete cases -- of interest later in this paper -- we 
consider nodal cubics $X\subset \CP^{4}$ with a small number of nodes;   
for a systematic study of projective small resolutions of nodal cubics in $\CP^{4}$ 
see Finkelnberg~\cite{finkelnberg1987small},
Finkelnberg--Werner~\cite{finkelnberg1989small} and Werner~\cite{werner:thesis}.

\subsubsection*{Small resolutions of nodal cubics and weak Fano 3--folds}
Finkelnberg--Werner~\cite{finkelnberg1989small} proved that if a
nodal cubic $X\subset \CP^{4}$ has fewer than $4$ nodes then $X$
itself is already $\Q$--factorial irrespective of the position of its
nodes; therefore $X$ admits no projective small resolutions.  However,
they showed that whether a nodal cubic $X \subset \CP^{4}$ with $4$
nodes is $\Q$--factorial or not depends on the position of the $4$
nodes; $X$ is $\Q$--factorial if and only if the $4$ nodes are not
contained in some projective plane $\Pi^{2} \subset \CP^{4}$. If the
$4$ nodes do lie in some plane then this special surface $\Pi$ gives
us a Weil divisor $D'$ on $X$ which is not $\Q$--Cartier and one
projective small resolution $Y$ is then obtained by blowing up this
plane, that is, by taking $Y=\Proj R(X,D')$ as in \fullref{l:small:morphism}.

In fact, since the nodal cubic $X$ is a mildly singular (Gorenstein
terminal) Fano 3--fold it will turn out that the projective small
resolution $Y$ is a smooth weak Fano 3--fold containing $4$ rigid
rational curves with normal bundle
$\mathcal{O}(-1)\oplus\mathcal{O}(-1)$: one from each of the $4$ nodes
in~$X$.  We will return to nodal cubics in $\CP^{4}$ in \fullref{S:wk:fano:examples} where we will explain how to obtain numerous
smooth weak Fano 3--folds from such nodal cubics, generalising this
example.

This example demonstrates clearly that the existence of projective small resolutions is
a global question which depends on the location of
singularities and not just the local analytic singularity type or
number of singularities.
It also illustrates that if $X$ does not contain some relatively special surfaces 
(in this case the projective plane $\Pi$) then we have no candidate Weil 
non--$\Q$--Cartier divisors $D'$ which we can ``blowup'' to obtain a nontrivial small birational morphism 
as in \fullref{l:small:morphism}.

\subsubsection*{The number of small projective resolutions}
We highlight another aspect of the subtlety of the projectivity of
small resolutions. A cubic $X$ with $4$ nodes containing a plane $\Pi$
as above has defect $\sigma(X)=1$
(see Finkelnberg--Werner~\cite[pages~190--191]{finkelnberg1989small}) and hence the projective small
resolution $\varphi \co Y \to X$ obtained by blowing up the plane
$\Pi$ has Picard rank $\rho(Y)=2$.\footnote{See also
  \fullref{exa:quartic_w_plane} where we prove similar statements
  for a \emph{quartic} 3--fold that contains a plane.}  Therefore, by
\fullref{R:flop:rk2}, $\varphi \co Y \to X$ has a unique flop,
$\varphi^{+}\co Y^{+} \to X$.  Hence, by \fullref{r:q:fact:flop}, $Y^{+}$ is the only other projective small
resolution of $X$ (moreover by~\cite[page~191]{finkelnberg1989small}
these two projective small resolutions are not isomorphic).  In other
words, only $2$ of the $2^{4}=16$ Moishezon small resolutions of $X$
are projective.

More generally, we will see that nodal Fano 3--folds arising as the
anticanonical models of non-singular weak Fano 3--folds of Picard rank
$2$ have exactly \emph{two} projective small resolutions (again
because of the uniqueness of flops when the Picard rank $\rho=2$).
However, in \fullref{S:wk:fano:examples} we will see that such
rank $2$ weak Fano 3--folds can have up to $46$ nodes in their
anticanonical models and therefore admit up to $2^{46}$ Moishezon
small resolutions!

\subsubsection*{Small and crepant resolutions of toric Fano 3--folds}
From the point of view of understanding small (respectively crepant)
projective resolutions one very nice class of Gorenstein terminal
(respectively canonical) 3--folds are the \emph{toric} Gorenstein Fano
3--folds. Some particularly pleasant features of the toric Gorenstein
Fano world are:
\begin{enumerate}
\item
All singularities of toric terminal Gorenstein Fano 3--folds are ordinary double points.
\item 
Toric terminal (respectively canonical) Gorenstein Fano 3--folds are completely classified. 
\item
Every toric terminal Gorenstein Fano 3--fold has at least one projective small resolution, 
and moreover one can enumerate all possible projective small resolutions combinatorially.
\end{enumerate}
We will give a more detailed description of the class of toric Gorenstein terminal (and more generally canonical) 
Fano 3--folds and small (respectively crepant) resolutions thereof later in \fullref{S:wk:fano:examples}; 
this will show that there is a very plentiful supply of toric weak Fano 3--folds.

\subsection*{Rudiments of Mori theory}
We recall some basic terminology from Mori theory: see Debarre's book~\cite{debarre:book} 
for a more detailed introduction and Koll\'ar--Mori's book~\cite{KM} for a complete treatment.
Mori theory will be needed only in \fullref{S:wk:fano:examples} 
when we discuss the classification scheme for weak Fano 3--folds with Picard rank $\rho=2$ and 
so the rest of this section may be safely ignored until then.

\begin{definition}
A $1$--cycle on a projective variety $Y$ is a formal linear combination of irreducible, reduced curves 
$C = \sum_{i}{a_i C_i}$. $C$ is \emph{effective} if $a_i \ge 0$ for every $i$. 
Two $1$--cycles $C$ and $C'$ are \emph{numerically equivalent} if they have the same intersection 
number with every Cartier divisor; we write $C \sim C'$.
$1$--cycles with real coefficients modulo numerical equivalence form a real vector space denoted $N_1(Y)$; 
the class of a $1$--cycle $C$ is denoted $[C]$.
\end{definition}

Inside $N_1(Y)$ sits the (convex) cone of curves $\nef(Y)$, the set of classes of effective $1$--cycles.
\begin{definition} The \emph{cone of curves} $\nef(Y)$ is defined by
$$\nef(Y) := \Big\{ \sum{a_i [C_i]} ~\Big|~ C_i \subset X, 0\le a_i \in \R
\Big\} \subset N_1(Y),$$
where $C_i$ are irreducible curves on $Y$. 
$\nefb(Y)$ is defined to the closure of $\nef(Y)$ in~$N_1(Y)$.
\end{definition}

Let $X$ and $Y$ be projective varieties. Define the \emph{relative
  cone of a morphism} $\pi\co Y \ra X$ as the convex subcone $\nef(\pi)
\subset \nef(Y)$ generated by the classes of curves contracted by
$\pi$.  Since $X$ is projective, an irreducible curve $C$ is
contracted by $\pi$ if and only if $\pi_*[C]=0$; in other words, being
contracted is a numerical property. $\nef(\pi)$ has the additional
property that it is extremal.

\begin{definition}
  Let $V$ be a convex cone in $\R^n$. A subcone $W \subset V$ is
  \emph{extremal} if it is closed and convex and if any two elements
  of $V$ whose sum lies in $W$ are both in~$W$.
  Geometrically, this
  means that the cone $V$ lies on one side of some hyperplane
  containing the extremal subcone~$W$.
  An extremal cone of
  dimension $1$ is called an \emph{extremal ray}.  \end{definition}

\begin{lemma}[Debarre~{{\cite[Propositions~1.14 and~1.43--1.45]{debarre:book}}}]
\label{L:rel:nefcone}
Let $\pi\co Y \ra X$ be a morphism of projective varieties.
\begin{enumerate}
\item The subcone $\nef(\pi) \subset \nef(Y)$ is a closed convex subcone which is extremal.
\item If additionally we assume $\pi_* \mathcal{O}_Y \simeq
  \mathcal{O}_X$ then the morphism $\pi$ is determined by $\nef(\pi)$
  up to isomorphism.
\end{enumerate}
\end{lemma}

\fullref{L:rel:nefcone} says that a morphism determines an extremal subcone of $\nef(Y)$ 
which, under the additional condition given in (ii), characterises that morphism. 
This motivates the following:

\begin{definition}
\label{D:contraction}
Let $Y$ be a projective variety and $F \subset \nefb(Y)$ an extremal face. A morphism 
$\cont_{F}\co Y \ra X$ to a projective variety $X$ is called the \emph{contraction of $F$} if 
\begin{enumerate}
\item $\cont_{F}(C) = pt$, for an irreducible curve $C$ if and only if $[C] \in F$; and
\item $(\cont_{F})_* \mathcal{O}_Y = \mathcal{O}_X$.
\end{enumerate}
\end{definition}

In general not every extremal face can be contracted. A central point
in Mori theory is to find conditions guaranteeing the existence of
$\cont_{F}$. The main result in this direction is the following deep
theorem often called the Contraction Theorem.

\begin{theorem}[Contraction Theorem]
Let $Y$ be a projective variety with at worst canonical singularities and let $F \subset \nefb(Y) $ be an extremal face 
on which $K_{Y}$ is negative; 
then the contraction $\cont_{F}$ (as in \fullref{D:contraction}) exists.
\end{theorem}

\begin{remark*}
Koll\'ar-Mori~\cite[Theorem~3.7.3]{KM} in fact proves a more 
general version of the Contraction Theorem (for klt pairs). 
\end{remark*}

From now on we focus on the case of contractions associated with extremal rays. 
The following result says that a contraction associated with an
extremal ray comes in three basic flavours: see \cite[Proposition~2.5]{KM}.

\begin{prop}
  \label{P:ext:contract:types}
  Let $Y$ be a normal projective variety that is $\Q$--factorial. Let
  $\cont_{R}\co Y \ra X$ be the contraction of an extremal ray $R \subset
  \nefb(Y)$. Then one of the following holds:
  \begin{enumerate}
\item (fibre type contraction) $\dim{Y} > \dim{X}$;
\item (divisorial contraction) $f$ is birational and $\ex(f)$ is an
  irreducible divisor;
\item \label{it:small_contr} (small contraction)
$f$ is birational and $\ex(f)$ has codimension $\ge 2$.
\end{enumerate}
\end{prop}

Mori~\cite{mori:1982} gave a description of all 
contractions of extremal rays on a \emph{non-singular} projective 3--fold.
\begin{theorem}
\label{T:Mori:3:contract}
Let $Y$ be a non-singular projective 3--fold and $\cont_{R}\co Y \ra X$
be the contraction of a $K_{Y}$--negative extremal ray 
$R \subset \nefb(Y)$. The following is a complete list of possibilities for $\cont_{R}$:
\begin{enumerate}
\item[E] (exceptional) $\dim{X}=3$, $\cont_{R}$ is birational and there are five types of local behaviour near the contracted 
surface
\begin{enumerate}
\item[E1]
$\cont_{R}$ is the inverse of the blowup of a non-singular curve in the non-singular threefold $X$
\item[E2]
$\cont_{R}$ is the inverse of the blowup of a non-singular point of the non-singular threefold $X$
\item[E3]
$\cont_{R}$ is the inverse of the blowup of an ordinary double point of $X$
\item[E4]
$\cont_{R}$ is the inverse of the blowup of an isolated cDV point of $X$ which is locally analytically given by the 
equation $x^{2}+y^{2}+z^{2}+w^{3}=0$.
\item[E5]
$\cont_{R}$ contracts a non-singular $\CP^{2}$ with normal bundle $\mathcal{O}(-2)$ to a point on $X$ which is 
locally analytically the quotient of $\C^{3}$ by the involution $(x,y,z) \mapsto -(x,y,z)$.
\end{enumerate}
\item[C] (conic bundle) $\dim{X}=2$, $\cont_{R}$ is a fibration whose fibres are plane conics.
\item[D] (del Pezzo) $\dim{X}=1$ and the general fibre of $\cont_{R}$ is a del Pezzo surface of degree $d\neq 7$.
\item[F] (Fano) $\dim{X}=0$, $-K_{X}$ is ample and hence $X$ is a Fano variety.
\end{enumerate}
\end{theorem}
\begin{remark}
\label{R:no:small:contract}
Note that for non-singular projective threefolds there are no small extremal contractions -- type (iii) in \fullref{P:ext:contract:types}.
\end{remark}

\normalsize

\section{Weak Fano \mbox{3--folds}: basic theory}
\label{sec:weak-fano-3}

\subsection{Weak Fano 3--folds and semi-Fano 3--folds}
\label{subsec:weak_fanos}

In this section, we review the definition and elementary properties of
weak Fano \mbox{3--folds}. We postpone any in-depth discussion of examples of
weak Fano 3--folds until Sections~\ref{sec:examples} and~\ref{S:wk:fano:examples}  giving only
two of the simplest weak Fano 3--folds as Examples~\ref{E:nodal:quadric} and~\ref{E:wk:fano:ruled}.

\begin{definition}
\label{dfn:weak_fano}
A \emph{weak Fano \mbox{3--fold}} is a non-singular projective complex
\mbox{3--fold} $Y$ such that the
anticanonical sheaf $-K_Y$ is a nef and big line bundle 
(recall \fullref{def:line:positivity} for the definitions of big and nef).
The \emph{index} of a weak Fano 3--fold $Y$ is the integer $r=\gdiv{c_1(Y)}$, 
that is, the greatest divisor of $c_{1}(Y) \in H^{2}(Y;\Z)$.
\end{definition}

\begin{remark}
\label{r:index}
The index $r(Y)$ of a weak Fano 3--fold belongs to $\{1,2,3,4\}$. 
The only weak Fano 3--fold with index $4$ is $\CP^3$ (which of course is Fano).
Weak Fano 3--folds with index $3$ are classified; besides the quadric in $\CP^{4}$ (which is Fano) 
there are only two further weak Fano 3--folds of index $3$, namely Examples~\ref{E:nodal:quadric} and~\ref{E:wk:fano:ruled}: 
see Casagrande--Jahnke--Radloff~\cite[Proposition~3.3]{cjr} and Shin~\cite[Theorem~3.9]{shin}.
Weak Fano 3--folds with index $2$ are called \emph{weak del Pezzo} (or sometimes \emph{almost del Pezzo}) 
3--folds. 
There are relatively few weak del Pezzo 3--folds: see Jahnke--Peternell~\cite{jp:almost:dp} for a partial classification. 
However, note that a single deformation class of smooth del Pezzo 3--folds may give rise 
to a fairly large number of different deformation classes of weak del Pezzo 3--folds: 
see the discussion of nodal cubics in \fullref{S:wk:fano:examples} for a more concrete demonstration
of this phenomenon.
Nevertheless, the vast majority of weak Fano 3--folds have index $1$.
\end{remark}

We construct a handful of examples of weak Fano 3--folds in \fullref{sec:examples} and discuss partial classification results and show the
existence of many more examples in \fullref{S:wk:fano:examples}; 
the crucial point is that there are \emph{many} more deformation families of weak Fano 3--folds
than Fano 3--folds, though still only finitely many: see \fullref{t:wk:fano:finite}.

In the next several paragraphs, we summarise a few standard 
facts on weak Fano \mbox{3--folds} which play an important role in this paper.
These are properties of Fano
\mbox{3--folds} which extend without much difficulty to the case when
the anticanonical bundle is only nef and big. Most of these follow by applying 
Kawamata--Viehweg vanishing (recall \fullref{T:KV:vanish})
wherever we would have used Kodaira vanishing in the Fano case.
For an introduction to some of the basic properties of weak Fano
\mbox{3--folds}, see Reid \mbox{\cite[Section~4]{reid:kaw}}.

\begin{corollary}
\label{C:weak:fano:vanish}
For any  weak Fano \mbox{3--fold} $Y$ we have
  \begin{enumerate} 
  \item All Hodge numbers $h^{i,0}=h^{0,i}=0$ for $i>0$.
  \item  The natural homomorphism $\Pic Y\to H^2(Y; \Z)$ is an isomorphism.
  \item \label{it:c1c2} $K_Y \cdot c_2(Y) = -24$,
     where $c_2(Y) \in H^4(Y;\ZZ)$ denotes the second Chern class of~$Y$.
  \item  The dimension of the space of holomorphic sections of $-K_{Y}$ is given by 
    \begin{equation}
  \label{eq:1}
  h^0(Y,-K_Y)=g+2
  \quad
  \text{where}
  \quad
  -K_Y^3=2g-2. 
    \end{equation}
  \end{enumerate} 
\end{corollary}
\begin{proof}
  Part~(i) follows immediately from Kawamata--Viehweg vanishing
  (\fullref{T:KV:vanish}) and Hodge theory.
  Part~(ii) follows from the exponential short exact sequence
  ${0 \ra \Z \ra \mathcal{O} \ra \mathcal{O}^{*} \ra 0}$ and
  $h^{1}(\mathcal{O}_{Y})=h^{2}(\mathcal{O}_{Y})=0$. For part~(iii)
  recall that Riemann--Roch in the case of a line bundle
  $L$ on a non-singular 3--fold $Y$ gives
\begin{equation}
\label{eq:rr}
\chi(Y,L) := \sum_{i=0}^{3}{(-1)^{i}h^{i}(L)} = \tfrac{1}{6}L^{3} -
\tfrac{1}{4}L^{2}K_{Y} + \tfrac{1}{12}L(K_{Y}^{2}+c_{2}) - \tfrac{1}{24}
K_{Y}\cdot c_{2}(Y) .
\end{equation}
Using part~(i) and setting $L=0$ we obtain $K_{Y} \cdot c_{2}(Y) =-24$.
For part~(iv) we now apply Kawamata--Viehweg vanishing with $L=-2K_Y$ to yield
\[
h^{0}(Y,-K_{Y}) = \chi(Y, -K_Y) = - \tfrac{1}{2}K_{Y}^{3} + 3 = g+2.\proved
\]
\end{proof} 
 
\begin{definition}
  The invariant $g$ in \eqref{eq:1} is called the  \emph{genus} of $Y$; 
  the even integer $(-K_{Y})^{3} = 2g-2$ is called the \emph{anticanonical degree} of $Y$.
\end{definition}

The following facts about anticanonical sections of weak Fano \mbox{3--folds} 
are well known to algebraic geometers; they are central to the current paper.

It follows from the vanishing results in \fullref{C:weak:fano:vanish}
together with adjunction that if a member $S\in \acls{Y}$ is
non-singular, then it is a K3 surface.  For smooth Fano 3--folds the existence of a non-singular
member $S \in \acls{Y}$ is due to Shokurov~\cite{shokurov:k3}.
For a weak Fano 3--fold $Y$ with at worst canonical Gorenstein singularities 
Reid~\cite[Theorem~0.5]{reid:kaw} proved that a general $S \in \acls{Y}$ is an irreducible K3 with at worst rational double 
point singularities. Using Reid's result Paoletti~\cite[Proposition~2.1]{paoletti} deduced the following:

\begin{theorem}
\label{thm:reid}
If $Y$ is a non-singular weak Fano 3--fold then 
a general anticanonical member $S\in  \acls{Y}$ is a non-singular K3 surface.
\end{theorem}

 To define the anticanonical morphism and the anticanonical model associated with 
 any weak Fano 3--fold we need the following:

\begin{theorem}
  If $Y$ is a weak Fano \mbox{3--fold}, then $-K_Y$ is  semi-ample.
\end{theorem}

\begin{proof}
The anticanonical divisor of $Y$ is big and nef and hence by the 
Basepoint-free Theorem (apply  Reid~\cite[Theorem~0.0]{reid:kaw} with $D=-K_Y$ and $a=1$)
the linear system $\abs{-nK_Y}$ is basepoint-free for $n$ sufficiently large.
\end{proof}

Since $-K_{Y}$ is semi-ample, by \fullref{r:semiample:fg} $-K_{Y}$ is finitely generated 
and the birational morphism $\varphi \co Y \to \Proj R(Y,-K_{Y})$ 
coincides with the algebraic fibre space $\varphi \co Y \to X$ given by \fullref{t:semiample}.

\begin{definition}
\label{d:ac:morphism}
  If $Y$ is a weak Fano \mbox{3--fold}, we call the finitely generated
  ring 
\[
R(Y,-K_Y)=\bigoplus_{n\geq 0} H^0(Y;-nK_Y)
\] 
the \emph{anticanonical ring} of $Y$, %
the birational morphism $\varphi\co Y \to X=\Proj R(Y,-K_Y)$ attached
to $\acls{Y}$ the \emph{anticanonical morphism} and $X$ the
\emph{anticanonical model} of $Y$.
\end{definition}

We will sometimes abbreviate anticanonical as AC and therefore refer to the AC morphism or AC model of a weak 
Fano \mbox{3--fold} $Y$.

\begin{remark}\quad
\label{R:weak:fano:ac:model}
\begin{enumerate}
\item It is clear that $Y$ is a resolution of singularities of the anticanonical model
  $X$. It is more-or-less a tautology that
\[
K_Y=\varphi^\ast K_X, \quad
\text{and}
\quad
R(Y,-K_Y)=R(X,-K_X) .
\pagebreak[1]
\]
In particular, $Y$ is a \emph{crepant} resolution of $X$. It follows
immediately that $X$ has Gorenstein canonical singularities and that
$-K_X$ is an ample line bundle; thus $X$ is in its own right a
singular Fano variety with at worst Gorenstein canonical
singularities.
\item Conversely, given a Fano 3--fold $X$ with Gorenstein canonical
  singularities one can ask whether $X$ admits a non-singular projective
  crepant resolution $Y$; any such $Y$ will be a non-singular weak Fano
  3--fold with $X$ as its anticanonical model.
\item If $\acls{Y}$ has two non-singular members $S_1,S_2$ intersecting
  \emph{transversally} then the intersection $S_{1}\cap S_{2}$ is a
  canonically polarized curve $C$ (that is, $-K_{Y|C}=K_C$) of
  genus~$g$.
\item In all examples we consider, the anticanonical ring
  $R(Y,-K_Y)$ is generated in degree $1$; 
equivalently, $-K_X$ is very ample. In this case non-singular members $S_1$, $S_2$ always exist.
The few Gorenstein canonical Fano 3--folds $X$ for which $-K_{X}$ 
fails to be very ample are classified by Jahnke--Radloff~\cite[Theorem~1.1]{jr:basepoints}.
\end{enumerate}
\end{remark}

Now we define a subclass of weak Fano 3--folds that will play an important role throughout 
the rest of the paper and in our paper~\cite{chnp2}.
First recall the definitions of small and semi-small projective birational
morphisms from \fullref{D:semismall:small}.

\begin{definition}
\label{d:semifano}
Let $Y$ be a weak Fano 3--fold and $\varphi\co Y \ra
  X$ its anticanonical morphism. If $\varphi$ is \emph{semi-small}, we
  call $Y$ a \emph{semi-Fano} 3--fold.
\end{definition}

\begin{remark}\quad
\label{R:ac:model}
\begin{enumerate}
\item The anticanonical morphism $\varphi \co Y \to X$ of a semi-Fano 3--fold may
  contract divisors to curves, or curves to points, but not divisors
  to points.
\item If the anticanonical morphism $\varphi \co Y \to X$ of a
  semi-Fano 3--fold is \emph{small} and not just semi-small then it
  contracts only a finite number of curves to points.  $X$ is then a
  Fano 3--fold with Gorenstein terminal and therefore isolated cDV
  singularities (recall \fullref{def:cDV}); the curves $C \subset Y$
  contracted by $\varphi$ give rise to the isolated cDV points in
  $X$.  In this case $\varphi$ is a flopping contraction in the sense
  of \fullref{D:flopping}.  Hence if $D$ is any
  $\varphi$--antiample (recall \fullref{r:f:ample}) Cartier divisor on $Y$
  by \fullref{T:flops} we may perform the $D$--flop of
  $\varphi$. This yields another semi-Fano 3--fold $Y^{+}$ whose
  anticanonical model $\varphi^{+} \co Y^{+} \to X$ is another
  small projective birational morphism, and where $D^+$ is
  $\varphi^+$--ample. Thus each semi-Fano
  3--fold $Y$ with small anticanonical morphism gives rise to at least
  one other semi-Fano 3--fold with small
  anticanonical morphism and the same anticanonical model $X$. (If
  $\rho(Y/X)>1$, there can be several other semi-Fano 3--folds with
  the same anticanonical model. For all Cartier divisors $D$ on~$Y$,
  there is a sequence of flops $Y\dasharrow Y^\prime$, with
  anticanonical model $\varphi^\prime \co Y^\prime \to X$, such
  that $D^\prime$ is $\varphi^\prime$--nef.)
\item In our construction of twisted connected sum \gtwo--manifolds
in~\cite{chnp2} we will be particularly interested in semi-Fano 3--folds
  $Y$ with \emph{nodal} anticanonical model~$X$, that is, the only singular
  points of $X$ are ordinary double points.  In this special case of
  (ii) the curves contracted by $\varphi$ are finitely many `rigid'
  rational curves with normal bundle
  $\mathcal{O}(-1)\oplus\mathcal{O}(-1)$.  These curves are the
  exceptional curves over the nodes of the anticanonical model $X$.
  As in the previous part we can flop $\varphi \co Y \to X$ to
  obtain finitely many other semi-Fano 3--folds with the same nodal
  anticanonical model $X$.  Each of these rigid rational curves $C$ in
  $Y$ will give rise to a compact rigid holomorphic curve in any ACyl
  Calabi--Yau 3--fold $V$ constructed from $Y$ using \fullref{prop:onestage}.  These compact rigid holomorphic curves in our
  ACyl Calabi--Yau 3--folds will in turn be the source of compact
  rigid associative 3--folds in the twisted connected sum
  $\gtwo$--manifolds we construct in~\cite{chnp2} out of pairs of ACyl
  Calabi--Yau 3--folds.
\end{enumerate}
\end{remark}

For smooth Fano 3--folds we know there are precisely $105$ deformation families. For weak Fano 3--folds we still have:
\begin{theorem}
\label{t:wk:fano:finite}
There are only finitely many deformation families of smooth weak Fano 3--folds.
\end{theorem}
\begin{proof}
The anticanonical model $X$ of a weak Fano 3--fold is a Gorenstein canonical Fano 3--fold. 
Gorenstein canonical Fano 3--folds form a bounded family: see 
Koll\'ar--Miyaoka--Mori--Takagi~\cite[Corollary~1.3]{kmmt} (which 
proves the same holds for all canonical $\Q$--Fano 3--folds).
Applying \fullref{t:crep:finite} we see that each deformation family of Gorenstein canonical Fano 3--folds 
gives rise to only finitely many deformation families of weak Fano 3--folds.
\end{proof}
We are not aware of another reference for \fullref{t:wk:fano:finite} but it was surely known to various experts.
\begin{remark}\hfill
\begin{enumerate}
\item
The previous theorem does not yield any estimate on the number of deformation families of weak Fano 3--folds. 
For this we would need an improvement of \fullref{t:crep:finite} that gives 
a quantitative bound on the number of projective crepant resolutions of  a given Gorenstein canonical Fano 3--fold.
\item
Even if we restrict to the toric world we will see in \fullref{S:wk:fano:examples} 
that there are over 4000 deformation families of toric Gorenstein canonical Fano 3--folds $X$.
Each such $X$ has at least one and often a large number of projective crepant (toric) resolutions and therefore 
gives rise to potentially many deformation families of toric weak Fano 3--folds with the same AC model $X$. 
Moreover, almost 900 of these toric Gorenstein canonical Fano \mbox{3--folds} $X$ give rise to semi-Fano 
\mbox{3--folds} in the sense of \fullref{d:semifano}. So even toric semi-Fano 3--folds are very plentiful.
\end{enumerate}
\end{remark}

To make the discussion more concrete we give two of the simplest
examples of \mbox{semi-Fano} 3--folds of different flavours; we will give
many more examples of weak Fanos in Sections~\ref{sec:examples} and~\ref{S:wk:fano:examples}.
Examples~\ref{E:nodal:quadric} and~\ref{E:wk:fano:ruled} are the only two weak Fano 3--folds 
of index three (recall \fullref{r:index}) besides the quadric $Q^{3}$ which of course is Fano.

\begin{example}
\label{E:nodal:quadric}
Let $X \subset \CP^{4}$ be the projective cone over a smooth %
quadric surface $Q \simeq \CP^{1}\times \CP^{1} \subset \CP^{3}$; $X$ is a
Gorenstein terminal Fano 3--fold with Picard rank $1$, defect $1$,
anticanonical degree $54$, index $3$ and $1$ ODP at the apex of the cone.
$X$~has
two small resolutions $Y$ and $Y^{+}$ both of which are projective and
isomorphic to ${\CP(\mathcal{O}\oplus\mathcal{O}(-1)\oplus
\mathcal{O}(-1))}$ (where $\mathcal{O}(d)$ denotes
$\mathcal{O}_{\CP^{1}}(d)$) which is a $\CP^{2}$--bundle over~$\CP^{1}$.
The anticanonical morphism $\varphi\co Y \ra X$ contracts
the unique section $C_{0}$ with normal bundle $\mathcal{O}(-1)\oplus
\mathcal{O}(-1)$. $Y$ is the unique non-singular toric weak Fano 3--fold with nodal
anticanonical model and Picard rank $\rho=2$: see \fullref{R:picard:rank:toric:small}.
\end{example}

Next we give a simple example of a semi-Fano 3--fold $Y$ with
$\rho=2$ which is \mbox{semi-small} but not small, that is, for which the anticanonical morphism
contracts a divisor to a curve; as in the previous example $Y$ is a
$\CP^{2}$--bundle over $\CP^{1}$ and $Y$ is toric.

\begin{example}
\label{E:wk:fano:ruled}
$Y = \CP(\mathcal{O}(-2)\oplus\mathcal{O}\oplus \mathcal{O})$ (where as
above $\mathcal{O}(d)$ denotes $\mathcal{O}_{\CP^{1}}(d)$) is a non-singular
rank $2$ toric weak Fano 3--fold. 
As in the previous example $Y$ is a weak Fano 3--fold of index $3$ and anticanonical degree $54$. 
However, in this case one can verify that the anticanonical morphism
$\varphi  \co Y \to X$
contracts the divisor $D=\CP(\mathcal{O}\oplus \mathcal{O})$
to a curve along which $X$ has $A_{1}$ singularities. 
\end{example}

\begin{remark*}
If $Y$ is a weak Fano 3--fold with index three then its 
anticanonical model $X$ is a Gorenstein Fano 3--fold of index three with at worst canonical singularities. 
Hence by Shin~\cite[Theorem~3.9]{shin} $X$ is isomorphic to some hyperquadric in $\CP^{4}$.
\end{remark*}

\subsection*{Smoothing terminal Fano 3--folds and semi-Fano 3--folds}

A very useful result which yields some modest control over terminal
Gorenstein Fano 3--folds and hence over non-singular weak Fano 3--folds with
\emph{small} anticanonical morphism is Namikawa's smoothing theorem for
terminal Fano 3--folds: see Namikawa
\mbox{\cite[Theorems~11 and~13]{namikawa:fano:smooth}}
and also Namikawa--Steenbrink~\cite[Lemma~3.4]{namikawa:steenbrink}.
\begin{theorem}
  Let $X$ be a Fano 3--fold $X$ with Gorenstein terminal
  singularities.
\label{T:fano:smooth}
\begin{enumerate}
\item $X$ is smoothable by a flat deformation, and hence is a
  degeneration of a non-singular Fano 3--fold $X_t$ from the
  Iskovskih--Mori--Mukai classification.  In particular, the
  anticanonical degrees, the Picard ranks and the Fano indices of $X$
  and $X_t$ are equal.
\item Suppose that a non-singular Fano 3--fold $X_t$ degenerates to
  $X$ by a flat deformation. Then we have
\begin{equation}
\label{E:bound:nodes}
e(X) + \sum_{p \in Sing(X)}\mu(X,p)  \le 21 - \tfrac{1}{2}\chi(X_t)  =  h^{2,1}(X_t)+ 20 - \rho(X_t),
\end{equation}
where $\chi(X_t)$ is the topological Euler characteristic of $X_t$,
$e(X)$ is the number of ordinary double points of $X$ and $\mu$ is the
non-negative integer invariant of an isolated rational singularity
defined in~\cite[Section~2]{namikawa:fano:smooth} ($\mu$ vanishes for an
ODP and is positive for other Gorenstein terminal singularities).
\item In the case considered in \textup{(ii)} we have
\[
H^{i}(X,\Z) \cong H^{i}(X_t,\Z) \qquad \text{for } i\neq 3,4,
\]
and the defect of $X$ satisfies
\begin{equation*}
  \sigma(X) = b^{3}(X) - b^{3}(X_t) + \sum_{P \in Sing(X)}{m_{P}},
\end{equation*}
where $m_{P}$ denotes the Milnor number of the isolated hypersurface
singularity $P \in X$.  In particular, the defect of a nodal Fano
3--fold $X$ with $e$ nodes and Fano smoothing $X_t$ satisfies
\begin{equation}
\label{E:b3:smooth}
\sigma(X) = b^{3}(X) - b^{3}(X_t) + e.
\end{equation}
\end{enumerate}
\end{theorem}

\begin{remark*}
  Namikawa proves a slightly more general smoothing result than we
  have stated. His result generalises earlier work of Friedman~\cite[Corollary~4.2]{friedman:double:pts}.  The anticanonical degree of a
  Gorenstein canonical Fano 3--fold can be as large as $72$; in
  particular, since the maximal anticanonical degree of a non-singular Fano
  3--fold is $64$ (attained only by $\CP^{3}$), there can be no general
  smoothing result for canonical Fano 3--folds analogous to \fullref{T:fano:smooth}. A more fundamental reason is that smoothings of
  canonical singularities are much more subtle than smoothings of
  terminal singularities. 
\end{remark*}

\begin{remark}
\hfill{}
\label{R:b3:bound}
\begin{enumerate}
\item Since we have only finitely many topological types of non-singular
  Fano \mbox{3--fold}, $\chi(X_t)$ is bounded over all non-singular Fano \mbox{3--folds}
  $X_t$ and hence we get a bound on the maximum number of singular points
  (in particular ODPs) that any terminal Gorenstein Fano 3--fold $X$
  can have.  Consulting the Iskovskih--Mori--Mukai classification we
  find that
\[
10 \le 21-\tfrac{1}{2}\chi(X_t) = h^{2,1}(X_{t}) + 20 - \rho(X_{t})\le 71.
\]
We can also consult the classification to compute $\chi(X_t)$ in any given case. 
\item The term $h^{2,1}(X_{t})$ on the RHS of \eqref{E:bound:nodes}
  varies between $0$ and $52$. Only $11$ Fano 3--folds have 
  $h^{2,1}\ge 5$ and all such examples have relatively small
  anticanonical degree, for example, smooth quartics have $h^{2,1}=30$ and anticanonical degree~$4$.
    On the other hand, $\rho(X_{t})$ varies only
  between $1$ and $10$, and exceeds $5$ only when the Fano 3--fold is
  the product of $\CP^{1}$ with a del Pezzo surface.  So the main
  contribution to the variation in the bound on the RHS of
  \eqref{E:bound:nodes} comes from the variation of $h^{2,1}$. In
  particular only terminal Fano 3--folds which smooth to Fano
  3--folds with large $h^{2,1}$ (which from the classification have
  small anticanonical degree) can have a large number of nodes.
  In~\cite{chnp2} we construct compact \mbox{\gtwo--manifolds} from a pair of 
  ACyl Calabi--Yau 3--folds via the twisted connected sum construction. 
  When both ACyl Calabi--Yau 3--folds arise from blowing up a generic AC pencil 
  on a semi-Fano 3--fold with nodal AC model $X$ we can produce one rigid 
  associative 3--fold with topology $S^1 \times \CP^1$ for each node of $X$.
  Therefore bounds on the number of nodes of nodal Fano 3--folds imply 
  bounds on the number of rigid associative 3--folds we can exhibit in our $\gtwo$--manifolds.
\item The maximum for $21 -\tfrac{1}{2}\chi(X_t)$ of $71$ is achieved
  only for sextic double solids.  The next highest value is $49$ which
  is achieved only for quartics in $\CP^{4}$; 
  by de~Jong--Shepherd-Barron--Van~de~Ven~\cite{dejong} a nodal quartic has at
  most $45$ nodes  (see \fullref{exa:burkhardt_quartic} for such a quartic), 
  so the bound from
  \eqref{E:bound:nodes} is not sharp (but not so far from sharp
  either).
\item If the singularity at $P\in X$ is given
by \mbox{$\{f(x,y,z,t)=0\}$} in local analytic coordinates
(recall \fullref{def:cDV}), for $f$ a polynomial with an
  isolated critical point at the origin, then the Milnor fibre is
  \mbox{$\{f(x,y,z,t)=1\}$}.  It is homotopic to a bouquet of
  $3$--spheres and hence its cohomology is supported in degrees $0$
  and~$3$; the Milnor number $m_{P}$ is equal to the number of spheres in
  the bouquet. In particular, $m_P=1$ if and only if $P$ is an ODP.
\end{enumerate}
\end{remark}

Let $Y$ be a non-singular semi-Fano 3--fold with nodal anticanonical model
$X$.  We can compute the third Betti number of $Y$ in terms of the
defect $\sigma$ of $X$, the number of nodes $e$ and the third Betti
number $b$ of a Fano smoothing $X_{t}$ of $X$ as follows.

\begin{lemma}
\label{L:weak:fano:b3}
Let $Y$ be a non-singular semi-Fano 3--fold with nodal anticanonical model
$X$, and containing $e$ exceptional $(-1,-1)$ curves. Let $\sigma$ be
the defect of $X$ and $b=b^{3}(X_{t})$ the third Betti number of a
Fano smoothing $X_{t}$ of the nodal Fano 3--fold $X$.  Then
\begin{equation}
\label{E:weak:fano:b3}
b^{3}(Y) = b - 2e + 2\sigma.
\end{equation}
\end{lemma}

\begin{proof}
  We will compare $Y$ to $X$ via the small resolution $\varphi\co Y \ra
  X$ and also $X=X_{0}$ to a $1$--parameter Fano smoothing $X_{t}$; the
  existence of the Fano smoothing of $X$ follows from \fullref{T:fano:smooth}.

In the computation of the cohomology of the small resolution, and elsewhere in
this paper, we work in the  derived category of sheaves with (Zariski)
constructible cohomology. If $X$ is an algebraic variety then $\ZZ_X$ denotes
the sheaf with constant fibre $\ZZ$, that is, if $U\subset X$ is open and
connected, then $\ZZ_X(U)=\ZZ$. The sheaf cohomology groups of
$\ZZ_X$ -- calculated by taking an injective resolution $\ZZ_X\to
\mathcal{I}^\bullet$ -- are isomorphic to the singular cohomology groups of $X$
(with integer coefficients).  If $\varphi \co Y \to X$ is a morphism then
$R\varphi_\star \ZZ_Y$ denotes the derived direct image: it is a complex of
sheaves on $X$ with the property that $H^m(X,R\varphi_\star \ZZ_Y)=H^m(Y;\ZZ)$.

We use the (nonsplit) exact triangle:
\[
\ZZ_X \to R\varphi_\star \ZZ_Y \to \bigoplus_{i=1}^e\ZZ_{P_i}[-2]\overset{+1}{\longrightarrow}
\]
where $P_i\in X$ are the $e$ nodes, and $\ZZ_{P_i}$ the skyscraper
sheaf at $P_i$. This %
gives rise to the long exact sequence:
\[
(0)\to H^2(X) \to H^2(Y) \to \ZZ^e \to
H^3(X)\to H^3(Y)\to (0)
\]
and $H^4(X)\simeq H^4(Y)$. 
The exact sequence shows that
\begin{equation}
 \label{eq:small_res}
b^3(Y)=b^3(X)-e+b^{2}(Y)-b^{2}(X)=b^3(X)-e + \sigma.
\end{equation}
From \eqref{E:b3:smooth} we have $b^{3}(X) = b-e+\sigma$.
\eqref{E:weak:fano:b3} follows immediately by substituting for
$b^{3}(X)$ into \eqref{eq:small_res}.
\end{proof}
We will use \fullref{L:weak:fano:b3} repeatedly to compute 
$b^3$ of the weak Fano 3--folds that arise in \fullref{sec:examples}.

\subsection{ACyl Calabi--Yau 3--folds from weak Fano 3--folds}

We now explain how one can obtain  a compact 3--fold
$Z$, to which \fullref{thm:acyl_calabi} can be applied to construct ACyl
Calabi--Yau manifolds, from almost any weak Fano 3--fold $Y$. 
 Recall that $Z$ needs to fibre over $\CP^1$ with fibres in the
anticanonical linear system, and some smooth fibres. Since by \fullref{thm:reid} a generic
anticanonical divisor on a weak Fano 3--fold $Y$ is a smooth K3 surface, it is natural to consider the 3--fold $Z$ obtained by
resolving the indeterminacies of a pencil in $\acls{Y}$.
We will mostly consider pencils with smooth base loci, so that we
can perform the resolution by blowing it up. As explained in \fullref{R:weak:fano:ac:model}(iv), this is the case for a generic anticanonical
pencil on almost any weak Fano; we assume this from now on.

\begin{assumption}
The linear system $\acls{Y}$ of the weak Fano 3--fold $Y$ contains 
two non-singular members $S_0, S_\infty$ intersecting transversally.
\end{assumption}

Under this additional (mild) assumption on the weak Fano $Y$ we can
apply the following Proposition to obtain a compact projective 3--fold
$Z$ that satisfies the hypotheses of the ACyl Calabi--Yau \fullref{thm:acyl_calabi}.
\begin{prop}%
\label{prop:onestage}
Let $Y$ be a closed Kähler (respectively projective) 3--fold, and
suppose that $|S_0,S_\infty| \subset \acls{Y}$ is a pencil with smooth
(reduced) base locus~$C$, and that $S \in |S_0,S_\infty|$ is a smooth
divisor.  Then the blow-up $\blow \co Z \to Y$ at $C$ is a closed
Kähler (respectively projective) 3--fold with a fibration $f \co Z
\to \bbp^1$ with anticanonical fibres. The proper transform of $S$ in
$Z$ is isomorphic to $S$, and the image in $H^{1,1}(S)$ of the Kähler
cone of $Z$ contains the image of the Kähler cone of $Y$.
\end{prop}

\begin{proof}
  The proper transform of each element of $|S_0,S_\infty|$ is an
  anticanonical divisor on $Z$, and together they form a
  base-point-free pencil in $\acls{Z}$, thus defining the required
  fibration. If $[\omega_0] \in H^{1,1}(Y)$ is a Kähler class, then
  there is $\lambda_0 > 0$ such that $\blow^*[\omega] - \lambda[E]$ is
  a Kähler class on $Z$ for any Kähler class $[\omega]$ in a
  neighbourhood $U$ of $[\omega_0]$ and $\lambda \leq \lambda_0$
  (where $E$ is the exceptional divisor).  The map from the Kähler
  cone of $Y$ to the image of $H^{1,1}(Y)$ in $H^{1,1}(S)$ is open, so
  for sufficiently small $\lambda > 0$ there is some $[\omega] \in U$
  such that $[\omega]_S = [\omega_0]_{|S} + \lambda[C]$. Thus
  $[\omega_0]_{|S}$ lies in the image of the Kähler cone of $Z$.
\end{proof}

We will also consider some pencils where the base locus is reducible, but
each component $C_i$ is smooth and of multiplicity one. Blowing up $C_1$ gives
a new smooth K\"ahler 3--fold $Z_{1}$
with an anticanonical pencil whose base locus is the proper transform of the
remaining components. We can thus obtain a suitable 3--fold $Z$ by blowing up the
components in sequence; compare with the discussion preceding Example~2.7 in
Kovalev--Lee~\cite{kovalev:lee}. When the $C_i$ meet transversely, this is equivalent to
blowing up the base locus and then making a projective small resolution of the
ordinary double points resulting from the double points of the base locus.

\begin{prop}
\label{prop:sequence}
Let $Y = Z_0$ be a closed Kähler 3--fold, and suppose that
$|S_0,S_\infty| \subset \acls{Y}$ is a pencil with base locus $C_1
\cup \cdots \cup C_k$ so that $C_i$ is smooth (and reduced) and
suppose $S \in |S_0,S_\infty|$ is a smooth divisor.  Let $Z_i$ be the
blow-up of $Z_{i-1}$ at the proper transform of $C_i$.  Then $Z = Z_k$
satisfies the conclusions of \fullref{prop:onestage}.
\end{prop}

\begin{proof}
  The proof is a straightforward variation on the proof of
  \fullref{prop:onestage}. The base locus of the pencil $|S_0,
  S_\infty|$ is resolved by blowing up all of the curves $C_i$, in
  any order. Thus there is a fibration $f\co Z\to \PP^1$.
\end{proof}

\section{Topology}
\label{sec:blocks}

As explained in \fullref{thm:acyl_calabi}, we can obtain an ACyl Calabi--Yau manifold $V=Z\setminus S$
from a compact K\"ahler  manifold $Z$ fibred over $\CP^1$ by a pencil of (generically smooth) anticanonical divisors
where $S$ is the smooth anticanonical divisor given by the fibre at $\infty$.
In this section, we collect some tools to compute basic topological invariants of $V$ and $Z$
when the complex dimension is 3.
The choice of topological invariants of $V$ and $Z$ we compute is motivated in
part by applications to the twisted connected sum construction of compact $\gtwo$--manifolds.
To compute the integral cohomology of the resulting $7$--manifolds  in~\cite{chnp2}
and in many cases also the diffeomorphism type we need sufficient
topological information about the topology of the building blocks used.

All homology and cohomology groups in this section are over $\ZZ$ unless
explicitly stated otherwise. 

\subsection{Cohomology of the ACyl manifolds}

We begin by discussing how the topology of the ACyl manifold $V=Z\setminus S$ is related to the
compact manifolds $Z$ and $S$. For convenience, we include some topological assumptions in
our definition of such building blocks (see \fullref{dfn:BLOCK}), and restrict to the case of complex
dimension 3. \fullref{prop:block_from_weak} provides a large class of 
building blocks $(Z,f,S)$ satisfying the conditions imposed in \fullref{dfn:BLOCK}.

\begin{definition}
  \label{dfn:BLOCK}
A \emph{building block} is a non-singular projective \mbox{3--fold} $Z$ together
with a projective morphism $f\co Z\to \PP^1$ satisfying the following four
assumptions: 
\begin{enumerate}[leftmargin=*]
\item The anticanonical class $-K_Z\in H^2(Z)$ is
  primitive.
\item $S=f^\star (\infty)$ is a non-singular K3 surface and $S\sim -K_Z$. 
\end{enumerate}
The fibration structure implies that $S$ has trivial normal bundle in $Z$ so
$c_1(Z)^2\sim {S\cdot S=0}$.
We denote by 
$j \co V = Z\setminus S \hookrightarrow Z$ the open embedding of the
complement and we still denote by $f \co V \to \CC$ the restricted
morphism.
Since the normal bundle of $S$ in $Z$ is trivial, there is an inclusion
$S \into V$ well-defined up to homotopy, and the restriction map
$H^m(Z) \to H^m(S)$ factors through $H^m(V)$, in the sense that it coincides
with the composition 
\begin{equation*}
 H^m(Z)\rightarrow H^m(V)\rightarrow H^m(S).
\end{equation*}
We write $H=H^2(V)$ and assume to have identified $S$ with a
``standard'' K3 and $H^2(S)$ with the K3 lattice $L$, and set
\begin{itemize}
\item
$\rho \co H \to L \quad
\text{the natural restriction map}$,
\item
$K=\ker (\rho)$,
\quad 
\text{and}
\item
$ N=\rho (H)\subset L$. 
\end{itemize}
\begin{enumerate}[resume]
\item \label{it:prim} The inclusion $N\hookrightarrow L$ is primitive, that is,
  $L/N$ is torsion-free.
\item \label{it:h3torsion}
The group $H^3(Z)$ is torsion-free.\
This implies that $H^4(Z)$ is also torsion-free.
\end{enumerate}
\end{definition}
\begin{remark*}
In the case that, as in \fullref{prop:onestage}, the building block $Z$ is obtained by blowing up the smooth (reduced) base locus $C$ 
of an anticanonical pencil on a K\"ahler 3--fold $Y$,  
then it will follow from \fullref{lem:top_of_blowup} that $H^{3}(Z)$ is torsion-free if and only if 
$H^{3}(Y)$ is. The possibility of torsion in $H^{3}(Y)$ is discussed in \fullref{R:torsion:weak:fano} when $Y$ is weak Fano.
\end{remark*}

\begin{lemma}
\label{lem:pi1block}
If $(Z,S)$ is a building block then $\pi_1(Z) = \pi_1(V) = (0)$.
\end{lemma}

\begin{remark*}
Together with assumption (iv), %
this implies that $H^*(Z)$ and $H_*(Z)$ are torsion-free.
\end{remark*}

\begin{proof}
  The critical values of the morphism $f$ are discrete in $\PP^1$.
  The statement will follow from the van Kampen theorem once we show that
  $\pi_1(V_\Delta)=(0)$, where \mbox{$V_\Delta = f^{-1}(\Delta)$} for any
  disc $\Delta \subset \PP^1$ containing at most 1 critical value $x$.
  To this end, write $\Delta^\times = \Delta \setminus \{x\}$ and
  $V_{\Delta}^\times=f^{-1}(\Delta^\times)$. Since $V_\Delta^\times$
  is a $C^\infty$ fiber bundle over $\Delta^\times$ with fibre a K3
  surface, which is simply connected, we see from the long exact
  sequence of homotopy groups in a fibration that
  $\pi_1(V_\Delta^\times)=\pi_1(\Delta^\times)=\ZZ$.
  Now let $f^\ast (x)=\sum m_iF_i$ be the fibre at $x$, where $F_i\subset V$
  are the irreducible components and $m_i$ their multiplicities.
  Condition~(i) in the definition of a building block implies 
  $\gcd (m_i)=1$. It is well known and easy to see that the natural
  homomorphism $j_\star \co \pi_1(V_\Delta^\times)\to
  \pi_1(V_\Delta)$, induced by the inclusion $j\co V_\Delta^\times
  \hookrightarrow V_\Delta$, is surjective. Examining the image of a
  loop that loops once around the generic point of $F_i$, we see that
  $m_ij_\star (1)=0$ in $\pi_1(V_\Delta)$. Since, as we noted,
  $\gcd (m_i)=1$, we conclude that $j_\star (1)=0$, that is,
  $\pi_1(V_\Delta)=0$ as was to be shown.
\end{proof}

We regard $N$ as a lattice with the quadratic form inherited from $L$ via the
primitive inclusion $N \subset L$. In examples $N$ is almost never unimodular,
thus the natural inclusion $N\hookrightarrow N^\ast$ is not an isomorphism.
We write
\[
T=\{l \in L \mid \langle l, n\rangle =0 \;\forall\; n\in N\} .
\]
$T$ stands for ``transcendental''; in examples, $N$ and $T$ are the Picard and
transcendental lattices of a lattice polarized K3 surface. Notice that, unless
$N$ is unimodular, we cannot write $L=N\oplus T$. However, since by (iii)
$N$ is primitive and $L$ is unimodular there exists a short exact
sequence
\begin{equation*} 
 0\rightarrow T\rightarrow L\rightarrow N^*\rightarrow 0,
\end{equation*}
that is, $L/T\simeq N^*$.

\begin{lemma}
  \label{lem:Z&V}
Let $(Z,f,S)$ be a building block and $V:=Z \setminus S$. Then:
\begin{enumerate}
\item \label{it:h1v} $H^1(V)=(0)$;
\item \label{it:h2v}
the class $[S]\in H^2(Z)$ fits in a split exact sequence
\[(0)\to \ZZ\overset{[S]}{\longrightarrow} H^2(Z)\to H^2(V)\to (0),\]
hence $H^2(Z)\simeq\ZZ[S]\oplus H^2(V)$ and the restriction
$H^2(Z)\to L$ maps onto~$N$;
\item \label{it:h3v}
there is a split exact sequence
\[
(0) \to H^3(Z) \to H^3(V) \to T \to (0),
\]
hence $H^3(V)\simeq H^3(Z)\oplus T$;
\item \label{it:h4v}there is a split exact sequence
\[
(0) \to N^\ast \to H^4(Z)\to H^4(V)\to (0),
\]
hence $H^4(Z)\simeq H^4(V)\oplus N^\ast$;
\item \label{it:h5v}
$H^5(V) = (0)$. 
\end{enumerate}
In particular, $H^*(V)$ is torsion-free.
\end{lemma}

\begin{proof}
We use the triangle
\[
\ZZ_S[-2] \to \ZZ_Z \to Rj_\star \ZZ_V \overset{+1}{\longrightarrow}.
\]
The associated long exact sequence is isomorphic via Poincar\'e duality to the
long exact sequence for homology of the pair $(Z, S)$.
It starts out as
\[
(0)\to H^1(Z)\to H^1(V)\to H^0(S)\hookrightarrow H^2(Z)\to \cdots
\]
We already know from \fullref{lem:pi1block} that the first two terms vanish
(i). %
The exact sequence continues with
\[
(0) \to H^0(S) = \ZZ[S] \to H^2(Z) \to H^2(V) \to (0) = H^1(S).
\]
The sequence splits because we assumed that $[S] = -K_Z$ is primitive
(ii). %
The long exact sequence continues with:
\[
(0)\to H^3(Z)\to H^3(V)\to L\to H^4(Z) \to H^4(V)\to (0).
\]
The Poincar\'e dual of $L\to H^4(Z)$ is $H_2(S)\to H_2(Z)$. This dualizes to
$H^2(Z)\to H^2(S)$, which has image $N$. Identifying $L\simeq L^\ast$, the
image of the dual map coincides with the orthogonal complement of the kernel of
$L \to H^4(Z)$; in other words, the kernel is $T$.
This implies exactness of (iii), %
and exactness of (iv) %
follows since $N^\ast \cong L/T$.
The sequence (iii) %
is split exact because~$T$, being a subgroup of $L$,
is torsion-free. The inclusion $N^* \to H^4(Z)$ is primitive because the
dual is surjective, so (iv) %
splits too. Finally, (v)
is immediate from the last piece of the long exact sequence. 
\end{proof}

\begin{remark*}
Apart from the conclusion that $H^*(V)$ is torsion-free, the proof did not use
the condition that $H^3(Z)$ is torsion-free.
\end{remark*}

As a corollary of the proof we obtain the following description.

\begin{corollary}
\label{cor:V&S^1(S)}
 Let $(Z,f,S)$ be a building block and $V:=Z\setminus S$. Since the normal
 bundle of $S$ is trivial, we get a natural
 inclusion $S\times \bbS^1\subset V$. Denote by ${\bfa^0\in H^0(\bbS^1)}$ and
$\bfa^1\in H^1(\bbS^1)$ the standard generators. 
The natural restriction homomorphisms:
\[
\beta^m\co H^m(V) \to H^m(S\times \bbS^1)=\bfa^0H^m(S)\oplus \bfa^1H^{m-1}(S)
\] 
are computed as follows:
\begin{enumerate}
\item \label{it:beta1}
$\beta^1 =0$;
\item \label{it:beta2}
  $\beta^2 \co H^2(V) \to H^2(S\times \bbS^1)=\bfa^0H^2(S)$ is the
  homomorphism $\rho \co H \to L$;
\item \label{it:beta3}
  $\beta^3\co H^3(V)\to H^3(S\times \bbS^1)=\bfa^1H^2(S)$ is
  the composition $H^3(V) \twoheadrightarrow T \subset L$;
\item \label{it:beta4}
  the natural restriction homomorphism
  $H^4(Z)\to H^4(S)=\ZZ$ is surjective and factors through
  $\beta^4\co H^4(V)\to H^4(S\times \bbS^1)=\bfa^0H^4(S)=\ZZ$ and, writing
  $K=\ker \bigl[H\to N\bigr]$ as in \fullref{dfn:BLOCK}, there is a split exact
  sequence:
\[
(0) \to K^\ast \to H^4(V)\overset{\beta^4}{\longrightarrow} H^4(S)\to (0) .
\]
\end{enumerate}
\end{corollary}

\begin{proof}
Part~(i) is trivial. %
Part~(ii) uses only that
$\text{pr}_1 \circ \beta^m \co H^m(V) \to H^m(S)$ is the natural map specified
in \fullref{dfn:BLOCK}.

For (iii), %
we use that the homomorphism $H^m(V) \to H^{m-1}(S)$ in the
long exact sequence in the proof of \fullref{lem:Z&V} is the
``Griffiths tube map'', that is, it is the composition:
\[
H^m(V) \overset{\beta^m}{\longrightarrow} H^m(S\times
\bbS^1)\overset{{\text{pr}}_2}{\longrightarrow}  \bfa^1H^{m-1}(S).
\]
To see this, first note that the boundary map
$H^m(S \times \bbS^1) \to H^{m+1}(S \times \Delta, \; S \times \bbS^1)
\cong H^{m-1}(S)$ is equivalent to $\text{pr}_2$ (where $S \times \Delta$
is a tubular neighbourhood of $S$). 
The restriction map $H^m(Z \setminus S) \to H^m(S \times \Delta^\times)$
is equivalent to $\beta^m$ while excision gives $H^{m+1}(Z, Z {\setminus} S)
\cong H^{m+1}(S {\times} \Delta, \, S {\times} \Delta^\times)$. 
Therefore $\beta^m \circ \text{pr}_2$ is the
boundary map in the long exact sequence for cohomology of
the pair $(Z, \; Z {\setminus} S)$, which is Poincar\'e dual to
the long exact sequence for homology of $(Z, \, S)$.

The content of (iv) %
is to show that $\beta^4$ is surjective and to
determine its kernel. $\beta^4$ fits into the long exact sequence for
cohomology of $V$ relative to its ``boundary'' $S \times \Sph^1$, and
surjectivity follows from $H^5(V, S \times \Sph^1) \cong H_1(V) = 0$.
The cup product gives a perfect pairing between the free parts of
$\ker \beta^m$ and $\ker \beta^{6-m}$ for any $m$, so in particular
$\ker \beta^4 \cong (\ker \beta^2)^* = K^*$.
\end{proof}

\begin{remark}
We can also compute the homology groups of $V$ as follows:
\begin{enumerate}
\item $H_1(V) = 0$;
\item $0 \to H_2(V) \to H_2(Z) \to \ZZ \to 0$ split exact;
\item $0 \to T^* \to H_3(V) \to H_3(Z) \to 0$ split exact;
\item $H_4(V) \cong K$;
\item $H_5(V) = 0$.
\end{enumerate}
\end{remark}

\subsection{Building blocks from semi-Fano 3--folds}
\label{sub:blocks_from_weak}

We now study the topology of the closed 3--folds $Z$ produced in \fullref{prop:onestage} by blowing up the base locus of a generic anticanonical
pencil on a weak Fano 3--fold $Y$.

We will use the following simple lemma in numerous places in the rest of the paper.
\begin{lemma}
  \label{lem:top_of_blowup} 
  Let $\blow \co (E\subset Z)\to (C \subset Y)$ be the
  blow up of a non-singular curve in a non-singular \mbox{3--fold}~$Y$. Then
  $H^m(Z)\simeq H^m(Y)\oplus H^{m-2}(C)$.
\end{lemma}

\begin{proof}
  The decomposition theorem holds over $\ZZ$:
\[
R \blow_\star \ZZ_Z\simeq \ZZ_Y \oplus \ZZ_C[-2] ;
\]
hence $H^m(Z)\simeq H^m(Y)\oplus H^{m-2}(C)$.
\end{proof}

The following result proves that we can always obtain a building block (in
the sense of \fullref{dfn:BLOCK}) 
by blowing up the base locus of a generic AC pencil (provided that this is smooth) on a semi-Fano 3--fold 
with torsion-free $H^3$; see also \fullref{R:torsion:weak:fano} for comments on the torsion-free assumption.
We use the same notation as in \fullref{dfn:BLOCK}.
\begin{prop}
\label{prop:block_from_weak}
Let $Y$ be a weak Fano 3--fold, $C$ the smooth base locus of a
generic pencil in $\acls{Y}$, $Z$ the blow-up of $Y$ at $C$, and
$f \co Z \to \CP^1$ the fibration induced by the pencil.
\begin{enumerate}
\item The anticanonical class $-K_Z\in H^2(Z)$ is primitive.
\item Some fibre $S=f^\ast(\infty)$ is a non-singular K3 surface and
$S\sim -K_Z$.
\item The restriction maps from $H^2(Y)$, $H^2(Z)$ and $H^2(V)$ to $H^2(S) = L$
have the same image $N$. If $Y$ is semi-Fano then $K = 0$ (that is,
$H^2(V)\rightarrow H^2(S)$ is injective) and the inclusion
$N\hookrightarrow L$ is primitive.
\item The group $H^3(Z)$ is torsion-free if and only if $H^3(Y)$ is.
\end{enumerate}
Furthermore, $\pi_1(Z) = (0)$.
\end{prop}

\begin{proof}
(i) and (ii) follow from the well-known formula $-K_Z=\blow^*(-K_Y)-E$
and \fullref{thm:reid}.
(iv) follows from \fullref{lem:top_of_blowup}.
(i) and (ii) are the only hypotheses used in the proof of
\fullref{lem:Z&V}(ii), %
which entails that $H^2(V)$ and $H^2(Z)$ have the
same image in $L$. The image of the class in $H^2(Z)$ of the
exceptional divisor is $[C] = -K_{Y|S} \in L$, so $H^2(Y)$ has the
same image too.

To complete the proof of (iii) we need the following fact: if $Y$ is a
semi-Fano 3--fold and $-K_Y \sim S\subset Y$ is a non-singular K3 surface then
$H^2(Y) \to H^2(S)$ is a primitive inclusion. This follows from \fullref{prop:lef}. 

It was proved in \fullref{lem:pi1block} that (i) and (ii) imply
$\pi_1(Z) = (0)$. This can also be deduced from some standard facts
about weak Fano 3--folds.  Recall that an algebraic variety $Y$ is
\emph{rationally connected} if given any two general points
$y_1,y_2\in Y$, there exists a morphism $f\co \PP^1 \to Y$ such
that $y_1$ and $y_2$ are both in the image of $f$. Campana
\cite[Theorem~3.5]{campana} proved that
rationally connected varieties are simply connected and 
Koll\'ar--Miyaoka--Mori~\cite[Corollary~3.11]{MR1158625} established 
that any smooth weak Fano 3--fold is rationally
connected.  Simple-connectivity of $Z$ now
follows from the fact that $\pi_1(Z) \cong \pi_1(Y)$ for the blow-up
of $Y$ in a smooth curve.
\end{proof}

\begin{remark}\hfill
\label{R:torsion:weak:fano}
\begin{enumerate}
\item The torsion subgroup $T_{2} \subset H^{3}(Y)$ is a birational
  invariant of a non-singular projective variety $Y$ of any
  dimension $n$. In particular, $T_{2}=0$ if $Y$ is rational, that is, $Y$
  is birational to $\CP^{n}$.  Rationality of Fano 3--folds (including those 
  with mild singularities) is somewhat subtle and still an area of current research activity.
\item It follows from the classification of non-singular Fano 3--folds that
  there is no torsion in $H^{3}(Y)$ for any non-singular Fano 3--fold; we
  are not aware of any conceptual proof of this fact.
\item There is a well-known example due to Artin--Mumford of a singular
  Fano \mbox{3--fold} with torsion in $H^{3}$~\cite{artin:mumford}. The
  Artin--Mumford example is a nodal quartic double solid with $10$
  nodes. Torsion in nodal double solids has been studied more
  systematically by Endra\ss~\cite{endrass:double:solids}.  Nodal quartic double solids can have
  any number of nodes $e$ between $1$ and $16$; Endra\ss \ showed that
  for nodal double quartics non-zero torsion $T_{2}$ can occur only
  when $e=10$~\cite[Theorem~3.6]{endrass:double:solids} (as in the
  Artin--Mumford example).  Very recently, a nodal double sextic solid
  with $35$ nodes and non-zero torsion was constructed (see
Iliev--Katzarkov--Przyjalkowski \mbox{\cite[Proposition~3.1]{iliev}}).
However, since these examples do not admit
  any projective small resolution they do not give rise to weak Fano
  3--folds with torsion in~$H^{3}$.
\item We can prove that $H^3(Y)$ is torsion-free
for many semi-Fano 3--folds, in which case by \fullref{prop:block_from_weak} $Y$
gives rise to a building block in the sense of \fullref{dfn:BLOCK}.
In fact  we do not currently know of any weak Fano 3--fold $Y$ with
  torsion in $H^{3}(Y)$. 
\end{enumerate}
\end{remark}

For a $Z$ obtained by a sequential blow-up as in \fullref{prop:sequence},
we replace part~(iii) of \fullref{prop:block_from_weak} by a more
limited claim. 

\begin{prop}
Let $Y$ be a weak Fano 3--fold, $C$ the base locus of a pencil in $\acls{Y}$
such that each irreducible component $C_1, \ldots, C_k \subseteq C$ is smooth,
and $Z$ the blow-up of $Y$ at the $C_i$ in sequence.
Then $Z$ satisfies the conclusions \textup{(i), (ii)} and \textup{(iv)}
of \fullref{prop:block_from_weak}, and $\pi_1(Z) = (0)$.

$H^2(Z)$ and $H^2(V)$ have the same image $N$ in $L$. Let $K_0$ and $N_0$ be
the kernel and image of $H^2(Y) \to L$.
Then $\rank N/N_0 + \rank K -\rank K_0 = k-1$.
\end{prop}

\begin{proof}
Parts~(i), (ii), and~(iv) are proved by repeated application of the proof of
\fullref{prop:block_from_weak}. Like there, it then follows that
$\pi_1(Z) = (0)$ and that $H^2(Z)$ and $H^2(V)$ have equal image.
The final claim is simply an application of rank-nullity, using
$b^2(V) = b^2(Z) - 1 = b^2(Y) + k-1$.
\end{proof}

Examples~\ref{exa:P3_deg},~\ref{ex:2conics} and~\ref{exa:toricV22_b} consider 
in detail building blocks obtained from nongeneric AC pencils on Fano or semi-Fano 3--folds.
In the terminology of the previous proposition 
$N = N_0$ in Examples~\ref{exa:P3_deg} and~\ref{exa:toricV22_b}. There the
anticanonical pencil is spanned by a generic K3 $S$ and a sum of smooth
divisors $D_i$ intersecting $S$ transversely: then the image of the exceptional
divisor in $H^2(Z)$ over $C_i = D_i \cap S$ is $[C_i] \in H^2(S)$, which is
already the image of $[D_i] \in H^2(Y)$.
On the other hand, if \emph{all} the anticanonical divisors
in the pencil are non-generic then $H^2(Z)$ can have bigger image than
$H^2(Y)$. If, like in \fullref{ex:2conics}, we take a pencil of
anticanonical divisors containing a special curve~$C_1$, then $[C_1] \in H^2(S)$
will be contained in the image of $H^2(Z)$, but not in the image of~$H^2(Y)$.

\subsection{Chern classes}
As in the previous subsection let $f\co Z \to \CP^1$ be a building block in the sense of \fullref{dfn:BLOCK}, $S$ be a smooth fibre of the morphism $f$ and $V=Z\setminus S$.
Let us first consider briefly the characteristic class $c_2(V) \in H^4(V)$.
When $K = 0$ (which by \fullref{prop:block_from_weak} always holds when $Z$ is 
a building block obtained from a generic AC pencil on a semi-Fano 3--fold with torsion-free $H^3$), 
\fullref{cor:V&S^1(S)}(iv) implies that
$H^4(V) \cong H^4(S) \cong \ZZ$, and $c_2(V)$ is completely determined by the
fact that the restriction of $c_2(V)$ to $S$ is $c_2(S) \cong \chi(S) = 24$.

More generally, $c_2(V)$ is the restriction of $c_2(Z)$. By
\fullref{lem:Z&V}(iv) %
the restriction map $H^4(Z) \to H^4(V)$ has
non-trivial kernel $N^*$, so $c_2(Z)$ contains strictly more information than
$c_2(V)$. In the twisted connected sum construction of compact \gtmfd s in~\cite{chnp2},  it turns out that in order to determine the characteristic class
$p_1(M)$ for the resulting smooth $7$--manifold one needs to understand $c_2(Z)$
of the building blocks. In order to apply classification results for smooth
$2$--connected $7$--manifolds there we are mainly concerned with determining the
greatest divisors of these classes,
see \mbox{\cite[Corollary~4.32]{chnp2}}.
We begin by observing:

\begin{lemma}
\label{lem:c2even}
$c_2(Z) \in H^4(Z)$ is even for any building block $Z$.
\end{lemma}

\begin{proof}
Recall from \fullref{dfn:BLOCK} following assumption (ii), that $c_1(Z)^2=0$ for any building block $Z$. 
Consider the short exact sequence 
\begin{equation*}
 0 \rightarrow \bbz \rightarrow \bbz \rightarrow \bbz/2\bbz \rightarrow 0
\end{equation*}
and the induced ``mod 2'' maps $H^{i}(Z;\ZZ)\rightarrow H^{i}(Z;\bbz/2\bbz)$.
To prove the lemma it suffices to prove that, for any complex 3--fold,
$c_2(Z) \mod 2 = c_1(Z)^2 \mod 2$. The proof requires several facts about 
characteristic classes for which we refer to the book by Milnor and Stasheff~\cite{MS}.

Let $w_i(Z)\in H^i(Z;\bbz/2\bbz)$ denote the Stiefel--Whitney classes of $Z$.
According to \cite[Theorem~11.14]{MS}, the classes $w_i(Z)$ can be written in terms of
Steenrod squares of the Wu classes $v_i(Z)$ in terms of the equation
$w_k(Z)=\sum_{i+j=k}Sq^i(v_j(Z))$. As in~\cite[page~171]{MS}, $w_{2i}(Z)=c_i(Z)
\mod 2$ and $w_{2i+1}(Z)=0$. Using the basic properties of Steenrod
squares, compare with~\cite[pages~90ff]{MS}, it follows that $v_1(Z) = 0$ and
$v_2(Z) = w_2(Z)$.
Since $w_2(Z)=c_1(Z) \mod 2$, the image of this class under the Bockstein
operator $\delta: H^2(Z;\bbz/2\bbz) \rightarrow H^3(Z;\ZZ)$ vanishes. 
On the other hand it is known, compare with Steenrod--Epstein~\cite[page~2]{steenrod65},
that $Sq^1$ is the Bockstein operator $H^2(Z;\bbz/2\bbz) \rightarrow
H^3(Z;\bbz/2\bbz)$ defined by the short exact sequence
\begin{equation*}
 0 \rightarrow \bbz/2\bbz \rightarrow \bbz/4\bbz \rightarrow \bbz/2\bbz
    \rightarrow 0.
\end{equation*}
Relating the above two sequences of coefficients via the obvious ``mod'' maps
proves that $Sq^1(w_2(Z)) = \delta(w_2(Z)) \mod 2 = 0$.
It follows that $v_3(Z) = 0$. Also $v_4(Z) = 0$ because Wu classes in degree
greater than half the dimension of the manifold vanish, compare with~\cite[page~132]{MS}.
Hence $w_4(Z) = Sq^2(v_2(Z)) = w_2(Z)^2$.
\end{proof}

The remainder of the section is devoted to providing tools to compute $c_2(Z)$
for the examples of building blocks in this paper.

\begin{prop}
\label{prop:c2blowup}
Let $Y$ be a compact complex 3--fold, $S \subset Y$ a smooth anticanonical
divisor and let $\blow \co Z \to Y$ be obtained by blowing up, in sequence,
smooth curves $C_1, \ldots, C_k \subset S$, such that $-K_{Y|S} = [C_1] +
\cdots + [C_k]$. Then
\begin{equation}
\label{eq:c2blowup}
c_2(Z) = \blow^*(c_2(Y) + c_1(Y)^2) + K_Y^3 [\CP^1_1] +
\sum_{i = 1}^{k-1} K_{Z_i}^3([\CP^1_{i+1}] - [\CP^1_i]) \in H^4(Z) ,
\end{equation}
where $\CP^1_i$ is a fibre in the exceptional set over $C_i$ and
$Z_i$ are the intermediate blow-ups if $k > 1$.
\end{prop} 

\begin{proof}
Since $c_1(Z)^2 = 0$,
the result follows from the following claim by induction.
Let $Y$ be a compact complex 3--fold, $S$ a non-singular anticanonical divisor,
and $\blow \co Z \to Y$ the blow-up of $Y$ at a non-singular curve
$C \subset S$. Then
\begin{equation}
\label{eq:c2step}
c_2(Z) + c_1(Z)^2 = \blow^*(c_2(Y) + c_1(Y)^2) + (-K_Z^3 + K_Y^3)[\CP^1]
\in H^4(Z) ,
\end{equation}
where $[\CP^1]$ is a fibre in the exceptional set $E$ over $C$.

For $d := c_2(Z) + c_1(Z)^2 - \blow^*(c_2(Y) + c_1(Y)^2)$ is 0 away from $E$, so
it is Poincar\'e dual to some class in $H_2(E)$. To show that $d$ is a multiple
of $[\CP^1]$ it suffices to show that $d \cdot (E + \tilde S) = 0$, where
$\tilde S$ is the proper transform of $S$. The line bundle corresponding
to $E + \tilde S$ is $\blow^*(-K_Y)$.

Because $\blow \co Z \to Y$ is a blow-up in a smooth curve, $H^k(Y,L)
\cong H^k(Z,\blow^*L)$ for any line bundle $L$ on $Y$ and any
$k$.\footnote{The projection formula states $R\pi_\star ( \pi^\star L) = L
  \otimes R\pi_\star \mathcal{O}_Z$; here $R\pi_\star
  \mathcal{O}_Z=\mathcal{O}_Y$; so, finally $H^k\bigl(Z, \pi^\star
  (L)\bigr)=H^k\bigl(Y, R\pi_\star (\pi^\star L)\bigr)=H^k\bigl(Y, L\bigr)$.}
In particular $\chi(Z, \blow^*L) = \chi(Y, L)$.  Applying this and the
Riemann--Roch formula \eqref{eq:rr} when $L$ is trivial gives that
$c_1(Z)c_2(Z) = c_1(Y)c_2(Y)$. Applying it to $L = -K_Y$ gives
\[ \blow^*(-K_Y)(c_2(Z) + c_1(Z)^2) = -K_Y(c_2(Y) + c_1(Y)^2) \]
because all the other terms in the Riemann--Roch formula agree.
The right-hand side equals $\blow^*(-K_Y)\blow^*(c_2(Y) + c_1(Y)^2)$,
so $d \cdot \blow^*(-K_Y) = 0$ as required.

To pin down the coefficient of $[\CP^1]$ we just evaluate on $-K_Z$.
\end{proof}

\begin{remark*}
Fulton~\cite[Theorem~15.4]{fulton} gives a general formula for the difference
between the Chern class of a blow-up and the pull-back of the Chern class of
the base. In addition, Fulton's Example 15.4.3 helpfully distills the formula for
blow-ups in a smooth codimension two variety. In our notation, it states
\[ c_2(Z) - \blow^*c_2(Y) = \blow^*[C] - (\blow^*c_1(Y))[E]  \]
for a single step blow-up. We can use this to prove \eqref{eq:c2step}; however
the ad hoc proof in terms of Riemann--Roch makes it easier to identify the
terms we want.
\end{remark*}

Recall that the index of a weak Fano $Y$ is $r = \gdiv(c_1(Y))$, the greatest
divisor of $c_1(Y)$.

\begin{corollary}
\label{cor:c2blowup}
Let $Z$ be a building block obtained by blowing up the smooth base locus of an
AC pencil on a semi-Fano 3--fold $Y$ with torsion-free $H^3$ as in
\fullref{prop:block_from_weak}. Then
\[ 2 \mid \gdiv(c_2(Z)) \mid \gcd\big(\tsfrac{1}{r}(24+K_Y^3), 24\big), \]
with equality on the right if $k = 1$ and $Y$ has Picard rank 1.
In particular, if $r = 1$ and $K_Y^3$ is not divisible by $3$ or $4$ then
$\gdiv(c_2(Z)) = 2$.
\end{corollary}

\begin{proof}
$c_2(Z)$ is even for any building block $Z$ according to
\fullref{lem:c2even}.
\[ \gdiv\big(c_2(Y) + c_1(Y)^2\big) \mid
\tfrac{1}{r}\big(c_2(Y)c_1(Y) + c_1(Y)^3\big) = \frac{24+K_Y^3}{r} \]
since $\frac{1}{r}c_1(Y)$ is integral. If $Y$ has Picard rank 1 then $Y$ is
Fano so $H^3(Y)$ is torsion-free. It follows that $\frac{1}{r}c_1(Y)$
spans $H^2(Y)$ and $H^4(Y)\times H^2(Y)\to\ZZ$ is a perfect pairing,
leading to equality. Thus, using \fullref{prop:c2blowup},
\[ \gdiv(c_2(Z))\mid\gcd\big(c_2(Y) + c_1(Y)^2, K_Y^3\big) \mid
\gcd\big(\tsfrac{1}{r}(24+K_Y^3), K_Y^3\big)
\]
with equality on the left if $k=1$ and equality on the right if $Y$ has Picard
rank 1.
\end{proof}

If the Picard rank of the weak Fano $Y$ is not 1 and the corollary does not
force $\gdiv c_2(Z) = 2$, then we can compute it by applying the lemma below to
non-singular divisors that together with $-K_Y$ form a basis for $H_4(Y)$.
In the examples considered in \fullref{sec:examples} there are obvious choices of such divisors.

\begin{lemma}
\label{lem:c2div}
Let $D$ be a non-singular divisor in a non-singular complex 3--fold $Y$. Then
\begin{equation}
\label{eq:c2div}
(c_2(Y) + c_1(Y)^2)_{|D} = \int_D (c_2(D) - c_1(D)^2) + D(-D-2K_Y)(-K_Y) .
\end{equation}
\end{lemma}

\begin{proof}
$TY_{|D} = TD + D_{|D}$ implies that $c_2(Y)_{|D} = c_2(D) + c_1(D)D_{|D}$ and
$c_1(D) = (-D - K_Y)_{|D}$. %
\end{proof}

\begin{remark}
\label{rmk:c2:flop}
In examples where $Y$ is a small resolution of a singular 3--fold $X$, taking a
different small resolution $Y^+$ (which is therefore related to $Y$ by a finite sequence of flops) leaves the quadratic form on the
Picard lattice unchanged, and hence also the last term of \eqref{eq:c2div}.
However, the divisors are transformed birationally, which changes the first
term of the RHS of \eqref{eq:c2div}. More precisely, observe that if $D^+$ is
the proper transform in $Y^+$ of $D$, then
\[ c_1(D)^2 - c_1(D^+)^2 = D(D + K_Y)^2 - {D^+(D^+ + K_{Y^+})^2} =
D^3 -(D^+)^3 . \]
Since $c_2(D) + c_1(D)^2$ is a birational invariant, it follows that
$c_2(Y)_{|D} - c_2(Y^+)_{|D^+} = -2(D^3 - (D^+)^3)$.
Hence it is easy to understand the change in $\gdiv c_2(Z)$ under a flop
if we know the difference of the intersection numbers $D^3 - (D^+)^3$.

In the case when $Z$ is obtained by blowing up a reducible base locus of an AC pencil,
changing the order of the components also corresponds to a flop, and
according to \eqref{eq:c2blowup} $\gdiv c_2(Z)$ may depend on $K_{Z_i}^3$
of the intermediate blow-ups.
\end{remark}

In the fundamental case when the block $Z$ is constructed from a weak Fano by a
single blow-up, viewing $c_2(Z)$ in a different basis can make it easier to
determine the greatest divisor of $c_2(Z)$ modulo certain subgroups of
$H^4(Z)$; this is useful in certain applications of~\cite[Corollary~4.32]{chnp2} to compute characteristic classes of
twisted connected sum \gtmfd s.
\eqref{eq:c2blowup} expresses $c_2(Z)$ in terms of the decomposition
$H^4(Z) = \blow^*H^4(Y) \oplus \ZZ [\CP^1]$, but we could also use the
exactness of the sequence
\[ 0 = H^3(S) \to H^4_{cpt}(V) \to H^4(Z) \to H^4(S) \to H^5_{cpt}(V) = 0 \]
to write $H^4(Z) = H^4_{cpt}(V) \oplus \ZZ [\CP^1]$. Composing
$\blow^* \co H^4(Y) \to H^4(Z)$ with the projection $H^4(Z) \to H^4_{cpt}(V)$
gives an isomorphism $g \co H^4(Y) \to H^4_{cpt}(V)$. Now \eqref{eq:c2blowup}
implies $c_2(Z) = g(c_2(Y) + c_1(Y)^2) + 24[\CP^1]$.
(This gives another way to phrase the proof of \fullref{cor:c2blowup}.)

There are natural maps $i_Z \co L  \cong H_2(S) \to H_2(Z) \cong H^4(Z)$
and $i_Y \co  L \to H^4(Y)$. Because $S$ has trivial self-intersection,
the image of $i_Z$ is contained in $H^4_{cpt}(V)$, and $g \circ i_Y = i_Z$.
Hence

\begin{lemma}
\label{lem:c2mod}
For any subgroup $N' \subset L$,
\[
\gdiv (c_2(Z) \mod i_Z(N')) = \gcd(c_2(Y) + c_1(Y)^2 \mod i_Y(N'), \; 24) .
\]
\end{lemma}

In turn, $\gdiv(c_2(Y) + c_1(Y)^2 \mod i_Y(N'))$ can be computed by using
\fullref{lem:c2div} to evaluate $c_2(Y) + c_1(Y)^2$ on elements
$D \in H^2(Y)$ such that $i_Y^*(D)$ is perpendicular to~$N'$.

\section{Anticanonical divisors in semi-Fano 3--folds}
\label{sec:div}

Almost any non-singular weak Fano 3--fold $Y$ -- recall both the standing Assumption 
preceding \fullref{prop:onestage} and \fullref{R:weak:fano:ac:model}(iv) -- can be blown up as in
\fullref{prop:onestage} to obtain a projective 3--fold
$Z$ with an anticanonical K3 divisor $S$, such that \fullref{thm:acyl_calabi}
produces ACyl Calabi--Yau structures on $V=Z\setminus S$. The
asymptotic limit is the product of the cylinder $\bbrp \times \Sph^1$
and $S$ as in \eqref{eq:cycyl}. $S$ is equipped with a non-vanishing
holomorphic 2--form $\Omega_S$ and a Kähler form $\omega_S$, and we can regard
the pair $(\Omega_S, \omega_S)$ as a \hk structure.  We now wish to understand
better \emph{which} \hk structures on K3 occur as the asymptotic limits of
ACyl Calabi--Yau \mbox{3--folds} constructed this way from a given family of weak
Fanos. In \fullref{prop:fanok3} we show that when $Y$ is a semi-Fano 3--fold then the subset of asymptotic
limit hyper-Kähler structures on K3 is ``large'' (as characterised by the de Rham
cohomology classes of the 2--forms) in the space of adapted hyper-Kähler structures, 
that is, those satisfying the a priori necessary polarisation condition described below.

In view of \fullref{thm:acyl_calabi}, we are interested in which complex
2--forms on K3 are holomorphic with respect to some smooth embedding of K3 as
an anticanonical divisor in an element $Y$ of a family of weak Fano 3--folds,
and which real 2--forms on K3 are restrictions of Kähler forms on $Y$.
We are therefore led to study the deformation theory of pairs $(Y,S)$ where
$Y$ is a weak Fano 3--fold and $S$ is a non-singular anticanonical divisor.

If $Y$ is a weak Fano 3--fold then by \fullref{thm:reid} a generic 
$S \in \acls{Y}$ is a smooth K3 surface. If moreover $Y$ is semi-Fano then 
by \fullref{lem:lef} below, the natural restriction homomorphism
$\Pic{Y} \to \Pic{S}$ is a primitive embedding. 
This implies that the K3 surfaces appearing as anticanonical 
divisors in a given (deformation class of) semi-Fano 3--fold are very special; 
we see only those K3 surfaces $S$ that contain a primitive sublattice
$\Pic{Y} \subset \Pic{S}$. Such K3 surfaces are called
\emph{lattice polarised K3 surfaces} and the moduli theory of lattice polarised
K3 surfaces is well understood; we recall it below.

In order %
to understand when a given lattice polarised K3 surface $S$
appears as a smooth anticanonical divisor in a given deformation class of
semi-Fano 3--folds, %
we also need to construct a sufficiently
well-behaved moduli space (stack) parameterising pairs consisting of 
a deformation class of semi-Fano 3--folds $Y$ and the choice of a smooth 
anticanonical section $S \in \acls{Y}$: see
\fullref{dfn:semifano_moduli_stack}.
The semi-Fano  assumption on $Y$ is used
in our proof  that the appropriate moduli stack parameterising such pairs
is a smooth stack: see \fullref{thm:fano_stack_smooth}. The smoothness
proof we give relies on the fact that semi-Fano 3--folds satisfy slightly better
vanishing theorems (\fullref{thm:aknava}) than the standard
Kawamata--Viehweg vanishing (\fullref{T:KV:vanish}). 
(However, see also the Remark following \fullref{thm:fano_stack_smooth}
for the general weak Fano case).
Most importantly of all, we need to understand the forgetful map $(Y,S) \mapsto
S$ from such pairs $(Y,S)$ to the moduli (stack) of lattice polarised K3
surfaces: see \fullref{thm:fanok3} for the statement of such a result.

\fullref{thm:fanok3} is a crucial ingredient in arguments in our
paper~\cite[Section~6]{chnp2} that allows us in many cases to solve the
so-called ``matching problem'' for a pair of \hk K3 surfaces and therefore
construct many compact $7$--manifolds with holonomy group \gtwo using the
twisted connected sum construction.

We now describe the relevant moduli theory first for lattice polarised K3
surfaces, secondly for pairs of semi-Fano 3--folds and smooth anticanonical
sections and then study the natural map between these two moduli spaces
(stacks).

\subsection{Lattice polarised K3 surfaces and the Torelli theorem}

We recall some standard facts about moduli of lattice polarised K3
surfaces. Our purpose is to fix notation and recall just the facts
that we need, not to give an introduction to moduli of~K3. The
constructions here are described in greater detail for example in
Dolgachev~\cite[Sections~1 and~3]{dol:k3}.

We denote by $L = 2E_8(-1) \perp 3U$ an abstract copy of the K3
lattice.  Fix a triple $(N, A, j)$ of a lattice $N$ of signature
$(1,\rho)$ ($\rho = 0$ is allowed), an element $A\in N$ with
$A^2=2g-2>0$, and a primitive lattice embedding $j\co N
\hookrightarrow L$. (In general, there may be several inequivalent
such embeddings.)

Write $\Delta=\{\delta \in N \mid \delta^2=-2\}$. As we will see
shortly, to specify the moduli space of $N$--polarised K3 we need to
choose a partition $\Delta=\Delta^+ \sqcup \Delta^-$ satisfying the
properties:
    \begin{enumerate}
    \item $\Delta^-=\{-\delta \mid \delta \in \Delta^+\}$,
    \item if $\delta_1, \ldots, \delta_k \in \Delta^+$ and $\delta = \sum
    \lambda_i\delta_i\in \Delta$ with all $\lambda_i\geq 0$ then also $\delta \in \Delta^+$, and
    \item \label{it:apos} $A\geq 0$ on $\Delta^+$.
    \end{enumerate}
In what follows, we always (implicitly) assume that such a choice has been made.

\begin{remark} \label{rem:choice_k3weilchamber} In general, it is not
  easy to make explicit the choice of a partition of $\Delta=\Delta^+
  \sqcup \Delta^-$ as just discussed.  In almost all cases of interest
  to us, it will be possible to verify that for all $\delta \in \Delta$,
  $A\cdot \delta \neq 0$. When this is the case, property (iii) %
  specifies that $\Delta^+=\{\delta \in \Delta \mid A\cdot \delta >0 \}$.
\end{remark}

Let $V^+\subset N_\RR$ be the connected component of the cone
$V=\{\xi\mid \xi^2>0\} \subset N_\RR$ containing~$A$, and write
\[
C^+=\{\xi \in V^+ \mid \xi \cdot \delta > 0 \; \text{for all} \;
\delta \in \Delta^+\} .
\]

\begin{definition} \label{dfn:K3moduli} The stack $\mathfrak{K}^{N,A}$
  of $(N,A, j)$--polarised K3 surfaces (we often just say $N$--polarised
  K3 surfaces) is the category whose objects are: families
  $f\co \df{S} \to B$ of (non-singular) K3 surfaces, together with an isometry
\[
N\hookrightarrow \uPic (\df{S}/B)\subset \mathbb{L}=R^2f_\star
\ZZ_{\df{S}}
\]
(where $\uPic (\df{S}/B)$ is the relative Picard group
functor\footnote{That is, the sheaf on $B$ for the faithfully flat
  topology associated to the presheaf $(B^\prime/B)\mapsto \Pic
  (\df{S}\times_B B^\prime)$. This sheaf is representable by a group
  scheme over $B$ that we also denote by $\uPic (\df{S}/B)$.}) such that: 
\begin{enumerate}
\item for every $b\in B$, the embedding $N\subset
\mathbb{L}_b$ is equivalent to $j\co N \subset L$,
\item $C^+ \cap \Amp (\df{S}/B)\neq \emptyset$. 
\end{enumerate}
Morphisms in the category are Cartesian diagrams.    
\end{definition}

It is well known that $\mathfrak{K}^{N,A}$ is a smooth and connected
Deligne--Mumford stack with quasi-projective coarse moduli space that
we denote by $\mathcal{K}^{N,A}$. (Our only reason for working with
stacks, and not with spaces, is that, because of smoothness, we can
use certain infinitesimal arguments below. The reader who wishes to do
so can pretend that the stack is in fact a smooth space, even though
this is not, strictly speaking, true.)

Next we summarise the construction of the coarse moduli space from
Hodge theory. 

\begin{itemize}
\item
Denote by $D$ the Griffiths domain of oriented positive real $2$--dimensional
vector subspaces $\Pi\subset N_{\R}^\perp \subset L_\R$.
Recall that giving $\Pi$ is equivalent to giving a polarised Hodge
structure on~$L$:
\[
L\otimes \C = H^{2,0}\oplus H^{1,1}\oplus H^{0,2}
\]
where $\Pi\otimes \C=H^{2,0}\oplus H^{0,2}$ is the complex structure
on $\Pi$ where multiplication by $\sqrt{-1}$ is achieved by the
(positive) rotation by $90$ degrees. Giving the real $2$--plane
$\Pi\subset L \otimes \R$ is equivalent to giving the complex line
$H^{2,0}\in L\otimes \C$. The \emph{period point} of a K3 surface $S$
is the plane $\Pi(S)$ corresponding to $H^{2,0}=H^{2,0}(S)$.

\item The stack $\mathfrak{M}^{N,A}$ of \emph{marked} $(N,A, j)$--polarised
  K3 surfaces is the category whose objects are: objects
  $f\co \df{S} \to B$ of $\mathfrak{K}^{N,A}$, together with an isometry
  (marking)
\[
h \co R^2f_\star \ZZ_{\df{S}} \to L\otimes \ZZ_B
\]
such that: (i) for every $b\in B$, the composition $N\subset
\mathbb{L}_b = H^2(S_b;\ZZ)\overset{h_b}{\to} L$ is $j\co N
\subset L$, and (ii) $C^+ \cap \Amp (\df{S}/B)\neq \emptyset$. 
Morphisms in the category are Cartesian diagrams.    

If $\delta \in N^\perp \subset L$ is a class with $\delta^2=-2$, we
denote by $H_\delta\subset D$ the hypersurface consisting of
$\Pi\subset \delta^\perp$. Taking the union over all such $\delta$, we
write
\begin{equation}
\label{eq:griff}
D^0= D \setminus \bigcup H_\delta \subset D .
\end{equation}

The \emph{period map} is the morphism $\Pi \co \mathfrak{M}^{N,A}
\to D^0$ that maps $f\co \df{S}\to B$ to the polarised variation of
Hodge structure on $R^2f_\star \ZZ_{\df{S}}$ -- that is, it maps the
surface $S$ to its period point $\Pi(S)$.  A variant of the Torelli
theorem for K3 surfaces (see Dolgachev~\cite[Corollary~3.2]{dol:k3}) states that $\Pi$ is an
isomorphism.

\item It follows from the Torelli theorem just stated that
\[
\mathfrak{K}^{N,A}=[D^0/\Gamma]
\]
as stacks, and $\mathcal{K}^{N,A}=D^0/\Gamma$ as spaces, where
$\Gamma\subset O(L)$ is a discrete group acting properly and
discontinuously on $D^0$ -- see Dolgachev~\cite{dol:quadratic,dol:k3} for
details. For our purposes, we only need to know that $\Gamma$ is
commensurable with the set of isometries of $L$ which restrict to the
identity on~$N$. We do not need the precise description of~$\Gamma$.
\end{itemize}

\subsection{Semi-Fano \mbox{3--folds} and their K3 sections}
\label{sec:weak-ast-fano}

We now come to the key purpose of this section, which is to extend
some of the notions and results of Beauville~\cite{beauville:fano} to the case
of semi-Fano \mbox{3--folds}.

\begin{lemma}
  \label{lem:lef}
  Let $Y$ be a non-singular semi-Fano \mbox{3--fold}, and let
  $S\in \acls{Y}$ be a non-singular surface, necessarily a K3 surface.
  Then, the natural restriction homomorphism $\Pic Y \to \Pic S$ is
  a primitive embedding.
\end{lemma}
\begin{proof}
  This was already shown in the proof of
  \fullref{prop:block_from_weak}. 
\end{proof}

\begin{remark*}
  By \fullref{t:wk:fano:finite} we know that the Picard rank $\rho
  < c$ for all weak Fano 3--folds, but in general we have no estimate
  of $c$.  \fullref{lem:lef} implies that $\rho(Y):=
  \rank{\Pic{Y}} \le 20$ for any non-singular semi-Fano 3--fold $Y$;
  \fullref{exa:burkhardt_quartic} gives a semi-Fano 3--fold that
  has Picard rank $\rho=16$ and a nodal AC model.
  
  For a general non-singular weak Fano
  3--fold this upper bound of $20$ on the Picard rank is false; according to
  \fullref{r:picard:rk} there exists a toric weak Fano
  3--fold with $\rho(Y)=35$. This is the maximal Picard rank that
  occurs for toric weak Fano 3--folds and is currently the largest
  Picard rank known for any non-singular weak Fano 3--fold. 
\end{remark*}

If $Y$ is a semi-Fano \mbox{3--fold}, we regard $\Pic Y\cong
H^2(Y,\Z)$ as a lattice by means of the quadratic form
$(D_1,D_2)\mapsto -D_1\cdot D_2\cdot K_Y$; this lattice has the
distinguished element $A=-K_Y$ with $A^2=2g-2$. (Note that $A$ is
\emph{not} a K\"ahler class on $Y$ when $Y$ is semi-Fano but not Fano.) 

\begin{definition}\label{dfn:semifano_moduli_stack}
Fix now a lattice $N$ of signature $(1,\rho)$, with a distinguished element $A$ with
$A^2=2g-2$. We also fix an embedding $j\co N \subset L$ in the K3
lattice. The stack $\mathfrak{F}^{N,A}$ is the category whose
objects are families $f\co (\df{S}, \df{Y})\to B$, such that:
\begin{enumerate}
\item[(i)] for every geometric point $b\in B$, the fibre $Y_b$ is a
  non-singular semi-Fano 3--fold, and the fibre $S_b\subset Y_b$ is
  a non-singular K3 surface in the linear system $\abs{-K_{Y_b}}$,
\end{enumerate}
together with an isometry $N\cong \uPic (\df{Y}/B)$
sending $A$ to $-K_{\df{Y}}$, such that:
\begin{enumerate}
\item[(ii)] for every geometric point $b\in B$, the composition $N \to \Pic (Y_b)\to
  H^2(S_b;\ZZ)$ is equivalent to~$j$.
\end{enumerate}
\end{definition}

\begin{theorem}
\label{thm:fano_stack_smooth}
  The stack $\mathfrak{F}^{N,A}$ is a smooth algebraic stack.
\end{theorem}

\begin{remark}
  The stack $\mathfrak{F}^{N,A}$ is often not connected: in examples, the connected
  components can often be understood in terms of flops relating
  different (partial) resolutions of singular Fano \mbox{3--folds}.
\end{remark}

\begin{proof}
  In the Fano case, Beauville~\cite{beauville:fano} shows that
  $\mathfrak{F}^{N,A}$ is a smooth algebraic stack. The proof in~\cite{beauville:fano} works word for word once we establish that
  $H^2(Y, \, T_Y)=(0)$ -- which implies that the stack $\mathfrak{F}^{N,A}$
  is smooth. But
  $H^2(Y, \, T_Y) = H^2\bigl(Y, \, \Omega^2_Y\otimes(-K_Y)\bigr)$ is
  Serre dual to $H^1(Y, \, \Omega^1_Y\otimes K_Y)$, and this group
  vanishes for any semi-Fano 3--fold (but not for a general weak Fano)
  thanks to \fullref{thm:aknava}.
\end{proof}

\begin{remark*}
$H^{2}(Y,T_{Y})$ does not always vanish for a weak Fano $Y$ (see~\cite[Example~2.7]{sano13})
and therefore Beauville's proof of smoothness of $\mathfrak{F}^{N,A}$ 
does not work in the general weak Fano setting. 
However, this does not necessarily mean that the stack $\mathfrak{F}^{N,A}$
fails to be smooth. In fact, using Paoletti's result (\fullref{thm:reid}) 
that a generic anti\-canonical member $S \in \abs{-K_{Y}}$ is a non-singular K3 surface, 
it follows from work of Ran \mbox{\cite[Corollary~3]{ran}}, using the so-called
$T^{1}$--lifting method, 
that $\mathfrak{F}^{N,A}$ is still smooth for any (smooth) weak Fano 3--fold $Y$.
(Very recently, Sano~\cite[Theorem~1.1]{sano13} considered the extension of this result 
to weak Fano $n$--folds for $n>3$ in which case it is no longer true that a general $S \in \abs{-K_{Y}}$ 
need be smooth.)
\end{remark*}

The only reason that we use moduli stacks rather than spaces here is
so we can use infinitesimal arguments in the proof of
\fullref{thm:fanok3} below: the stack is smooth even when the
space is not. As we already noted, for any semi-Fano 3--fold $Y$ the restriction homomorphism
$\Pic Y \to \Pic S\subset H^2(S;\Z)$ is a primitive embedding. Thus,
we view $S$ as an $(N,A)$--polarised K3 surface. As above, let
$\mathfrak{K}^{N,A}$ be the stack of $(N,A)$--polarised K3
surfaces. There is an obvious forgetful morphism
\[
s^{N,A} \co \mathfrak{F}^{N,A} \to \mathfrak{K}^{N,A} .
\]
The following is the key result of this section and lies at the core of the
matching argument in~\cite[Section~6]{chnp2}. 

\begin{theorem}
\label{thm:fanok3}
The morphism $s^{N,A}\co \mathfrak{F}^{N,A}\to
\mathfrak{K}^{N,A}$ is smooth and generically surjective. More
precisely, let $\mathfrak{F}\subset \mathfrak{F}^{N,A}$ be any
connected component, and denote by $s\co \mathfrak{F} \to
\mathfrak{K}^{N,A}$ the restriction of $s^{N,A}$ to
$\mathfrak{F}$. Then $s$ is smooth and generically surjective.
\end{theorem}

\begin{proof}
 Beauville's proof in~\cite{beauville:fano} works word for word. 
\end{proof}
\begin{remark*}
As already remarked above, \fullref{lem:lef} can definitely fail for general weak Fano \mbox{3--folds}.
Hence, in the general weak Fano case it is not a priori clear what moduli space (stack) of lattice polarised 
K3 surfaces $\mathfrak{K}$ should appear as the target of the forgetful morphism $s$ above.
An appropriate modification of \fullref{thm:fanok3} may still hold in the general weak Fano 
case; we will not consider this issue further in this paper.
\end{remark*}

In order to show that the set of \hk structures that appear in the limits of
our ACyl Calabi--Yau manifolds is large in the sense we want, it remains to find
a sufficient condition for a class in $L_\bbr$ to correspond to a restriction
of a Kähler class from $Y$.

\begin{prop}
  \label{prop:fanok3} As in the previous theorem, let
  $\mathfrak{F}\subset \mathfrak{F}^{N,A}$ be any connected component
  of the moduli stack $\mathfrak{F}^{N,A}$ of semi-Fano
  \mbox{3--folds}, and $s\co \mathfrak{F} \to \mathfrak{K}^{N,A}$
  the forgetful morphism. Recall that, according to the discussion in
  the previous section, $\mathfrak{K}^{N,A}=[D^0/\Gamma]$ where $D^0\subset
  D$ is an open subset of the appropriate Griffiths domain of oriented
  positive planes $\Pi \subset N^\perp_\RR$. Then there exist:

\begin{enumerate}[label=\textup{(\arabic*)}]
\item a subset $U_{\mathfrak{F}} \subseteq D^0$ with complement a locally finite
union of complex analytic submanifolds of positive codimension, and
\item an open subcone $\Amp_{\mathfrak{F}}\subset N_\bbr$,
\end{enumerate}
with the following property: Fix any pair $(\Pi, k)$ of $\Pi \in
U_{\mathfrak{F}}$ and $k \in \Amp_{\mathfrak{F}}$; denote by $(S, h)$
the \emph{marked} $(N,A,j)$--polarized K3 surface with period point
$\Pi(S,h)=\Pi$ (this means, in particular, that $h\co
H^2(S;\ZZ) \to L$ is an isometry); then there is an embedding $S\subset Y$ in a semi-Fano
\mbox{3--fold}~$Y$, and a K\"ahler class $[\omega]\in (\Pic Y)\otimes
\RR$ such that $\hdg ([\omega_{\mid S}]) = j(k) \in L_{\R}$.
\end{prop}

\begin{proof}
  By the following \fullref{lem:paoletti}, there is a Zariski open subset
  $\mathfrak{F}^0\subset\mathfrak{F}$ (the only open stratum of the
  Zariski locally closed stratification in the statement of that
  lemma) such that the cone $\Amp{Y_b}\subset N_{\R}$ is constant for
  $b\in \mathfrak{F}^0$. Let $\Amp_{\mathfrak{F}}$ denote this
  constant cone: it is an open cone, because ampleness is an open
  property. By \fullref{thm:fanok3} the restriction of $s$ to
  $\mathfrak{F}^0$ is generically surjective. Therefore the image
  $s(\mathfrak{F}^0)$ contains a Zariski open subset $W\subset
  s(\mathfrak{F}^0)\subset \mathfrak{K}^{N,A}$ and, denoting by
  $p\co D^0\to [D^0/\Gamma]=\mathfrak{K}^{N,A}$ the natural
  projection, we take $U_{\mathfrak{F}}=p^{-1}(W)$. Here (1) holds
  because $\Gamma$ is a discrete group and the action on $D^0$ is
  properly discontinuous. We claim that the open $U_{\mathfrak{F}}$
  and the cone $\Amp_{\mathfrak{F}}$ just defined satisfy the
  conclusion.

  Indeed, choose a pair $(\Pi, k)$ with $\Pi \in U_{\mathfrak{F}}$ and
  $k \in \Amp_{\mathfrak{F}}$. By construction, $p(\Pi)\in W\subset
  s(\mathfrak{F}^0)$: that is, $S=p (\Pi)$ is part of a pair
  $(S\subset Y)\in \mathfrak{F}^0$ where $Y$ is a semi-Fano \mbox{3--fold}
  and $k \in \Amp_{\mathfrak{F}} = \Amp Y$, so tautologically $k$
  corresponds to a K\"ahler class $[\omega]$ on $Y$ under the
  identification $N= \Pic (Y)$ that is part of the data of the moduli
  problem. The statement that $h([\omega_{|S}])=j(k)$ is a tautology.
\end{proof}

The proof of  \fullref{lem:paoletti} used above
rests on Paoletti's study~\cite{paoletti} of how
the K\"ahler cone of a weak Fano (quasi-Fano in his terminology)
\mbox{3--fold} changes under deformation. It is well known that the
K\"ahler cone of a non-singular Fano \mbox{$n$--fold} is locally constant under
deformation (see Wi\'sniewski~\cite{wis91}). Paoletti's main
result~\cite[Theorem~1.1]{paoletti} is a characterisation of how the
K\"ahler cone of a weak Fano \mbox{3--fold} can fail to be locally
constant under deformation. We don't actually need the precise
formulation of his result; we only need to know that, in an algebraic family, the
cone is constant on a Zariski open subset. Also note \emph{loc. cit.}~Corollary~1.2
stating that the K\"ahler cone is constant in a family of weak Fano
\mbox{3--folds} whose anticanonical morphism is small and  
in particular for any semi-Fano obtained as a (projective) small resolution of nodal Fano 3--fold;
this is the case for almost all of the examples we consider in detail in this paper.
\begin{lemma}
\label{lem:paoletti}
Let $f \co \df{Y} \to B$ be a flat algebraic family of semi-Fano
\mbox{3--folds} together with an isometry $N\cong \uPic (\df{Y}/B)$.
There is a Zariski locally closed stratification \mbox{$\coprod B_i = B$}
of $B$ such that for all $i$ the ample cone $\Amp Y_b\subset N_\bbr$
is constant in $b\in B_i$.
\end{lemma}

\begin{proof}
  The result follows easily from~\cite[Theorem~1.1]{paoletti}. Indeed
  consider the flat family $\df{X} \to B$ of anticanonical models of the
  family $\df{Y}$. For all $b\in B$, let $E_b\subset Y_b$ be the
  exceptional set of the birational morphism $Y_b\to X_b$, with its
  reduced scheme structure. Let
  $\coprod B_i \to B$ be a Zariski locally closed stratification of
  $B$ such that for all~$i$:
\[
E_i=\bigcup_{b\in B_i} E_b \to B_i
\]
  is a flat family. Now~\cite[Theorem~1.1]{paoletti} immediately
  implies that  $\Amp Y_b\subset N_\bbr$ is constant in $b\in B_i$. Indeed, 
  if for some $b_0\in B_i$ $E_{b_0}$ contains a surface $F_b\subset Y_b$
  contracting to a curve $C_b\subset X_b$, then, by flatness, so
  does every $b\in B_i$. 
\end{proof}

\begin{remark*}
  It is important to understand that $\Amp_{\mathfrak{F}}$ is not the whole
  K\"ahler cone of $S$, even generically. If $Y$ is semi-Fano but not
  Fano, and the anticanonical morphism $Y\to X$ is small, 
  then $-K_Y$ is not a K\"ahler class on $Y$ but it is when restricted
  to a generic $S$. 

  There is, however, an issue even in the strict Fano case when rank
  $\geq 2$.  For example consider a tridegree $(2,2,2)$ hypersurface $S$
  in $Y = \PP^1 \times \PP^1 \times \PP^1$. Then
  $\Amp Y_{|S}$ is a fundamental domain for the action of $\Aut S$
  (a free group on the three involutions) on $\Amp S$:
  see Oguiso~\cite{oguiso}.
\end{remark*}

\begin{remark*}
  Different components $\mathfrak{F} \subset \mathfrak{F}^{N,A}$ have
  different $\Amp_{\mathfrak{F}}$, for example, in \fullref{exa:quartic_w_plane} where we consider a generic quartic
  containing a plane the $\Amp_{\mathfrak{F}}$ of the two different
  small resolutions are the two components of $\Amp S \setminus \gen{A}$.
\end{remark*}

\begin{example}
\label{ex:ky:nonample}
  The restriction $-K_{Y|S}$ is ample if and only if the anticanonical
  morphism $Y\to X$ is small. The following examples further
  illustrate the statements of \fullref{thm:fanok3} and \fullref{prop:fanok3}:
  \begin{enumerate}
  \item $Y=\FF_2 \times \PP^1$ where $\FF_2$ is the Segre
    surface. The anticanonical morphism contracts the surface
    $E=e \times \PP^1$ where $e \subset \FF_2$ is the curve
    of self-intersection $e^2=-2$. In particular, if $S\in \acls{Y}$
    is non-singular then $-K_{Y|S}$ is not ample and it always contracts
    \emph{two} curves of self-intersection $-2$. Consider the basis
    of $N = \Pic Y$ consisting of $D=\FF_2\times \{\text{pt}\}$,
    $E$~as above, and $F=f\times \PP^1$ where $f\subset \FF_2$ is the
    class of a fibre. $-K_Y = 2D+2E+4F$, and the matrix of
    the intersection form is:
\[
\begin{pmatrix}
  0 & 0 & 2\\
  0 & -4 & 2\\
  2 & 2& 0
\end{pmatrix}\,.
\]
   Note that there are no $-2$ classes in $N$. In fact, $\Pic S$
   always has rank $4$: as $Y$ deforms to $\PP^1\times \PP^1 \times
   \PP^1$, the surface $E\subset Y$ ``evaporates,'' and $S$ deforms to a rank 3
   K3 surface. 
 \item Let $X$ be a general quartic \mbox{3--fold} containing a
   double line $\ell \subset \PP^4$. It is easy to check that the proper
   transform $Y$ of $X$ in the blowing up of $\ell \subset \PP^4$ is
   non-singular and the exceptional divisor $E\subset Y$ mapping to $\ell$
   is a conic bundle surface with 6 singular fibres. In this case $Y$
   has rank 2 and $E^2\cdot A=-2$, thus $E\in N$ is a $-2$ class.
   Moreover, it is clear that $E$ survives all deformations of~$Y$.
  (See \fullref{ex:doubleline} below.)
  \end{enumerate}
\end{example}

\begin{remark*}
\fullref{ex:ky:nonample}(i) illustrates that the property of being Fano is unstable under 
deformation; see also Paoletti~\cite[Example~1.3]{paoletti} for another such example.
Also a weak Fano 3--fold with small AC morphism may be deformation equivalent 
to a weak Fano 3--fold with AC morphism which is not small: see~\cite[Example~1.6]{paoletti}.
\end{remark*}

\section{Examples: building blocks}
\label{sec:examples}

\newcommand{\amb}{G}

We construct a handful of building blocks and compute the topological
invariants considered in \fullref{sec:blocks} of these building blocks. 
The results are summarized in Tables~\ref{table:c2} and~\ref{table:blocks}.
These building blocks and their topological invariants will be used in~\cite{chnp2} to
construct examples of compact \gtmfd s and to determine their topology.
In this section we make no attempt at being systematic: we
only construct a very small number of typical examples.

In \fullref{S:wk:fano:examples} we discuss many more general classes of 
semi-Fano 3--folds to illustrate the variety of examples available. 
Applying \fullref{prop:block_from_weak} to any of these examples yields 
a building block $(Z,f,S)$ in the sense of \fullref{dfn:BLOCK} and therefore by \fullref{thm:acyl_calabi} an ACyl Calabi--Yau 
structure on the quasiprojective 3--fold $Z\setminus S$.
Similar methods to those utilised in the present section would also allow the computation 
of the basic topological invariants of the corresponding building blocks and ACyl Calabi--Yau \mbox{3--folds}.

To construct the examples in this section typically we start with a singular Fano \mbox{3--fold}
$X$ with only ordinary double points; resolve this to a non-singular
semi-Fano $Y$ and anticanonical morphism $Y\to X$; choose a K3 surface $S$
and pencil in $\acls{Y}$, and resolve the indeterminacies to obtain
$S\subset Z \to \PP^1$ where $S$ is the fibre of $\infty\in \PP^1$.
According to \fullref{prop:block_from_weak}, $Z$ is a
``building block'' in the sense of \fullref{dfn:BLOCK},
and therefore by \fullref{thm:acyl_calabi}, $V = Z \setminus S$ admits ACyl Calabi--Yau metrics.

We compute the following topological invariants of the building blocks:
the degree $-K_Y^3$, the integral cohomology
groups $H^2(Z)$ and $H^3(Z)$, the primitive sublattice $N\subset L$ of the
K3 lattice, the
kernel $K$ of $H^2(V) \to L$, the greatest divisor of $c_2(Z)$, and the number
$e(Z)$ of $(-1,-1)$--curves.
In examples involving small resolutions of 3--folds with ordinary double points,
all these invariants are independent of the choice of small resolution, except
possibly the greatest divisor of $c_2(Z)$ (recall \fullref{rmk:c2:flop}).

In the calculation of $H^m(Z)$ we use \fullref{lem:top_of_blowup}; 
to compute $b^{3}(Y)$ we use \fullref{L:weak:fano:b3}; 
to compute $c_{2}(Z)$ we use \fullref{prop:c2blowup}.
In all cases, except in \fullref{ex:2conics} which requires some extra
work, it is immediate from \fullref{prop:block_from_weak} that the
sublattice $N\subset L$ is primitive.

\begin{example}
\label{ex:species1}
A class of examples, already considered in Kovalev~\cite{kovalev:connectsums},
is to take $Y$ to be a Fano ``of the first species'', that is, a member of one of
the $17$ deformation families of smooth Fano 3--folds with Picard
rank $1$, and let $Z$ be the building block arising from blowing up the
(smooth) base locus of a generic transverse anticanonical pencil.
Let $r$ be the index of $Y$, and $d = (-\frac{1}{r}K_Y)^3$ the degree.
Then by definition $-K_Y = rH$ for $H$ a generator for $\Pic Y$,
and $(-K_Y)H^2= rd$. So the polarising lattice is $N = \gen{rd}$.

For these examples, \fullref{cor:c2blowup} easily gives the
greatest divisor of $c_2(Z)$.
Consulting Iskovskih~\cite[Table~6.5]{isk:fano2} and
Iskovskih--Prokhorov's book \cite[Table 12.2]{fano:varieties} we summarise the values of
$b^3(Z)$ and greatest divisor of $c_2(Z)$ in \fullref{table:c2}.

\begin{table}[ht!]
\[
\begin{array}[t]{l c  c  c  c  c} \toprule
  Y & r &   -K_Y^3  &b^{3}(Y)&b^{3}(Z) & \gdiv{c_2(Z)} \\ \midrule
          \CP^3 &  4 & 4^3         &   0 &  66 &   2 \\ 
             Q_2 \subset \CP^4 &  3 & 3^3 \cdot 2 &   0 &  56 &   2 \\ 
                    V_1 \to W_4 &  2 & 2^3         &  42 &  52 &   8 \\ 
                  V_2 \to \CP^3 &  2 & 2^3 \cdot 2 &  20 &  38 &   4 \\ 
              Q_3 \subset \CP^4 &  2 & 2^3 \cdot 3 &  10 &  36 &  24 \\ 
      V_{2\cdot2} \subset \CP^5 &  2 & 2^3 \cdot 4 &   4 &  38 &   4 \\ 
              V_5 \subset \CP^6 &  2 & 2^3 \cdot 5 &   0 &  42 &   8 \\ 
                  V_2 \to \CP^3 &  1 &           2 & 104 & 108 &   2 \\ 
              Q_4 \subset \CP^4 &  1 &           4 &  60 &  66 &   4 \\ 
      V_{2\cdot3} \subset \CP^5 &  1 &           6 &  40 &  48 &   6 \\ 
V_{2\cdot2\cdot2} \subset \CP^6 &  1 &           8 &  28 &  38 &   8 \\ 
           V_{10} \subset \CP^7 &  1 &          10 &  20 &  32 &   2 \\ 
           V_{12} \subset \CP^8 &  1 &          12 &  14 &  28 &  12 \\ 
           V_{14} \subset \CP^9 &  1 &          14 &  10 &  26 &   2 \\ 
        V_{16} \subset \CP^{10} &  1 &          16 &   6 &  24 &   8 \\ 
        V_{18} \subset \CP^{11} &  1 &          18 &   4 &  24 &   6 \\ 
        V_{22} \subset \CP^{13} &  1 &          22 &   0 &  24 &   2 \\ 
        \bottomrule
\end{array} 
\]
\caption{Building blocks $Z$ from Fanos $Y$ with Picard rank 1}
\label{table:c2}
\end{table}
\end{example}

\begin{example}
\label{exa:mm}
We can also readily construct building blocks from the 36 rank 2 Fanos in the
Mori--Mukai classification, but do not describe the examples in further detail
here.
\end{example}

Examples~\ref{exa:quartic_w_plane} through~\ref{exa:quartic_w_22} arise in a
uniform way. We impose the condition that a quartic in $\CP^{4}$ contain a
special surface $W$: 
a projective plane $\Pi$, a quadric surface $Q_{2}^{2}$, a cubic scroll surface $\FF$ and 
the complete intersection of two quadrics $F_{2,2}$ respectively. 
The generic such quartic $X$ has only ODPs, the number $e$ of which is determined by the 
special surface $W$ imposed and all of which are contained in $W$. 
$W$ gives us a Weil divisor on $X$ which is not \mbox{$\Q$--Cartier}; 
blowing up $W\subset X$ as in \fullref{l:small:morphism} 
yields a smooth projective small resolution $Y$ with anticanonical morphism 
$\varphi\co Y \to X$. $Y$ is a smooth semi-Fano 3--fold with Picard rank $\rho=2$  
(so the defect $\sigma(X)$ is $1$)  whose
integral cohomology group $H^2(Y)$ is spanned by the anticanonical class
$A=-K_{Y}$ and~$\wtilde{W}$, 
the proper transform of the special surface $W \subset X$.
Since the anticanonical morphism $\varphi\co Y \to X$ is small by \fullref{T:flops} we
can flop $\varphi$ to obtain another smooth weak Fano $Y^{+}$ with $\rho(Y^{+})=2$; 
by \fullref{R:flop:rk2} there is a unique such flop of $\varphi$.
In general the flop $Y^{+}$ is not isomorphic to $Y$ but shares the same topological invariants 
except possibly for $c_{2}(Y)$ which we compute.

\begin{example}[Generic AC pencil on a small resolution of a generic
  quartic containing a plane]
\label{exa:quartic_w_plane}
The following semi-Fano also appears in 
work of Cheltsov~\cite[Lemma~25]{cheltsov:nodal:quartics}, Jahnke--Peternell--Radloff~\cite[3.15]{peternell2}
and Takeuchi~\cite[2.9.6 and~6.6.6]{takeuchi}.
  Fix a 2--plane $\Pi\subset \PP^4$ and let $\Pi \subset X\subset
  \PP^4$ be a general quartic 3--fold containing $\Pi$. Choose
  homogeneous coordinates $x_0,\ldots, x_4$ on $\PP^4$ such that
  $\Pi=(x_0=x_1=0)$; then $X=(f_4=0)$ is the zero locus of a
  homogeneous quartic in the ideal $(x_0,x_1)$:
\[
f_4=x_0 a_3+x_1b_3
\]
  where $a_3,b_3$ are degree 3 homogeneous in $x_0,\ldots,x_4$. To
  say that $X$
  is general is to say that the forms $a_3,b_3$ are general; thus, $X$ has
  $9$ ordinary double points $x_0=x_1=a_3=b_3=0$ -- see also the remark at 
  the end of this example.
  Blowing up $\Pi$ yields a non-singular semi-Fano 3--fold $Y\to X$ with $e=9$
  $(-1,-1)$--curves resolving the $9$ ordinary double points of $X$ on $\Pi$.  
  We show that:
  \begin{itemize}
  \item $H^2(Y)=\ZZ^2$ with basis $\tilde \Pi$ (the proper transform of $\Pi$)
   and $A=-K_Y$, and quadratic form in this basis
  \[\begin{pmatrix}
    -2 & 1\\
     1  & 4
  \end{pmatrix};\] 
  \item $H^3(Y)\simeq \ZZ^{44}$.
  \end{itemize}
Below we discuss how to compute $H^2(Y)$ and $H^3(Y)$. The building block $f\co
Z\to \PP^1$ is obtained by blowing up the base locus of a pencil
$|S_0,S_\infty|\subset \acls{Y}$ where $S_0,S_\infty$ are non-singular
and meet transversely. The base locus is a non-singular curve $C$ of
genus $3$ (naturally a plane quartic); hence, $H^3(Z)\simeq
H^3(Y)\oplus H^1(C)\simeq \ZZ^{44}\oplus \ZZ^6$.

To calculate $H^2(Y)$ and $H^3(Y)$, we proceed as follows. 
First, $Y$ is the proper transform of $X$ in the blowup $\amb \to \PP^4$
of the plane $\Pi$; this is the scroll with weight data
\[
\begin{array}[t]{cccccc}
  s_0 & s_1 & x_2 & x_3& x_4 & x\\\hline
  1    & 1    & 0    & 0  &    0 &-1\\
  0    & 0    & 1    & 1  &    1 & 1
\end{array}
\]
with morphism to $\PP^4$ given by $x_0=s_0x, x_1=s_1x$. Denoting by $L$ the
line bundle on $\amb$ with sections $s_0, s_1$ and $M$ the line bundle with
sections $x_2,\ldots$, we see that the equation of $Y \subset \amb$ is:
\[
s_0a_3+s_1b_3
\]
that is $Y\in |L+3M|$. Thus $Y\subset \amb$ is an ample divisor; it then
follows easily from the Lefschetz theorems that the restriction
$H^2(\amb)\to H^2(Y)$ is an isomorphism and $H^3(Y)$ is
torsion-free. To see this consider the long exact cohomology sequence of
  the pair $\amb, Y$:
\[
\cdots \to H^m(\amb,Y)\to H^m(\amb)\to H^m(Y)\to H^{m+1}(\amb,Y)\to \cdots \,,
\]
note that $H^m(\amb,Y) = H^m_{cpt}(\amb\setminus Y)=H_{8-m}(\amb\setminus Y)$,
and recall that the Lefschetz homotopy dimension theorem states that
$\amb\setminus Y$ has the homotopy type of a CW complex of real dimension $4$.
It follows that $H^m(\amb,Y)=(0)$ for $m<4$ and that $H^4(\amb,Y)$ is
torsion-free.

We calculate $b^{3}(Y)$ by applying \fullref{L:weak:fano:b3}. 
In the present case we have $e=9$, $\sigma=b^{2}(Y)-b^{2}(X)=2-1=1$ and $b=60$ 
since the third Betti number of a smooth quartic in $\CP^{4}$ is $60$. 
Hence by \eqref{E:weak:fano:b3} we have
\[
b^{3}(Y) = 60 - 2 \times 9 +2 \times 1 = 44.
\]
It remains to compute $c_{2}(Z)$ for both $Y$ and its (unique) flop $Y^{+}$; 
for this we first need to compute $c_{2}(Y)$ and then apply the blow-up formula 
\fullref{prop:c2blowup} to compute $c_{2}(Z)$. 
$\wtilde \Pi$ is the blow-up of $\Pi$ at 9 points, so $c_2(\wtilde \Pi) = 12$
and $c_1(\wtilde \Pi)^2 = 0$. The second term in \eqref{eq:c2div} we can
compute from the quadratic form on $H^2(Y)$:
$$(-\wtilde \Pi^2 + 2\wtilde\Pi(-K_Y))(-K_Y) = -(-2) + 2 = 4.$$
Hence \fullref{lem:c2div} gives $\smash{(c_2(Y) + c_1(Y)^2)_{|\wtilde
\Pi}} = 16$.
Since $\chi(C) = -4$, it
follows from \fullref{prop:c2blowup} that $\gdiv{c_2(Z)} = 4$.
If we flop $Y$ to the other small resolution $Y^+$ of~$X$, then the proper transform
of $\Pi$ is isomorphic to $\Pi$, $(c_2(Y^+) + c_1(Y^+)^2)_{|\Pi} = -2$
and $\gdiv c_2(Z^+) = 2$. 
\end{example}

\begin{remark*}
We have the following elementary lemma, for example, see Finkelnberg
\mbox{\cite[Proposition 1.1]{finkelnberg1987small}}: 
if $\Pi$ is a plane contained in a hypersurface $X \subset \CP^{4}$ 
of degree $d \ge 2$ then $X$ is singular and $\Pi$ contains at least one singular point of $X$. 
If $X$ contains only finitely many singular points then it contains at most $(d-1)^{2}$ singular points.
If $X$ has only nodes there are exactly $(d-1)^{2}$ singular points on $X$.
\end{remark*}
\begin{example}[Small resolution of a generic quartic containing a quadric surface]
See also Cheltsov~\cite[Example~10]{cheltsov:nodal:quartics}.
  \label{exa:quartic_w_quadric}
  Fix a quadric surface $Q=Q^2_2\subset \PP^4$ and let $Q \subset X\subset
  \PP^4$ be a general quartic 3--fold containing $Q$. Choose
  homogeneous coordinates $x_0,\ldots, x_4$ on $\PP^4$ such that
  $Q=(x_0=x_1x_2+x_3x_4=0)$; then $X=(f_4=0)$ is the zero locus of a
  homogeneous quartic in the ideal of $Q$:
\[
f_4=x_0 a_3+(x_1x_2+x_3x_4)b_2
\]
  where $a_3,b_2$ are general homogeneous forms of degrees $3, 2$ in
  $x_0,\ldots,x_4$. Thus, $X$~has
  $12$ ordinary double points $x_0=x_1x_2+x_3x_4=a_3=b_2=0$.  
  Blowing up $Q$ yields a non-singular semi-Fano 3--fold $Y\to X$ with $e=12$
  $(-1,-1)$--curves resolving the $12$ ordinary double points of $X$ on $Q$.  
  We show that:
  \begin{itemize}
  \item $H^2(Y)=\ZZ^2$ with basis $\tilde Q,A$, and quadratic form in this basis
  \[\begin{pmatrix}
    -2 & 2\\
     2  & 4
  \end{pmatrix};\] 
  \item $H^3(Y)\simeq \ZZ^{38}$.
  \end{itemize}
The building
block $f\co Z \to \PP^1$ is obtained by blowing up the base locus of a
pencil $|S_0,S_\infty|\subset \acls{Y}$ where $S_0,S_\infty$ are
non-singular and meet transversely. The base locus is a non-singular
curve $C$ of genus $3$ (naturally a plane quartic); hence,
$H^3(Z)\simeq H^3(Y)\oplus H^1(C)\simeq \ZZ^{44}$. 

To calculate $H^2(Y)$ and $H^3(Y)$ we proceed as follows. First, $Y$
is the proper transform of $X$ in the blowup $\amb\to \PP^4$ of the
quadric $Q=(x_0=x_1x_2+x_3x_4=0)$. We realize $\amb$ as the
hypersurface with equation
\[
sx_0+t(x_1x_2+x_3x_4)=0
\] 
in the 5--dimensional toric scroll with weight data
\[
\begin{array}[t]{ccccccc}
  x_0&x_1&x_2&x_3&x_4&s&t\\\hline
  1   & 1  & 1  & 1  & 1  &0&-1\\
  0   & 0  & 0  & 0  & 0  &1&1\\ 
\end{array} 
\]
We denote by $L$, respectively\ $M$, the line bundles on this scroll with
global sections $x_i$, respectively\ $s, tx_i$. 
Thus, $Y$ is given in $\amb$ by the two simultaneous equations:
\[
\begin{cases}
  sx_0+t(x_1x_2+x_3x_4)&=0\\
  sb_2-ta_3                    &=0
\end{cases}
\]
Hence, $Y\subset \amb$ is the complete intersection of two ample
hypersurfaces of type $L+M$ and $2L+M$; as in the previous example, it
follows from Lefschetz that the restriction $H^2(\amb)\to H^2(Y)$ is an
isomorphism and $H^3(Y)$ is torsion-free.
We compute $b^{3}(Y)$ as in the previous example using \fullref{L:weak:fano:b3}. Since in this case we have $e=12$, $\sigma=1$
and $b=60$, \eqref{E:weak:fano:b3} yields $b^{3}(Y)=60-24+2=38$.

$c_2(\wtilde Q) - c_1(\wtilde Q)^2 = 20$, and \fullref{lem:c2div} gives
$(c_2(Y) + c_1(Y)^2)_{|\wtilde Q} = 26$. Since $\chi(C) = -4$, \fullref{prop:c2blowup} implies $\gdiv c_2(Z) = 2$. Flopping does not change any
invariants of $Y$, since it corresponds to blowing up a different quadric
surface ($x_0 = b_2 = 0$) contained in $Y$.
\end{example}

\begin{example}[See also Entry~30 in Kaloghiros~{\cite[Table~1]{kaloghiros:thesis}}]
  \label{exa:quartic_w_scroll}
  Fix a cubic scroll surface $\FF\subset \PP^4$ and let $\FF \subset X\subset
  \PP^4$ be a general quartic 3--fold containing $\FF$. One can choose
  homogeneous coordinates $x_0,\ldots, x_4$ on $\PP^4$ such that $\FF$
  is the locus where the matrix
 \[
M=
 \begin{pmatrix}
   x_0 & x_1 & x_2 \\
   x_2 & x_3 & x_4
 \end{pmatrix}
\]
has rank $<2$; then $X=(f_4=0)$ is the zero locus of a homogeneous
quartic in the ideal of the $2\times 2$ minors of $M$:
\[
f_4=
\begin{pmatrix}
x_1x_4-x_2x_3 & -x_0x_4+x_2^2 & x_0x_3-x_1x_2  
\end{pmatrix}
\begin{pmatrix}
  a_2\\
  b_2\\
  c_2
\end{pmatrix}
\]
  where $a_2,b_2,c_2$ are general homogeneous forms of degrees $2$ in
  $x_0,\ldots,x_4$.
  A straightforward calculation with the Porteous
  formula (see Arbarello--Cornalba--Griffiths--Harris
  \cite[Chapter~II, (4.2)]{ACGH}) shows that $X$ has $17$
  ordinary double points; the singularities of $X\subset \PP^4$ are the
  locus in $\PP^4$ where the matrix
\[
A=\begin{pmatrix} x_0 & x_1 & x_2 \\ x_2 & x_3 & x_4 \\ a_2 & b_2 &
  c_2 \end{pmatrix}
\]
 has rank $1$.
  Blowing up $\FF$ yields a non-singular semi-Fano 3--fold $Y\to X$ with $e=17$
  $(-1,-1)$--curves resolving the $17$ ordinary double points of $X$ on $\FF$.  
  We show that:
  \begin{itemize}
  \item $H^2(Y)=\ZZ^2$ with basis $\tilde\FF,A$, and quadratic form in this
    basis
  \[\begin{pmatrix}
    -2 & 3\\
     3  & 4
  \end{pmatrix};\] 
  \item $H^3(Y)\simeq \ZZ^{28}$.
  \end{itemize}
The building block $f\co
Z\to \PP^1$ is obtained by blowing up the base locus of a pencil
$|S_0,S_\infty|\subset \acls{Y}$ where $S_0,S_\infty$ are non-singular
and meet transversely. The base locus is a non-singular curve $C$ of
genus $3$ (naturally a plane quartic); hence, $H^3(Z)\simeq
H^3(Y)\oplus H^1(C)\simeq \ZZ^{28}\oplus \ZZ^6$. 

To calculate $H^2(Y)$
and $H^3(Y)$, the strategy, as usual, is to show that $Y$ is a
complete intersection of ample hypersurfaces in a non-singular toric
variety. Indeed, the blow up $\amb$ of $\FF\subset \PP^4$ is the complete
intersection given by equations:
\[
\begin{pmatrix}
  x_0&x_1&x_2\\x_2&x_3&x_4
\end{pmatrix}
\cdot
\begin{pmatrix}
 y_0\\y_1\\y_2 
\end{pmatrix}
=0
\quad
\text{in}
\quad
\PP^4\times\PP^2_{y_0,y_1,y_2}
\]
and $Y$ is given in $\amb$ by the equation:
\[
a_2y_0+b_2y_1+c_2y_2=0.
\]
Thus, $Y$ is the complete intersection of 3 ample hypersurfaces in
$\PP^4\times \PP^2$. Everything else from now on proceeds as in the
previous examples. Since $e=17$, $\sigma=1$ and $b=60$, \eqref{E:weak:fano:b3} yields 
$b^{3}(Y)=60-34+2=28$.

$\FF$ is $\CP^2$ blown up in 1 point, and $\tilde \FF$ is $\FF$ blown up in 17
points. \fullref{lem:c2div} gives $(c_2(Y) + c_1(Y)^2)_{|\tilde \FF} = 38$.
Hence $\gdiv c_2(Z) = 2$ by \fullref{prop:c2blowup}.
In the other small resolution $Y^+$ of $X$, the proper transform of $\FF$ is isomorphic to $\FF$.
There $(c_2(Y^+) + c_1(Y^+)^2)_{|\FF} = 4$ and $\gdiv c_2(Z^+) = 4$.
\end{example}

\begin{example}[See also Cheltsov~{\cite[Theorem~11,
Lemma~21]{cheltsov:nodal:quartics}},
Jahnke--Peternell\allowbreak--Radloff~{\cite[3.9.II.6.a]{peternell2}},
Kaloghiros~{\cite[Example 3.9]{kaloghiros:classify}},
 and Takeuchi~{\cite[2.11.10]{takeuchi}}.]
\label{exa:quartic_w_22}
  Fix the complete intersection of two quadrics $F=F_{2,2}\subset
  \PP^4$ (that is, a del Pezzo surface of degree $4$) and let $F \subset X\subset \PP^4$ be a general quartic
  3--fold containing $F$. In homogeneous coordinates $x_0,\ldots, x_4$
  on $\PP^4$, $F=(p_2=q_2=0)$ where $p_2,q_2$ are general homogeneous
  quadratic polynomials; then $X=(f_4=0)$ is the zero locus of a
  homogeneous quartic in the ideal of $F$:
\[
f_4=p_2 a_2+q_2b_2
\]
where $a_2,b_2$ are general homogeneous quadratic forms in
$x_0,\ldots,x_4$. Thus, $X$ has $16$ ordinary double points
$p_2=q_2=a_2=b_2=0$.  Blowing up $F$ yields a non-singular semi-Fano
3--fold $Y\to X$ with $e=16$ $(-1,-1)$--curves resolving the $16$
ordinary double points of $X$ on $F$.  By using the methods described
in the previous examples, it is easy enough to show that:
  \begin{itemize}
  \item $H^2(Y)=\ZZ^2$ with basis $\tilde F,A$, and quadratic form in this basis
  \[\begin{pmatrix}
     0  & 4\\
     4  & 4
  \end{pmatrix};\] 
  \item $H^3(Y)\simeq \ZZ^{30}$.
  \end{itemize}
The building block $f\co
Z\to \PP^1$ is obtained by blowing up the base locus of a pencil
$|S_0,S_\infty|\subset \acls{Y}$ where $S_0,S_\infty$ are non-singular
and meet transversely. The base locus is a non-singular curve $C$ of
genus $3$ (naturally a plane quartic); hence, $H^3(Z)\simeq
H^3(Y)\oplus H^1(C)\simeq \ZZ^{30}\oplus \ZZ^6$.

\fullref{prop:c2blowup} implies $\gdiv c_2(Z) = 4$, using
$(c_2(Y) + c_1(Y)^2)_{|\tilde F} = 44$. Flopping does not change any invariants
of $Y$, since it corresponds to blowing up another complete intersection of
quadrics ($p_2 = q_2 = 0$) contained in $X$.
\end{example}

For the next example we again consider a weak Fano 3--fold $Y$ whose AC model 
is a nodal quartic $X \subset \CP^{4}$ but in this case with the maximal number 
of possible nodes (which is $45$); such an $X$ is unique up to projective equivalence. 
It is a classical fact that $X$ admits projective small resolutions $Y$.
Unlike the previous four examples in which $\rho(Y)=2$ in this case we will show that 
$\rho(Y)=16$ and hence $X$ has defect $\sigma=15$, that is, $X$ contains many 
Weil divisors that are not $\Q$--Cartier;  by a result of Kaloghiros~\cite{kaloghiros:defect}, $15$ is also the maximal possible 
defect for any quartic in $\CP^{4}$ with only terminal singularities.
Because of the high Picard rank of $Y$ the computation of the lattice structure on $H^{2}(Y)$ is considerably 
more involved in this case.
As far as we know the number of distinct projective small resolutions of $Y$ has not been computed in this case.
\begin{example}
  \label{exa:burkhardt_quartic}
  The Burkhardt quartic 3--fold is the hypersurface
\[
X= \bigl(x_0^4-x_0(x_1^3+x_2^3+x_3^3+x_4^3)+3x_1x_2x_3x_4=0\bigr)\subset
\PP^4.
\]
It is well known (and one can verify by inspection), that: $X$ contains $40$ planes,
has $45$ ordinary nodes as singularities, defect $\sigma = 15$, and
several projective small resolutions. (See Finkelnberg's thesis~\cite{finkelnberg89:burk} for these and other facts on the Burkhardt
quartic.) Below we take one such projective small resolution $Y\to X$, and make a
building block $f\co Z\to \PP^1$ by blowing up the (non-singular) base
curve of a general pencil $|S_0,S_\infty|\subset \acls{Y}$.  In what
follows, we establish the following facts about $X$,~$Y$, and $Z$:
 \begin{itemize}
 \item Write $N=H_4(X)$ with the integral quadratic form 
\[
D_1,D_2 \mapsto q(D_1,D_2) = (-K_X) \cdot D_1 \cdot D_2.
\]
 Then: $N$ is a hyperbolic lattice of rank $16$; $N$ is 3--elementary; more
 precisely, the discriminant of $N$ is $(\ZZ/3\ZZ)^5$ (thus $\ell
 =5$); and, finally:
\[
N\cong E_6^\ast (-3) \perp E_8(-1)\perp U \, .
\]
(Here $U$ is the rank 2 hyperbolic lattice, while $E_6^\ast$ is the dual
lattice of the lattice $E_6$. In other
words, if $B$ is the intersection matrix for $E_6$, then $B^{-1}$ is
the intersection matrix for $E_6^\ast$. In particular $B^{-1}$ is not
an integer matrix: it has entries in $\frac{1}{3}\ZZ$; however,
$E_6^\ast (-3)$ is an integral lattice. Since $E_6$ has rank 6 and
discriminant $\ZZ/3\ZZ$, it immediately follows that $E_6^\ast(-3)$
has discriminant $(\ZZ/3\ZZ)^5$, which of course can also be checked
by direct computation.)
 \item The embedding $N\subset L$ in the K3 lattice $L$ is unique, and
  $N^\perp = T = A_2(-1) \perp 2U(3)$, where
  $A_2(-1)$ and $U(3)$ denote the rank 2 lattices with intersection forms
\[ \begin{pmatrix} -2 & 1 \\ 1 & -2 \end{pmatrix}   \textrm{ and }
\begin{pmatrix} 0 & 3 \\ 3 & 0 \end{pmatrix}  . \]
 \item All projective small resolutions $Y\to X$ have $45$
   $(-1,-1)$--curves, $H^2(Y) \cong N$, and $H^3(Y) = (0)$.
 \item Let $f\co Z\to Y$ be the blow up of the base locus of a pencil
   $|S_0,S_\infty| \subset \acls{Y}$ where $S_0$, $S_\infty$ are non-singular
   and meet transversely. Then $Z$ is a building block with
   $H^2(Z)=\ZZ^{17}$, $H^3(Z) = \ZZ^6$.
 \end{itemize}
We now prove all of these claims. 
Todd~\cite{todd36} gives an explicit birational map
$\PP^3\dasharrow X$. Resolving this map by explicit blow ups,
Finkelnberg~\cite{finkelnberg89:burk} constructs a small resolution
$Y\to X$ and a basis of $H_4(Y)$ consisting of planes. His
notation for this basis is:
\[
V,\,E^k_1,\,E^k_2,\,E^k_3,\,E^l_1,\,E^l_2,\,E^l_3,\,E^m_1,\,E^m_2,\,E^m_3,\,F^1_1,\,F^1_2,\,F^2_1,\,F^2_2,\,F^3_1,\,F^3_2
\]
Finkelnberg also makes a list of the curves contracted by $Y\to X$;
using this information, it is not difficult (though tedious)
to write down the matrix of the intersection form on $H_4(Y)$ in this
basis:
\[
\begin{pmatrix}
-2&0 &0 &0 &0 &0 &0 &0 &0 &0 &1 &1 &1 &1 &1 &1  \\
  0&-2&0 &0 &1 &0 &0 &0 &0 &1 &0 &1 &0 &0 &0 &0  \\
  {}&  {}&-2&0 &0 &1 &0 &1 &0 &0 &0 &1 &0 &0 &0 &0  \\
  {}&  {}&  {}&-2&0 &0 &1 &0 &1 &0 &0 &1 &0 &0 &0 &0  \\
  {}&  {}&  {}&  {}&-2&0 &0 &1 &0 &0 &0 &0 &0 &1 &0 &0  \\
  {}&  {}&  {}&  {}&  {}&-2&0 &0 &1 &0 &0 &0 &0 &1 &0 &0  \\
  {}&  {}&  {}&  {}&  {}& {}&-2&0 &0 &1 &0 &0 &0 &1 &0 &0  \\
  {}&  {}&  {}&  {}&  {}&  {}&  {}&-2&0 &0 &0 &0 &0 &0 &0 &1  \\
  {}&  {}&  {}&  {}&  {}&  {}&  {}&  {}&-2&0 &0 &0 &0 &0 &0 &1  \\
  {}&  {}&  {}&  {}&  {}&  {}&  {}&  {}&  {}&-2&0 &0 &0 &0 &0 &1  \\
  {}&  {}&  {}&  {}&  {}&  {}&  {}&  {}&  {}&  {}&-2&1 &0 &0 &0 &0  \\
  {}&  {}&  {}&  {}&  {}&  {}&  {}&  {}&  {}&  {}&  {}&-2&0 &0 &0 &0  \\
  {}&  {}&  {}&  {}&  {}&  {}&  {}&  {}&  {}&  {}&  {}&  {}&-2&1 &0 &0  \\
  {}&  {}&  {}&  {}&  {}&  {}&  {}&  {}&  {}&  {}&  {}&  {}&  {}&-2&0 &0  \\
  {}&  {}&  {}&  {}&  {}&  {}&  {}&  {}&  {}&  {}&  {}&  {}&  {}&  {}&-2&1  \\
  {}&  {}&  {}&  {}&  {}&  {}&  {}&  {}&  {}&  {}&  {}&  {}&  {}&  {}&  {}&-2  \\
\end{pmatrix}
\]
From this it is easy to compute (for example, by computer algebra) that the
discriminant $A\cong (\ZZ/3\ZZ)^5$. Recall that, for $p$ prime, a
lattice is said to be $p$--elementary if the discriminant is of the
form $(\ZZ/p\ZZ)^\ell$: we have just shown that $N$ is $3$--elementary
with $\ell = 5$. Rudakov and Shafarevich~\cite[Section~1, Theorem]{rs81}
prove that an even, hyperbolic (meaning it has signature $(1,r-1)$),
$p$--elementary for $p\neq 2$ prime, lattice of rank $\ge 3$ is uniquely
determined by its discriminant (that is, equivalently, the invariant
$\ell$). This implies that
\[
N\cong E_6^\ast(-3)\perp E_8(-1)\perp U \, .
\]
The proof of \cite[Section~1, Theorem]{rs81} goes through, with the appropriate
small modifications, for lattices of any indefinite signature.
This implies that the transcendental lattice
\[
T=N^\perp \cong A_2(-1) \perp 2U(3) \, ,
\]
as this is the unique lattice with signature $(2,4)$ and
discriminant $(\ZZ/3\ZZ)^5$.
By Dolgachev~\cite[Theorem~1.4.8]{dol:quadratic}, the fact that
$\rank T + \ell(T) + 2\leq \rank L$ implies the primitive embedding $T\subset L$
is unique up to automorphisms, and so the same is true for the embedding
$N \subset L$ (note that $\rank N + \ell(N) + 2 > \rank L$, so
\cite[Theorem~1.4.8]{dol:quadratic} does not apply directly to $N$).

All other assertions are straightforward.

We can compute $\gdiv c_2(Z)$ by evaluating it on the basis of planes.
If the proper transform in $Y$ of some plane remains a plane (that is, no points are blown
up on the plane) then that forces $\gdiv c_2(Z) = 2$ like in
\fullref{exa:quartic_w_plane}. This is easy to arrange. For
instance choose a plane $\Pi \subset X$ and let
$Y^\prime=\Proj \bigl(\bigoplus_{n\geq 0}\oo_X (n\Pi )\bigr)\to X$. Then the
proper transform $\Pi^\prime \subset Y$ is relatively ample and isomorphic to
$\Pi$ and therefore $Y^\prime$ is non-singular in a neighbourhood of $\Pi^\prime$. 
Let next $Y\to Y^\prime$ be a small resolution of $Y^\prime$: this does not alter
$Y^\prime$ in a neighbourhood of $\Pi^\prime$, hence the proper transform of $\Pi$
in $Y$ is still isomorphic to $\PP^2$.
\end{example}

The next pair of examples consider building blocks slightly different from
those above. They are obtained by blowing up the base locus of a
\emph{nongeneric} AC pencil on the simplest smooth Fano 3--fold $Y = \CP^{3}$.

\begin{example}[compare with Kovalev--Lee~{\cite[Example~2.7]{kovalev:lee}}]
  \label{exa:P3_deg}
Instead
  of a generic transverse pencil we consider the 
  pencil $|S_0, S_\infty|\subset |\oo(4)|$, where
  $S_0{=}(x_0x_1x_2x_3{=}0)$ is the sum of the four coordinate planes,
  and $S_\infty$ is a non-singular quartic surface meeting all
  coordinate planes transversely. The base curve of the pencil is the
  union $C=\sum_{i=0}^3\Gamma_i$ of the four non-singular curves
  $\Gamma_i =(x_i=0) \cap S_\infty$. Let $Z$ be obtained from
  $Y=\PP^3$ by blowing up the four base curves one at a time; 
  $Z$ is a non-singular building block containing $e=6\times 4=24 $
  $(-1,-1)$--curves; 
  the blow-up resolves the base locus of the pencil,
  which then defines a (projective) morphism $Z\to \PP^1$. It is clear
  that $H^2(Z)\simeq H^2(\PP^3)\oplus\smash{\bigoplus_{i=0}^3}H^0(\Gamma_i)\simeq
  \ZZ^5$ and, since each $\Gamma_i$ is a curve of genus~3,
  $H^3(Z)\simeq \smash{\bigoplus_{i=0}^3} H^1(\Gamma_i)\simeq \ZZ^{24}$.

  The image of $H^2(Z)$ in $H^2(S)$ equals the image of $H^2(\CP^3)$, that
is,
  it is generated by the hyperplane class. This is because the image of each
  exceptional divisor is just the hyperplane class, so they contribute
  only to the kernel $K$ of $H^2(V) \to H^2(S)$, which is~$\ZZ^3$.
  Since $c_2(\CP^3) + c_1(\CP^3)^2 = 22H^2$ while $-K_{\CP^3}^3 = 64$, it
  follows from \fullref{prop:c2blowup} that $\gdiv c_2(Z) = 2$.
\end{example}

\begin{example}
  \label{ex:2conics}
  We can also consider a pencil of anticanonical divisors in $Y = \CP^3$ where
  \emph{each} divisor is non-generic. Fix a pair of generic plane conics
  $C_1, C_2 \subset \CP^3$. It is easy to see that a generic quartic
  surface $S$ containing both $C_1$ and $C_2$ is non-singular. The
  curves $C_1$, $C_2$ and the hyperplane class generate a lattice
 $N\subset H^2(S)$ with intersection form represented by
  \[
  \left(\begin{array}{rrr}
  -2 &  0 &  2 \\ 
   0 & -2 &  2 \\ 
   2 &  2 &  4 
  \end{array}\right) .
  \]
  We next argue that $N\subset H^2(S)$ is a primitive
  sublattice. Indeed consider the blow up $Y'$ of $\mathbb{P}^3$ along
  $C_1\sqcup C_2$. Since the union $C_1\sqcup C_2$ is cut out scheme
  theoretically by quartics, it follows that the anticanonical linear
  system $\acls{Y'}=|I_{C_1\sqcup C_2}(4)|$ is base point free on $Y'$
  and hence $-K_{Y'}$ is nef. A small calculation gives $-K_{Y'}^3=64-36=28$
  so $-K_{Y'}$ is also big and $Y'$ is a weak Fano \mbox{3--fold}. It is
  clear from the construction that $N$ is the image of $H^2(Y')\to
  H^2(S)$, hence, in particular, $H^2(Y')\to H^2(S)$ is injective (the
  matrix above is non-singular). This implies that $Y'$ is a semi-Fano
  \mbox{3--fold} (any contracted divisor would lie in the kernel). The
  lattice $N\subset H^2(S)$ is therefore primitive by
  \fullref{prop:block_from_weak}.

  Now take a generic pencil of quartic K3s containing both $C_1$ and
  $C_2$. The base locus consists of $C_1$, $C_2$ and a degree 12 curve
  $C_3$ (of genus 15) meeting each of $C_1$ and $C_2$ in 10 points.
  Let $Z$ be obtained by blowing up the $C_i$ in any order, and $S$ the proper
  transform of a smooth element of the pencil. Then $(Z,S)$ is
  a building block. $H^2(Z) \to L$ maps onto $N$, and $H^2(V) \to L$ is
  injective. Regardless of the order of the blow-ups, $\gdiv c_2(Z) = 2$ like
  in the previous example.

  By varying $C_1$, $C_2$ and the pencils, we get three different families of
  blocks, depending on whether we blow up $C_3$ first, second or last.
  By \fullref{thm:fanok3}, a generic $N$--polarised K3 surface $S$ can be
  embedded as an anticanonical divisor in a deformation of $Y'$, and hence as a
  quartic K3 in $\CP^3$ containing a pair of conics. It will therefore occur as
  the K3 fibre of a building block in each of the three families.
\end{example}

The next pair of examples arise from blowing up AC pencils on a
projective small resolution $Y$ of a very particular terminal
Gorenstein \emph{toric} Fano 3--fold $X_{22}$.  $X_{22}$ is a
singular Fano 3--fold with Picard rank $1$, AC degree $22$ and $9$
ODPs.  Every Gorenstein toric Fano variety $X$ has an associated
combinatorial object called a reflexive polytope which determines $X$;
see Chapters 1 and 2 in the thesis of Nill~\cite{nill:thesis} for a
review of basic definitions and facts in toric Fano geometry.  See
also \fullref{S:wk:fano:examples} for a brief overview of basic
properties of toric weak Fano 3--folds in general.

All such $3$d reflexive polytopes and hence all Gorenstein toric Fano
3--folds were classified by Kreuzer--Skarke~\cite{Skarke}.  The terminal toric Fano
3--folds are precisely those reflexive polytopes whose facets are
either standard triangles or standard parallelograms (see \fullref{L:terminal:toric}).  Small resolutions of $X$ are also toric and
their projectivity can be seen in terms of the combinatorics of the
associated reflexive polytope; in particular one can prove that any
Gorenstein toric Fano 3--fold admits at least one projective small
resolution (see \fullref{L:mpcp}).  For the toric Fano
3--fold $X_{22}$ chosen in Examples~\ref{exa:toricV22_a} and~\ref{exa:toricV22_b} one can prove that \emph{all} $512=2^{9}$
possible small resolutions of $X$ are projective (this follows from
the fact, shown below, that the defect $\sigma (X)=9$ is equal to the number of nodes of
$X$); using computer algebra one can show that these $512$ projective small resolutions 
consist of $84$ distinct isomorphism classes of weak Fano
3--folds; see also \fullref{R:picard:rank:toric:small}.

For any toric weak Fano 3--fold $Y$ all odd cohomology groups vanish; 
in particular we never have to worry about the possibility of torsion in  $H^{3}(Y)$ for toric weak Fano 3--folds.
Toric  semi-Fano 3--folds therefore give rise to a very large number of building blocks.

\newcommand{\polytopefig}{
\begin{figure}[ht!]
  \centering
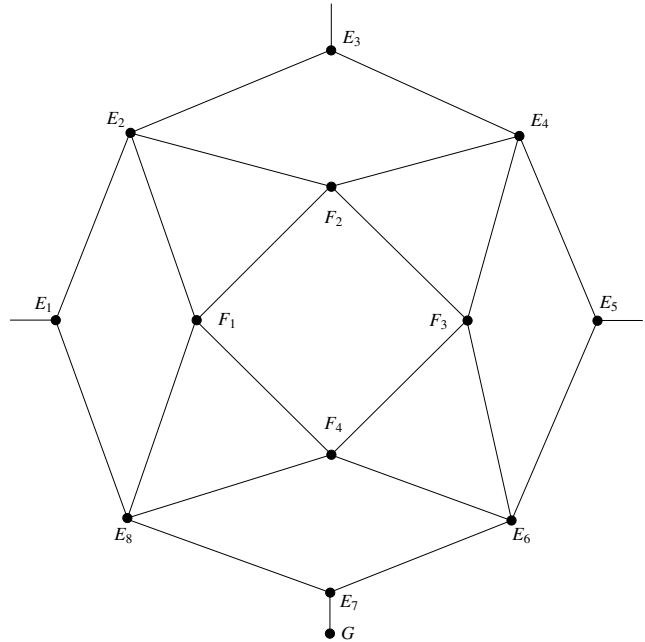
  \caption{The dual graph}
  \label{fig:1}
\end{figure}}

\begin{example}
   \label{exa:toricV22_a} 
   Let $X$ be the terminal Gorenstein toric Fano 3--fold with Fano polytope
   the reflexive polytope in $N=\Hom (\CC^\times, \TT)$ with vertices
\[
\left(\begin{array}{rrrrrrrrrrrrr}
1 & 0 & 0 & -1 & 1 & 1 & -1 & -1 & -1 & 1 & 0 & 0 & 0 \\
0 & 1 & 0 & 1 & 0 & -1 & 1 & 0 & 0 & -1 & 0 & -1 & -1 \\
0 & 0 & 1 & 1 & -1 & 0 & 0 & 1 & 0 & -1 & -1 & 0 & -1
\end{array}\right);
\]
this is polytope 1942
   in the Sage implementation of Kreuze--Skarke's database of 4319 reflexive polytopes in 3 dimensions.
$X$ has Picard rank $1 = \rank H^2(X)$ and $10 = \rank H_4(X)$.  The
polytope can be viewed in Sage (see also below). (Note, incidentally,
that the polytope is self-polar: thus, there is no point in wasting
your efforts trying to determine whether you are working in the fan
picture or its dual: your conclusions will be correct in either case.)
Direct inspection shows that $X$ has a toric projective small
semi-Fano resolution $Y\to X$ with $e=9$ $(-1,-1)$--curves resolving
the ordinary nodes of $X$. Note that the defect $\sigma=e=9$. 
Below we show that $H^2(Y)=\ZZ^{10}$ and
\[
    N=E_8(-1)\perp \langle 8\rangle \perp \langle -16\rangle
\]
and $H^3(Y)=(0)$ since $Y$ is a toric variety. %
The building block $f\co Z\to \PP^1$ is obtained blowing up the base locus
of a pencil $|S_0,S_\infty|\subset \acls{Y}$ where $S_0,S_\infty$ are non-singular
and meet transversely. Below we also denote by $S$ a general member of the
pencil $|S_0,S_\infty|$. The base locus is a non-singular curve $C$ of
genus $12$; hence, $H^3(Z)\simeq H^3(Y)\oplus H^1(C)\simeq \ZZ^{24}$.

Since $\chi(C) = -22$, \fullref{prop:c2blowup} implies that
$\gdiv c_2(Z) = 2$.

\begin{figure}[ht!]
  \centering
\input{graph.pdf_t}
  \caption{The dual graph}
  \label{fig:1}
\end{figure}

Now we calculate the lattice $N$. Inspecting the polytope -- see also
\fullref{fig:1} -- and in particular the boundary surface, shows that,
in $\Pic (Y)$:
\[
S=-K_Y=\sum_{i=1}^9 Q_i + \sum_{j=1}^4\Pi_j
\]
is the union of 9 copies $Q_i$ of $\PP^1\times \PP^1$ and 4 copies
$\Pi_i$ of $\PP^2$, and $-K_{Y|Q_i}\simeq \oo (1,1)$ and
$K_{Y|\Pi_j}\simeq \oo(1)$ so the total degree of the surface is
$2\times 9 + 4=22$, as it should be.
From this it follows that the curves $S\cap Q_i$ and $S\cap \Pi_j$ are
all rational curves, hence they are all curves of self-intersection
$-2$ on $S$ ($S$ is a K3).  These curves meet in a configuration with
a dual graph that looks like \fullref{fig:1}. (The vertices of the
graph correspond to $-2$--curves on $S$/components of the ``boundary''
surface of~$X$/\allowbreak{}vertices of the polytope. Two vertices are connected by
an edge if and only if the corresponding $-2$--curves intersect. The
figure signifies that the vertex $G$ is joined to the vertices $E_1$,
$E_3$, $E_5$, $E_7$.)

Note that the curves $E_1$, $E_2$, $E_3$, $E_4$, $F_1$, $F_4$, $E_6$, $E_7$
generate a sublattice of type $E_8(-1)$.
Since $E_8(-1)$ is unimodular, it follows that $N=E_8(-1)\perp
(E_8(-1)^\perp)$, where $E_8(-1)^\perp$ is a lattice of rank 2. Our next
task is to compute $E_8(-1)^\perp$. Looking at elliptic fibrations on $S$ we
discover the following relations in $N=\Pic (S)$:
\begin{align*}
  2G+E_1+E_3+E_5+E_7&=F_1+F_2+ F_3+F_4\\
  E_1+E_2+F_1+E_8      &=F_3+E_4+E_5+E_6\\
  F_1+F_2+E_2             &=G+E_5+E_6+E_7
\end{align*}
The first of these, for instance, is obtained from an elliptic
fibration with fibres ${2G+E_1+E_3+E_5+E_7}$ (a fibre of type
$\wtilde{D}_4$) and $F_1+F_2+F_3+F_4$ (a fibre of type~$\wtilde{A}_3$).
The other two relations are obtained similarly. To
find a basis for $E_8(-1)^\perp$, we look at these relations
modulo~$E_8(-1)$:
\[
\begin{array}{ccccccc}
  E_5 &       &-F_2 & -F_3 & +2G & \equiv &  0 \\
  -E_5&+E_8&       &-F_3 &         & \equiv &  0\\
  -E_5&       &+F_2&       & -G     & \equiv & 0
\end{array}
\quad \mod E_8(-1)
\]
It is immediate from these relations that $E_5$, $F_2$ is a basis of
$N\;\mod E_8(-1)$. It is easy to check that the vectors
$$
\begin{array}{rrrrrrrrrr}
& E_5 & +8E_4 & +15E_3 & +22E_2 & +18F_1 & +14F_4 & +10E_6 & +5E_7 & +11E_1, \\
F_2 & & +22E_4 & +43E_3 & +64E_2 & +52F_1 & +39F_4 & +26E_6 & +13E_7 & +32E_1
\end{array}
$$
are perpendicular to $E_8(-1)$ (for instance by computing 16 inner
products); thus, by what has been said, they form a basis of
$E_8(-1)^\perp$. In this basis the intersection matrix is computed to be
\[
\begin{pmatrix}
  16&48\\48&136
\end{pmatrix}.
\]
and a small change of coordinates then puts this in the form $\langle 8
\rangle \perp \langle -16\rangle$.
\end{example}

\newcommand{\blocktable}{
\begin{table}[t]
\centering
\addtolength{\tabcolsep}{-2pt}
\begin{tabular}{lccccccc} \toprule
What & $-K_Y^3$ & $H^2(Z)$ & $N$ & $K$ & $H^3(Z)$ & $\gdiv c_2(Z)$ &
$e$ \\ \midrule
\fullref{exa:quartic_w_plane} & $4$ & $\ZZ^3$
& $\begin{pmatrix}-2&1\\1&4\end{pmatrix}$ & $(0)$ & $\ZZ^{50}$
& $2,4$ & $9$ \\[0.5em]
\fullref{exa:quartic_w_quadric} & $4$ & $\ZZ^3$
& $\begin{pmatrix}-2&2\\2&4\end{pmatrix}$ & $(0)$ & $\ZZ^{44}$
& $2$ & $12$ \\[0.5em]
\fullref{exa:quartic_w_scroll}     & $4$ & $\ZZ^3$
& $\begin{pmatrix}-2&3\\3&4\end{pmatrix}$ & $(0)$ & $\ZZ^{34}$
& $2,4$ & $17$ \\[0.5em]
\fullref{exa:quartic_w_22}          & $4$ & $\ZZ^3$
& $\begin{pmatrix}0&4\\4&4\end{pmatrix}$ & $(0)$ & $\ZZ^{36}$
& $4$ & $16$ \\[0.5em]
\fullref{exa:burkhardt_quartic}   & $4$ & $\ZZ^{17}$
& $E_6^\ast (-3){\perp} E_8(-1){\perp} U$ & $(0)$ & $\ZZ^6$
& $2$ & $45$ \\[0.5em]
\fullref{exa:P3_deg}                  & $64$ & $\ZZ^{5}$
& $\langle 4\rangle$ & $\ZZ^3$ & $\ZZ^{24}$ & $2$ & $24$ \\[0.5em]
\fullref{ex:2conics} & $64$ & $\ZZ^4$ &
$\begin{pmatrix}-2 & 0 & 2 \\ 0 & -2 & 2 \\ 2 & 2 & 4\end{pmatrix}$ &
$(0)$ & $\ZZ^{30}$ & $2$ & $20$ \\[0.5em]
\fullref{exa:toricV22_a}           & $22$ & $\ZZ^{11}$
& $E_8(-1){\perp} \langle 8\rangle {\perp} \langle -16 \rangle$ & $(0)$ &
$\ZZ^{24}$
& $2$ & $9$ \\[0.5em]
\fullref{exa:toricV22_b}          & $22$ & $\ZZ^{23}$
& $E_8(-1){\perp} \langle 8\rangle {\perp} \langle -16 \rangle$ &
$\ZZ^{12}$ & $(0)$
& $2$ & $33$ \\[0.5em]
\fullref{ex:doubleline}          & $4$ & $\ZZ^3$
& $\gen{4} {\perp} \gen{-2}$ & $(0)$ & $\ZZ^{46}$
& $2$ & $12$ \\ \bottomrule
\end{tabular}
\addtolength{\tabcolsep}{2pt}
\medskip
\caption{A small number of examples of building blocks}
\label{table:blocks}
\end{table}}

\begin{example}
  \label{exa:toricV22_b}
  In this example, $Y\to X$ is the same as in the previous
  \fullref{exa:toricV22_a}, but we construct the building block
  $Z$ by blowing up a different pencil.  Indeed, let us choose the
  more interesting pencil $|S_0, S_\infty|\subset |\oo(4)|$, where
  $S_0=\sum_{i=1}^9Q_i+\sum_{j=1}^{4}\Pi_j$ is the toric
  boundary surface of $Y$ and $S_\infty$ is a non-singular element of
  $\acls{Y}$ meeting all the components of $S_0$ transversely. The base
  curve of the pencil is the union
  $C=\sum_{i=1}^9\Gamma_i+\sum_{j=1}^4 G_j$ of 13 non-singular rational
  curves. Let $Z$ be obtained from $Y$ by blowing up the 13 curves one
  at a time; $Z$ is a non-singular building  block containing $e=9+
  24=33$ $(-1,-1)$--curves: 9~that were present in~$Y$, and $24$ from the
  intersection points of $C$ (corresponding to edges in \fullref{fig:1}).
  The blow-up resolves the base locus of the pencil, which then defines a
  (projective) morphism $Z\to \PP^1$.
  It is clear that $H^2(Z)\simeq H^2(Y)\oplus \ZZ^{13}\simeq \ZZ^{23}$
  and $H^3(Z)=(0)$.

  From \fullref{prop:c2blowup}, $\gdiv c_2(Z) \mid (-K_Y^3) = 22$.
  Since also $\gdiv c_2(Z) \mid 24$, it must be~2.
\end{example}

\begin{example}
\label{ex:doubleline}
Now we give an example using a semi-Fano 3--fold whose anti\-canonical morphism
is not small, but contracts a divisor to a curve. Let $X \subset \CP^4$ be
defined by
\[ \sum_{0 \leq i \leq j \leq 2} X_i X_j Q_{ij} , \]
where $Q_{ij}$ are homogeneous quadrics. This is the general form of a
quartic ``containing a double line'' $\ell = \{X_0 = X_1 = X_2 = 0\}$.
For generic $Q_{ij}$, the sextic polynomial $\det (Q_{ij})$ on $\ell$ has
simple zeros, and the blow-up $Y$ of $X$ at $\ell$ is smooth,
compare with Conte--Murre~\cite[Lemma~1.15]{conte77}. Then $Y$ is a crepant resolution of $X$, and
$X$ is the anticanonical model of $Y$. In particular, $Y$ is semi-Fano.

A generic hyperplane section $S$ of $Y$ is the resolution of a quartic K3 with
a single node, so $N = \Pic S = \gen{4} \perp \gen{-2}$.

\blocktable

To understand more about the topology of $Y$, consider it as the proper
transform of $X$ in~$\amb$, the blow-up of $\CP^4$ in $\ell$. Thinking of
$\amb$ as the union of all planes containing $\ell$ identifies it with the
total space of a $\CP^2$--bundle over $\CP^2$. To be precise,
$\pi \co  \amb \to \CP^2$ is the projectivisation of
$V = 2\oo \oplus \oo(-1)$ on $\CP^2$.
Let $T$ be the associated tautological bundle on $\amb$, and $F = \pi^*\oo(1)$.
The exceptional divisor of $G \to \CP^4$ is $E = -T-F$, while the tautological
bundle on $\CP^4$ pulls back to $T$. Therefore $Y$ is a section of
$-4T - 2E = -2T + 2F$ (so $Y$ is a conic bundle over~$\CP^2$).
This is an ample class ($-T$~and $F$ span the nef cone of $\amb$), so
$H^3(Y)$ is torsion-free by the Lefschetz theorem, and we can apply
\fullref{prop:block_from_weak} to get a building block $Z$.

To compute characteristic classes of $Y$, first note that as a complex
vector bundle $T\amb = T_{vert}\amb \oplus \pi^*T\CP^2$, which is stably
isomorphic to $(T^{-1} \otimes \pi^*V) \oplus \pi^*(3\oo(1)) =
2T^{-1} \oplus T^{-1}F^{-1} \oplus 3F$. Therefore the total Chern class of
$Y$ is
\begin{align*}
c(Y) &= (1-T)^2(1-T-F)(1+F)^3\frac{1}{1 - 2T + 2F} \\
&= 1 - T + T^2 - 5FT + T^3 - 4FT^2 + 7F^2T,
\end{align*}
where the addition and multiplication are now in the cohomology ring, and
we use that $F^3 = 0$ in~$H^6(\amb;\ZZ)$.
By interpreting $T^2$ as the class of the section $\bbp(\oo(-1))$ of
$G = \bbp(2\oo \oplus \oo(-1))$, we see that $T^4 = -FT^3 = F^2T^2 = [\amb]$.
Hence
\[ \chi(Y) = \int_Y c_3(Y) = (T^3 - 4FT^2 + 7F^2T)(-2T + 2F) = -34, \]
so $b^3(Y) = 40$. Similarly we find that $c_2(Y) + c_1(Y)^2 = 2T^2 - 5FT$
evaluates to $-28$ on $T$ (as it should, since $T = K_Y$) and 18 on $F$. Hence
$b^3(Z) = 46$, and $\gdiv c_2(Z) = 2$.

The exceptional set of $Y \to X$ is a conic bundle over $\ell$ with 6
degenerate fibres.
Each degenerate fibre consists of two $\CP^1$s intersecting in a single point.
These 12 $\CP^1$s have normal bundle $\oo(-1) \oplus \oo(-1)$.
\end{example}

\section{Weak Fano 3--folds: further examples and partial classification results}
\label{S:wk:fano:examples}

In this section we give some further examples of weak Fano 3--folds. Our aim
is to back up our statement that there are \emph{many} more non-singular weak Fano
3--folds than non-singular Fano 3--folds. We will not be too systematic since weak Fano
3--folds are far from being classified.

Any weak Fano $Y$ for which $-K_{Y}$ is big and nef but not ample has
$$\rho(Y) = \rank{\Pic(Y)} \ge 2;$$
thus weak Fano 3--folds with
$\rho=2$ are the simplest class of weak Fano 3--folds that are not actually Fano 3--folds. 

Examples~\ref{exa:quartic_w_plane} to~\ref{exa:quartic_w_22} already
gave a small number of semi-Fano 3--folds with Picard rank $\rho=2$:
all of anticanonical degree $4$ obtained by a small resolution of a
(sufficiently generic) quartic containing a special surface; Examples~\ref{E:nodal:quadric} and~\ref{E:wk:fano:ruled} are \emph{toric} weak
Fano 3--folds with $\rho=2$.  As we will discuss below there are many
other weak Fano 3--folds that generalise both classes of examples:
toric or $\rho=2$.

For the purposes of the differential geometry of ACyl Calabi--Yau 3--folds,
we are interested in the classification of weak Fano 3--folds up to
\emph{deformation}. In \fullref{ex:ky:nonample}(i) we considered how
the semi-Fano $\FF_2 \times \PP^1$ can deform to the rigid Fano
$\PP^1 \times \PP^1 \times \PP^1$; these are different varieties from the
algebraic point of view as one is Fano and the other is not, but the ACyl
Calabi--Yau 3--folds we construct from them using \fullref{prop:onestage}
are deformation-equivalent.
We will not discuss the problem of classifying weak Fano 3--folds up to
deformation in depth. We use the deformation properties of extremal
contractions to distinguish between many of the rank two examples we describe.
For the toric examples, we determine whether they are rigid (as complex
manifolds) in order to get a crude lower bound on the number of deformation
classes.

\subsection*{Weak Fano 3--folds with Picard rank $\rho=2$: classical examples}

We now exhibit some further concrete examples of weak Fano 3--folds with Picard rank $\rho=2$.
These weak Fanos were studied initially because of their connection to so-called \emph{elementary rational maps} 
between rank one Fano 3--folds, for example, see Iskovskih--Prokhorov's book~\cite[Section~4.1]{fano:varieties}
to which we refer the reader for further details and references. 
All these rank two weak Fano 3--folds 
arise as blowups in points or in low degree curves in rank one Fano 3--folds.

\subsubsection*{Rank two semi-Fano 3--folds from smooth blowups of Fano 3--folds}
We have seen that one way to obtain smooth weak Fano 3--folds is to look for 
projective small (respectively crepant) resolutions of Gorenstein terminal (respectively canonical) Fano 3--folds, 
but often it is difficult to determine if a projective small (respectively crepant) resolution exists.
Another potential way to obtain weak Fano 3--folds is to realise them as smooth blowups 
of other simpler 3--folds. 

In this direction we have the following result of Fujino--Gongyo~\cite[Theorem~4.5]{fujino}
generalising the analogous result by Koll\'ar--Mori~\cite[Corollary~2.9]{MR1158625} in the Fano setting.
\begin{theorem}
\label{t:wk:fano:morphisms}
Let $f\co Y \to W$ be a smooth projective morphism between smooth projective varieties. 
If $Y$ is weak Fano (respectively Fano) then $W$ is also weak Fano (respectively Fano).
\end{theorem}
In particular if $Y$ is weak Fano 3--fold with Picard rank $\rho=2$ and $f \co Y \to W$ 
is the inverse of the blowup of a smooth point or curve, then $W$ is a weak Fano 3--fold with 
$\rho(W)=1$; but since $\rho(W)=1$ this forces $W$ to be Fano not just weak Fano.
In other words, to find rank two weak Fano 3--folds we ought to consider smooth blowups of smooth
rank one Fano 3--folds; in fact we will see below -- in our discussion of the classification scheme 
for rank two weak Fano 3--folds -- that the majority of all rank two weak Fano 3--folds 
arise this way. The particular rank two weak Fano 3--folds that arise 
as blowups of smooth rank one Fano 3--folds in low degree 
curves $C$ -- that is, lines, conics, and rational normal cubics -- have been known at least since the late 1980s 
and in some special cases since the late 1970s: see 
Iskovskih--Prokhorov \mbox{\cite[Sections~4.3--4.6]{fano:varieties}}
for further details and references.

We will use several times the following well-known result on the behaviour of the canonical class of a smooth threefold
under blow-up of a smooth curve or a point.
\begin{lemma}
\label{l:blowup:can}
Let $C \subset W$ be a smooth curve of genus $g(C)$ in a smooth threefold $W$, let $\pi \co Y \to W$
be the blowup of $C$ and let $E$ denote the exceptional divisor of $\pi$. Then 
\begin{align*}
(-K_Y)^3 &= (-K_W)^3 + 2K_W \cdot C - 2 + 2g(C);\\
(-K_Y)^2 \cdot E &= -K_W \cdot C + 2 -2g(C);\\
-K_Y \cdot E^2 &= 2g(C)-2;\\
E^3 &= K_W \cdot C + 2 -2g(C).
\end{align*}
Let $Y$ be the blowup of a smooth threefold $W$ in a point. Then
\begin{align*}
(-K_Y)^3 &= (-K_W)^3 -8; \\
(-K_Y)^2 \cdot E &=  4;\\
-K_Y \cdot E^2 &= -2;\\
E^3 &=1.
\end{align*}
\end{lemma}
\begin{proof}
The result follows from the fact that $K_Y = \pi^*K_W +E$: see Blanc--Lamy
\mbox{\cite[Lemma~2.4]{blanc:lamy}} for details.
\end{proof}
\begin{remark}
\label{r:fano:blowup}
If a weak Fano 3--fold $Y$ arises as the blowup of a smooth curve in a smooth rank one Fano 3--fold $W$ 
then  by  \fullref{lem:top_of_blowup}
$H^*(Y)$ is torsion-free since the cohomology of $W$ is torsion-free. 
Therefore whenever $Y$ is semi-Fano (recall
\fullref{prop:block_from_weak}(iv))
we can obtain building blocks $Z$ satisfying \fullref{dfn:BLOCK}
from $Y$ by blowing up the base locus of a generic AC pencil. 

\fullref{l:blowup:can} allows us to compute the lattice structure on $\Pic(Y)$ 
and \fullref{lem:top_of_blowup} the Betti numbers of $Y$ from those of $W$ and the genus of the curve $g(C)$.
We can also understand $c_{2}(Y)$ and therefore $c_{2}(Z)$ for the associated building block $Z$
by using the behaviour of $c_{2}$ under smooth blowups. Therefore we can obtain all the topological 
information we need about building blocks that arise this way with relatively little work.
\end{remark}

Recall from the Iskovskih classification of smooth rank 1 Fano 3--folds that there are $17$ families of examples: 
$\CP^{3}$, the quadric $Q \subset \CP^{4}$, the del Pezzo (that is, Fano with index 2) 3--folds $V_{1}, \ldots ,V_{5}$ 
and $10$ index one Fanos 3--folds $V_{2g-2}$ with genus $g \in \{ 2, \ldots ,10, \,12\}$. 
We shall concentrate on weak Fano 3--folds obtained by blowing up curves in index one rank one Fano 3--folds.

If $W$ is a rank $1$ Fano 3--fold of index $1$ and genus $g$
and $C \subset W$ is a smooth curve of degree $\deg{C}:= -K_{W}\cdot C$ and genus $g(C)$ 
and $Y=Bl_{C}(W)$ then \fullref{l:blowup:can} specialises to yield
\begin{equation}
\label{E:ac:degree:blowup}
-K_{Y}^{3} =2g'-2, \quad \text{where}\quad g'= g + g(C) - \deg{C} -1.
\end{equation}
In particular if $C$ is a line, quadric or rational normal cubic then $g'=g-2$, $g'=g-3$ or $g'=g-4$ respectively.
If  $E \subset Y$ denotes the exceptional divisor of the blowup then 
the Picard lattice of $Y$ is generated by $-K_{Y}$ and $E$.
The lattice structure induced on $Y$ is also determined by the information in 
\fullref{l:blowup:can}.
In particular, if $Y=Bl_{C}(W)$ is the blowup of a line, conic or rational normal cubic then with respect to the 
basis $E$, $A=-K_{Y}$ of $\Pic(Y)$ the quadratic form $-K_{Y} \cdot D_{1} \cdot D_{2}$  is 
\[\begin{pmatrix}
    -2 & 3    \\
     3 & 2(g-3) 
  \end{pmatrix},\qquad
\begin{pmatrix}
    -2 & 4   \\
     4 & 2(g-4) 
  \end{pmatrix},\qquad
\begin{pmatrix}
    -2 & 5   \\
     5 & 2(g-5) 
  \end{pmatrix},\] 
respectively. 

To ensure that $Y$ is a rank $2$  weak Fano 3--fold one needs to ensure that
$-K_{Y}$ is big and nef. 
As soon as one shows that $-K_{Y}$ is nef then for bigness we need only show $-K_{Y}^{3}>0$ and this can be checked 
immediately from \eqref{E:ac:degree:blowup}.
One also needs  to ensure the existence of lines, conics and rational normal cubics on the appropriate rank one Fano 3--folds.
To show that the AC morphism is small one also needs to know that it contracts only a finite number of curves. 

\subsubsection*{Blowups of lines}
Iskovskih--Prokhorov~\cite[Proposition~4.3.1]{fano:varieties} shows that for every line $C$ on an anticanonically embedded 
rank one Fano 3--fold $W$ of genus $g \ge5$ the blowup $Y=Bl_{C}(W)$ is a rank two semi-Fano  
of genus $g'=g-2$ with small AC morphism; moreover the fibres of the AC morphism are all $\CP^{1}$s and they can be understood 
in terms of the geometry of $W$, for example, the generic fibre type is any curve $F \subset Y$ whose 
proper transform in $W$ intersects the chosen line $C \subset W$.
In particular by blowing up any line on a rank one Fano 3--fold $W$ of genus $g=5$ we get a rank two 
semi-Fano 3--fold $Y$ of genus $g'=3$ with small AC morphism and quadratic form given in the basis $E$ and $-K_{Y}$ by 
\[\begin{pmatrix}
    -2 & 3\\
     3 & 4
  \end{pmatrix}.
  \]
This is the same quadratic form that appeared in \fullref{exa:quartic_w_scroll}: 
the small resolution of a general quartic containing a cubic scroll surface. Indeed the rank 
two semi-Fano $Y$ we have constructed is a projective small resolution of such a nodal quartic 
3--fold; see also Entry~30 in Kaloghiros~\cite[Table~1]{kaloghiros:thesis}. 
We have similar rank two semi-Fano 3--folds of genus $4, 5,6,7,8, 10$ 
by blowing up lines on rank one Fano 3--folds of genus $6, 7,8,9, 10, 12$ respectively.

\subsubsection*{Blowups of conics}
If $Y=Bl_{C}(W)$ is the blowup of any smooth conic $C$ on an anticanonically embedded rank one 
Fano 3--fold of genus $g$ then by Iskovskih--Prokhorov~\cite[4.4.3]{fano:varieties} $Y$ is a weak Fano 3--fold of genus $g'=g-3$ for 
$g \ge 5$. Furthermore, if $g \ge 7$ then $Y$ is a semi-Fano 3--fold with small AC morphism 
for any sufficiently generic conic in $W$  and if $g \ge 9$ the same holds for all conics. 
If we take $g=6$ then $Y$ is a weak Fano 3--fold of genus $3$, that is, its AC model is a terminal quartic 3--fold:
see also Kaloghiros \mbox{\cite[Table~1, number~25]{kaloghiros:thesis}}.
Its quadratic form in the basis $E$, $-K_{Y}$ is 
\[\begin{pmatrix}
    -2 & 4\\
     4 & 4
  \end{pmatrix}.
  \]
  We have similar rank two semi-Fano 3--folds of genus $4, 5, 6, 7, 9$
  by blowing up sufficiently generic conics on rank one Fano 3--folds of genus $7, 8, 9, 10, 12$ respectively.

\subsubsection*{Blowups of points}
If $Y=Bl_{P}(W)$ is the blowup of a point $P$ not lying on a line in
$W$ (such points exist: see~\cite[4.2.2]{fano:varieties}) on an
anticanonically embedded rank one Fano 3--fold of genus $g \ge 6$,
then $Y$ is a rank two weak Fano 3--fold of genus $g'=g-4$, moreover
for a sufficiently general point $P$ the AC morphism of $Y$ is small~\cite[4.5.1]{fano:varieties}.  If we take $g=7$ then $Y$ is a weak
Fano 3--fold of genus $3$, that is, its AC model is a terminal quartic
3--fold: see also~\cite[Table~1, number~24]{kaloghiros:thesis}. Its
quadratic form in the basis $E$, $-K_Y$ is
\[\begin{pmatrix}
    -2 & 4\\
     4 & 4
  \end{pmatrix}
  \]
  which is the same lattice which arose above by considering the
  blowup of a genus $6$ Fano 3--fold in a sufficiently generic
  conic. This pair of rank two weak Fano 3--folds with the same
  Picard lattice structure are not deformation-equivalent -- for
  instance because they have extremal contractions of different types,
  compare with Mori \mbox{\cite[Theorem~3.47]{mori:1982}}.
  We have similar rank two semi-Fano
  3--folds of genus $4, 5, 6, 10$ by blowing up sufficiently generic
  conics on rank one Fano 3--folds of genus $8, 9, 10, 12$
  respectively.

\subsubsection*{Blowups of rational normal cubics}
A rank one Fano 3--fold $V_{2g-2}$ of $g \ge 5$ which contains a line
and a conic also contains a rational normal cubic~\cite[4.6.1]{fano:varieties}. So we can also consider blowups along
rational normal cubics.  If $Y=Bl_{C}(W)$ is the blowup of any
(respectively a sufficiently general) rational normal cubic on a rank
one Fano 3--fold of genus $g \ge 7$ (respectively $g\ge 6)$, then $Y$
is a rank two weak Fano 3--fold of genus $g'=g-4$~\cite[4.6.2]{fano:varieties}. If we take $g=7$ then $Y$ is a weak Fano
3--fold of genus $3$, that is, its AC model is a Gorenstein at worst
canonical quartic 3--fold.  Its quadratic form in the basis
$E$, $-K_{Y}$ is
\[\begin{pmatrix}
    -2 & 5\\
     5 & 4
  \end{pmatrix}.
  \]

\subsection*{The classification scheme for rank two weak Fano 3--folds}
In the Mori--Mukai classification there are 36 families of non-singular Fano 3--folds with
$\rho=2$: see Iskovskih--Prokhorov~\cite[Table~12.3]{fano:varieties} for the list.
The classification of  non-singular weak Fano 3--folds with $\rho=2$ was initiated recently by 
Jahnke--Peternell--Radloff~\cite{peternell1,peternell2}
with subsequent contributions by Takeuchi~\cite{takeuchi},  
Cutrone--Marshburn~\cite{cutrone:marshburn}, Arap--Cutrone--Marshburn~\cite{arap:c:m} 
and Blanc--Lamy~\cite{blanc:lamy}; 
see also related work by Kaloghiros~\cite{kaloghiros:thesis,kaloghiros:defect}.
The classification is not yet complete, but already more than $200$
families of rank two weak Fano 3--folds are known (with around 
$50$ further cases still to be settled).
Below we summarise the basic strategy of this classification scheme and some of the main results 
obtained; we refer the reader to the references above for further details.

Throughout the rest of this section $Y$ will denote a non-singular weak Fano
3--fold of rank $2$ and $\varphi\co Y \ra X$ its anticanonical
morphism.  In general the anticanonical model $X$ of a rank two weak Fano 3--fold $Y$ is a
Gorenstein canonical Fano 3--fold with $\rho(X)=1$ 
whose anticanonical degree is the same as that of  $Y$. 
There are two main classes:
\begin{enumerate}
\item the anticanonical morphism $\varphi\co Y \ra X$ is divisorial,
  that is, it contracts a divisor. In this case, $X$ is a Gorenstein Fano
  3--fold with canonical non-terminal singularities, $\rho(X) =1$ and
  $\sigma(X) = 0$. (In the vast majority of cases we will see that $\varphi$ 
  is semi-small, so that $Y$ is a semi-Fano 3--fold in the sense of \fullref{d:semifano}).
\item the anticanonical morphism $\varphi\co Y \ra X$ is small. 
$X$ is a \mbox{non--$\Q$--factorial} Gorenstein Fano 3--fold with
  terminal singularities, $\rho (X) =1$ and $\sigma (X)=1$. 
  (In many of these cases $X$ has only ordinary double points).
\end{enumerate}

Recall that in the classification of non-singular rank $2$ Fano 3--folds a
fundamental role is played by the two different Mori contractions that
any such 3--fold admits.  By Mori's classification of non-singular 3--fold
extremal rays (\fullref{T:Mori:3:contract}) the possible
contractions are completely understood and fall into three basic
classes: type C (conic bundle type), D (del Pezzo fibre type) and E
(exceptional/divisorial) type.  
For example if a Fano 3--fold with $\rho=2$ admits an extremal 
contraction of type E1, that is, the inverse of the blowup of a non-singular curve 
in a smooth 3--fold $W$,  then by \fullref{t:wk:fano:morphisms} $W$ itself 
must be a smooth Fano 3--fold with $\rho(W)=1$. In general the existence of two extremal rays of
known type together with the condition of being Fano put severe
constraints on the 3--fold, enough to allow a complete
classification.

For non-singular rank $2$ \emph{weak} Fano 3--folds there is only a single
Mori contraction $\psi\co Y \ra W$.  A substitute for the missing second
extremal ray is provided by the AC morphism $\varphi\co Y \ra X$.  When
the anticanonical morphism contracts a divisor, an almost complete
classification was given recently by Jahnke--Peternell--Radloff~\cite{peternell1}. When the 
anticanonical morphism is small the analysis is more involved and
the classification is not yet close to complete. Nevertheless, as we will describe below, 
many examples are known and there are classification results under additional assumptions.

\subsubsection*{Rank $2$ weak Fanos with divisorial AC morphism}
  \label{sec:rank-2-weak}

  In~\cite{peternell1}, Jahnke, Peternell and
Radloff classify rank two weak Fano 3--folds
  of type~(i) -- where the AC morphism $\varphi$ is a divisorial
  contraction -- according to the type of the Mori contraction
  $\psi\co Y \ra W$; see~\cite[Tables A.2--A.5, pages 627--630]{peternell1}.  There are at most $59$ deformation
  families (the existence of two possible families A.2.7, A.2.8
  remains to be shown) with (even) anticanonical degrees $-K_{Y}^{3}$
  between $2$ and $72$; because of the length and complexity of the
  classification we do not reproduce it here. A key technical role is
  played by Mukai's classification~\cite{mukai:95} of all Gorenstein
  Fano threefolds with canonical singularities such that the
  anticanonical divisor does not admit a moving decomposition: see~\cite[4.7]{peternell1}.

  One important fact to note from the classification is that rank two
  weak Fano 3--folds that are \emph{not} semi-Fano are extremely rare; when
  the extremal ray is of type D or E2--5 one can show that $Y$ is
  always semi-Fano~\cite[2.3 and~5.2]{peternell1} and there is a single
  exception out of 25 cases with an extremal ray of type E1~\cite[Section~4]{peternell1}.  
  Altogether only in  four (A.3.1, 3.9, 3.12 and 4.25 in~\cite{peternell1}) out of the $59$ families does the AC morphism
  contract a divisor to a point. Hence we have $53 (+2?)$ non-singular rank
  $2$ semi-Fano 3--folds for which the anticanonical morphism
  $\varphi\co Y \ra X$ contracts a divisor $D$ to a curve $B$. In all
  such cases $B\subset X$ is a non-singular curve of cDV singularities, $Y$
  is the blowup of $X$ in the curve $B$ and $D$ is a conic bundle over
  $B$ \mbox{\cite[1.8]{peternell1}}.

\subsubsection*{Rank $2$ weak Fanos with small AC morphism}
\label{sec:rank-2-weak-1}
As mentioned above, the classification of rank two semi-Fano 3--folds whose AC
morphism $\varphi$ is small is more involved and not yet complete,
despite recent activity in this direction by several authors.
In
this case the anticanonical model $X$ is a non--$\Q$--factorial
Gorenstein Fano 3--fold with terminal singularities, which by
\fullref{P:cdv:terminal} are isolated cDV singularities; 
in many cases $X$ has only ordinary double points.  

By Namikawa's smoothing result (\fullref{T:fano:smooth}) $X$ admits a
smoothing \mbox{$\mathcal{X} \ra \Delta\subset \C $} such that $\mathcal{X}_{0}
\simeq X$ and $\mathcal{X}_{t}$ for $ t\neq 0$ is a non-singular Fano 3--fold (this is not
always true in the case of Gorenstein canonical singularities);
moreover, the Picard groups (over $\Z$) of $X$ and the general
$\mathcal{X}_{t}$ are isomorphic.  Hence $X$ and $\mathcal{X}_{t}$
have the same Fano index and $\mathcal{X}_{t}$ is a non-singular Fano
3--fold of Picard rank~$1$.  The cases where $X$ has index $>1$ are
relatively straightforward -- see Jahnke--Peternell--Radloff~\cite[2.12--3]{peternell2} -- and the
main case is when $X$ has index $1$. In this case the Iskovskih
classification of rank $1$ non-singular Fano 3--folds
(see~\cite[Table~12.2]{fano:varieties}
for a convenient list or see our \fullref{table:c2} in 
\fullref{sec:examples}) implies $2 \le -K_{Y}^{3} \le 22$ with $-K_{Y}^{3}
\neq 20$ and, in fact, all such possible anticanonical degrees
actually occur.

The anticanonical morphism $\varphi \co Y\to X$ is a flopping
contraction (recall \fullref{D:flopping}) and thus it can be flopped (recall \fullref{T:flops}); 
that is, there is another non-singular rank 2 weak Fano 3--fold $Y^{+}$, whose (small) anticanonical
morphism we denote \mbox{$\varphi^{+}\co Y^{+} \ra X$}.  $Y^{+}$
has the same anticanonical degree as $Y$.  For any divisor $D$ on $Y$
let $D^{+}$ denote the strict transform of $D$ under the flop $\chi$;
the map $D \ra D^{+}$ induces an isomorphism between the Picard groups
of $Y$ and $Y^{+}$.  Moreover, the lattice structures induced on the
Picard lattices of $Y$ and $Y^{+}$ are isomorphic, that is,  for any
divisors $D_{1}$ and $D_{2}$ on $Y$ we have
\begin{equation}
\label{E:pic:lattice}
-K_{Y}\cdot D_{1} \cdot D_{2} = -K_{Y^{+}} \cdot D_{1}^{+} \cdot D_{2}^{+}.
\end{equation}
$Y^{+}$ also admits a ($K_{Y^{+}}$--negative) extremal contraction
$\psi^{+}\co Y^{+} \ra W^{+}$.  Everything fits into the following
diagram:
\begin{equation}
\label{F:link}
\xymatrix{Y \ar@{-->}[rr]^{\chi} \ar[d]^{\psi} \ar[dr]^{\varphi}& & Y^{+} \ar[d]^{\psi^{+}} \ar[dl]_{\varphi^{+}} \\
          W & X & W^{+}}
\end{equation}

The classification programme has two steps: a \emph{numerical classification}
stage and the more delicate \emph{geometric realisability} question.  In the
numerical classification stage one first writes down a system of
Diophantine equations determined by the relations among various
intersection numbers that any non-singular weak Fano 3--fold of rank $2$
with small AC morphism would have to satisfy.
The precise form of these Diophantine equations depends on the pair of Mori contractions
$\psi$ and $\psi^{+}$; as a result there are various subcases
depending on the type of the pair of Mori contractions.  
See Cutrone--Marshburn~\cite[Section~2.1]{cutrone:marshburn} -- particularly equations (2.6) and (2.7) therein -- for
the Diophantine equations in the case where both Mori contractions 
$\psi$ and $\psi^{+}$ are of type~E1; 
in this latter case both $Y$ and its flop $Y^{+}$ 
arise as the blowups of smooth curves $C$ and $C^{+}$ in 
rank one Fano 3--folds $W$ and $W^{+}$.
So rank two weak Fanos with link type E1--E1 constitute a direct 
generalisation of the concrete rank two weak Fanos constructed 
in the previous subsection as blowups of low degree curves in 
non-singular index one rank one Fano 3--folds.

A solution of
the numerical classification problem means a finite list of all possible
solutions to these Diophantine equations. Each such solution is
referred to as a \emph{numerical link}.  For some pairs of Mori
contractions there are many numerical links while for others there are
relatively few.  However, not every numerical link is realisable by a
weak Fano 3--fold. For each numerical link further (often more delicate)
argument is required either to find a weak Fano realising that
numerical link or to prove that no such weak Fano exists.  This is the
geometric realisability question.

Jahnke--Peternell--Radloff~\cite{peternell2} give a complete list of numerical links
in the case that at \emph{most} one of the Mori contractions $\psi$ or
$\psi^{+}$ from \eqref{F:link} is of type E. Cutrone--Marshburn~\cite{cutrone:marshburn}
completed the numerical classification when both Mori contractions are
of type E. Takeuchi~\cite{takeuchi} considers the case where $\psi$
is of type $D$ and gives a complete classification including the
geometric realisability question when this del Pezzo fibration has
degree different from $6$; this augments (but also overlaps
considerably with) the geometric realisability studies contained in~\cite[Section~3]{peternell2}.  Thus the classification of rank two weak
Fano 3--folds where at least one Mori contraction is not of type E
is now close to complete; the 6? entries in~\cite[Table~7.7]{peternell2}
and the del Pezzo fibrations of degree $6$ are still outstanding.

\enlargethispage{\baselineskip}
The classification of rank two weak Fano 3--folds where both Mori
contractions are of type E is substantially less complete.
Cutrone--Marshburn~\cite{cutrone:marshburn} gives a list of 111 numerical links of type
E1--E1: meaning both Mori contractions are of type E1, that is, 
both $Y$ and $Y^{+}$ arise as the blowup of smooth curves in rank one Fano 3--folds $W$ and $W^{+}$. 
Of these 111 numerical links, they prove 11 to be geometrically realisable, 13 not to be
geometrically realisable, and leave 87 numerical links unsettled: see~\cite[5.1]{cutrone:marshburn}. 
Recall from \fullref{r:fano:blowup} that if $Y$ is a rank two weak Fano 3--fold with link type E1--E1 
(or more generally E1--**) then $H^{3}Y$ is torsion-free.
Hence we can always construct building blocks in the sense of \fullref{dfn:BLOCK} 
from any such weak Fano 3--fold.

More recently Blanc--Lamy~\cite{blanc:lamy} settled the geometric realisability question 
when the weak Fano $Y$ arises as the blowup of a space curve in $\CP^{3}$; 
this gives the existence of $13$ further pairs ($Y$ and its unique flop $Y^{+}$) 
of rank two semi-Fano 3--folds with small AC morphism. Very recently
Arap--Cutrone--Marshburn~\cite{arap:c:m} settled most of the geometric realisability questions in the cases 
when the weak Fano $Y$ arises as the blowup of a curve in: 
a smooth quadric in $\CP^{4}$, a pair of quadrics in $\CP^{5}$ or a del Pezzo 3--fold of degree $5$. 
These give another $8$, $6$ and $13$ pairs of examples of rank two semi-Fano 3--folds with small AC morphism respectively. 

So geometric realisability currently remains open for approximately $50$ of the $111$ numerical links of type E1--E1 
listed in~\cite{cutrone:marshburn}. 
The situation for other numerical links of type E--E is far more heavily constrained
with only relatively few numerical link types; for most of these
numerical links the geometric realisability question is already solved~\cite[5.2--5.7]{cutrone:marshburn}.

\subsubsection*{Rough enumeration of rank two weak Fano 3--folds with small AC morphism}
Let us give a rough enumeration of the number of deformation types generated by
the rank two weak Fano 3--folds with small AC morphism currently known
to exist.

If $Y$ and $Y'$ are smooth rank two weak Fano 3--folds belonging to the same 
deformation type then by Mori's deformation theory for extremal rays~\cite[Theorem~3.47]{mori:1982} 
both $Y$ and $Y'$ admit extremal rays of the same type except possibly in the cases E3/E4 
(the point being that E3 can degenerate to E4).
In particular, if both $Y$ and $Y'$ have small AC morphisms and different numerical links 
then $Y$ and $Y'$ do not belong to the same deformation type. 

Takeuchi~\cite[2.2--2.13]{takeuchi} gives a list of $33$ families of del Pezzo
fibred non-singular rank two Fano 3--folds with small AC morphism and shows
that none of them are deformation equivalent~\cite[Theorem~2.15]{takeuchi}.
He also lists their anticanonical models
$X$ and their flops $Y^{+}$; in almost all cases the anticanonical
model $X$ has only ordinary double points and therefore both $Y$ and
$Y^{+}$ have nodal AC model.  The number of curves contracted by the
anticanonical morphism $\varphi$ varies between $1$ and $46$.  In $19$
cases $Y^{+}$ is not itself del Pezzo fibred and is therefore not
deformation equivalent to any of the rank two weak Fanos in Takeuchi's list of
$33$. Hence we obtain $52$ distinct families of rank $2$ Fano
3--folds with small AC morphism from Takeuchi's work, almost all of
which have nodal AC model.  When $Y$ is a del Pezzo fibration of
degree $6$~\cite[A.2--A.4]{peternell2} provides 5 additional examples
(plus their flops which are different) and leaves open a further $8$
possibilities for del Pezzo fibrations of degree $6$.  Finally~\cite[Tables~7.5--7.7]{peternell2} yields $12=2+3+7$ cases (plus their flops) where
$\psi$ is a conic bundle  (with the geometric
realisability of $6$ further numerical links left open).

In total this gives us $84$ (that is, $52+ 2\times5 + 1\times2 + 2\times3 +2\times7$)
currently known deformation types generated by rank $2$ Fano 3--folds with small AC
morphism for which at least one of the Mori contractions $\psi$ and
$\psi^{+}$ is \emph{not} birational, and the majority of these have nodal AC
model.  In addition we have $26$ cases from~\cite{cutrone:marshburn}
where both Mori contractions $\psi$ and $\psi^{+}$ are birational and 
over $40$ further examples of link type E1--E1 from 
Arap--Cutrone--Marshburn~\cite{arap:c:m} and Blanc--Lamy~\cite{blanc:lamy}.

To summarise: 
we have at least $150$ deformation types arising from known families of rank two semi-Fano
3--folds for which the AC morphism is small (many of which have nodal
AC model) in addition to the $36$ deformation types of rank two genuine Fano 3--folds.
(There are additional deformation types which arise from the known rank two 
semi-Fano 3--folds for which the AC morphism is only semi-small,  
but we must take some care enumerating these because these may 
belong to the $150+$ deformation types above or may be 
deformation equivalent to a rank two genuine Fano 3--fold.)

This abundance of rank two semi-Fanos will allow us to
construct a large number of new compact \gtwo--holonomy manifolds in~\cite{chnp2}.  
If we use at least one building block built from one of the many 
semi-Fano 3--folds with nodal AC model, then we will be able to construct
\gtwo--holonomy manifolds containing a variety of different numbers of
rigid associative 3--folds.

\subsection*{Toric weak Fano 3--folds}

In this section we give an overview of the results one can obtain for
projective small (respectively crepant) resolutions of
toric Gorenstein terminal (respectively canonical) Fano 3--folds.
This will prove the existence of very many (hundreds of thousands of) non-singular toric weak Fano
3--folds.  We will also see that the set of non-singular toric weak Fano
3--folds with nodal AC model is essentially disjoint from the many Picard
rank $2$ semi-Fanos discussed above. The building blocks in 
Examples~\ref{exa:toricV22_a} and~\ref{exa:toricV22_b} 
both use one very particular toric  semi-Fano with nodal AC model.

Although by Batyrev's work~\cite{Batyrev:toric} there are only 18 deformation classes of non-singular toric Fano
3--fold  (see also the table in~\cite[Appendix~12.8]{fano:varieties}), 
there are many deformation classes
of singular toric Fano 3--fold as soon as one allows even relatively
mild singularities.  In the following whenever we refer to a toric
Fano 3--fold we shall mean a Gorenstein toric Fano 3--fold; these
automatically have at worst canonical singularities~\cite[2.2.5]{Batyrev:toric}.

Toric Fano 3--folds correspond (uniquely up to isomorphism) to
so-called \emph{reflexive polytopes}, see for example,
Nill's thesis~\cite[Chapters~1 and~2]{nill:thesis} for basic definitions in toric Fano
geometry.  Kreuzer--Skarke~\cite{Skarke} developed an algorithm to
classify reflexive polyhedra in arbitrary dimensions; as an
application of this algorithm they showed that there are 4319
$3$--dimensional reflexive polytopes, including the 18 that correspond
to non-singular toric Fano 3--folds.

A big advantage of Gorenstein toric Fano 3--folds compared to more
general Gorenstein canonical Fano 3--folds (where often no projective
crepant resolution exists, for example, any nodal quartic in $\CP^{4}$ with fewer than $9$ nodes) 
is that one can use toric geometry to prove
that \emph{any} toric Fano 3--fold admits a projective crepant
resolution.  Since every such crepant resolution is a non-singular toric
weak Fano 3--fold this proves the existence of at least $4301$
deformation families of toric weak Fano 3--fold.  In fact there are
many more such families because many singular toric Fano 3--folds
admit numerous non-isomorphic projective crepant resolutions;
moreover, all the projective crepant resolutions are toric and can be
enumerated purely combinatorially (see below).  
The topology of toric weak Fano 3--folds is also relatively straightforward:
as smooth toric varieties they have no cohomology in odd degree 
and their even cohomology is torsion-free. In particular 
we never have to worry about the condition $H^{3}(Y)$ being torsion-free.
These features make toric weak Fano 3--folds a very rich class of examples which 
nonetheless can be studied relatively easily.

\begin{prop}
\label{L:mpcp}
Any $3$--dimensional Gorenstein toric Fano variety $X$ admits at least
one projective crepant resolution $Y$; $Y$ is a non-singular toric weak
Fano 3--fold whose anticanonical model is $X$.
\end{prop}
\begin{remark}\hfill{}
\label{R:fan:refinement}
\begin{enumerate}
\item This result is not true for higher-dimensional Gorenstein toric
  varieties $X$; what is true is that there is a projective birational
  morphism $f\co X' \ra X$, such that $f$ is crepant and $X'$ is
  toric with only $\Q$--factorial terminal singularities
  \cite[Theorem~2.2.24]{Batyrev:toric}. Batyrev calls $f\co X' \ra X$ a
  \emph{maximal projective crepant partial desingularisation} of $X$
  or \emph{MPCP-desingularisation} for short.  \fullref{L:mpcp} is a special case of the existence of
  MPCP-desingularisations; since any $3$--dimensional Gorenstein toric
  variety with $\Q$--factorial terminal singularities must in fact be
  non-singular, any $3$--dimensional MPCP-desingularisation is non-singular.
\item Crepant resolutions of a toric variety $X$ correspond to fans
  $\Delta'$ refining the original fan $\Delta$ defining $X$.  The
  toric variety $X'$ associated to the fan $\Delta'$ is in general not
  projective; when the toric variety associated to the fan $\Delta'$
  is again projective, the fan $\Delta'$ is called a \emph{coherent
    crepant refinement} of $\Delta$.
\item Batyrev shows that any MPCP-desingularisation $f\co X' \ra X$
  defines a ``maximal projective triangulation'' of the reflexive
  polytope $P$ associated to $X$ and conversely that any maximal
  projective triangulation of the reflexive polytope $P$ determines a
  MPCP-desingularisation of $X$.  Since Gelfand, Kapranov and
  Zelevinsky~\cite{gelfand} already proved the existence of maximal
  projective triangulations (regular triangulations in their
  terminology) of any integral polyhedron $P$, the existence of
  MPCP-desingularisations (and hence projective crepant resolutions in
  the $3$--dimensional case) then follows immediately.
\item One can use the correspondence between projective crepant
  resolutions of a toric Fano 3--fold and maximal projective
  triangulations of the corresponding reflexive polytope to enumerate
  \emph{all} projective crepant resolutions of a given toric
  Gorenstein Fano 3--fold. Together with Tom Coates and Al Kasprzyk 
  we have used TOPCOM~\cite{topcom} in combination with
  PALP and Sage to find \emph{all} toric semi-Fano 3--folds up to
  isomorphism. 
  A more detailed description of this computation, the full data and a systematic treatment of 
  \gtwo--manifolds arising from them  will appear elsewhere~\cite{toric:g2}.
\end{enumerate}
\end{remark}

Many features of any crepant projective resolution of a toric Fano
3--fold can be read immediately from the associated reflexive
polytope. For example we have the following:
\begin{remark}
\label{r:picard:rk}
The Picard rank $\rho$ of \emph{any} crepant resolution of a
  toric Fano 3--fold equals the number of lattice points (including
  the origin) of the corresponding reflexive polytope minus~4.
  Hence from Kreuzer--Skarke~\cite[Table~2]{Skarke} we have that for a non-singular toric
  weak Fano 3--fold $Y$, $\rho=b^{2}(Y)$ can attain any value between
  $2$ and $35$ except $32$ and $33$.  
\end{remark}
We can also recognise the various flavours of
non-singular toric weak Fano $Y$ from the geometry of the reflexive polytope
associated with its (Gorenstein toric Fano) anticanonical model $X$.

For toric Fano 3--folds with small AC morphism we have:

\begin{lemma}[Terminal toric Fano 3--folds]\hfill{}
\label{L:terminal:toric}
\begin{enumerate}
\item A toric Fano 3--fold $X$ is terminal if and only if all facets
  of its reflexive polytope are either standard triangles or standard
  parallelograms.
\item The only singularities of a terminal toric Fano 3--fold are
  ordinary double points and the number of ODPs of $X$ is equal to the
  number of parallelograms in its reflexive polytope. In particular,
  every toric weak Fano 3--fold with small AC morphism has nodal AC model.
\item Every terminal toric Fano 3--fold $X$ admits at least one small
  projective resolution $Y$; $Y$~is a non-singular toric semi-Fano
  3--fold. Conversely every non-singular toric semi-Fano 3--fold $Y$ with nodal
  AC model arises as a small projective resolution of a terminal (nodal) toric
  Fano 3--fold $X$.
\end{enumerate}
\end{lemma}

\begin{proof}
  For (i) and (ii) see the thesis of Nill~\cite[4.2.4 and~4.3.1--4.3.2]{nill:thesis}.  (iii) is a special case of
  \fullref{L:mpcp}; see also the remark below.
\end{proof}

\begin{remark}
  In the special case of a terminal toric (and therefore nodal) Fano 3--fold $X$ any ``crepant
  refinement'' of the reflexive polytope of $X$ as in \fullref{R:fan:refinement} arises as follows: for each parallelogram
  facet in the reflexive polytope pick one of its two diagonals and
  make a new polytope by adding the chosen diagonals as additional edges to
  the reflexive polytope.  Clearly there are $2^{e}$ such refinements
  where $e$ is the number of parallelograms (by %
  \fullref{L:terminal:toric}(ii) parallelograms correspond to the nodes
  of~$X$); each such refinement gives a (toric) small but not
  necessarily projective resolution of~$X$. By %
  \fullref{L:terminal:toric}(iii)
  at least one of these small resolutions is projective.
\end{remark}

\begin{corollary}\hfill{}
\label{C:small:weak:toric:fano}
\begin{enumerate}
\item There are precisely 82 singular toric Fano 3--folds with
  terminal singularities.
\item The Picard rank $\rho$ of a terminal toric Fano 3--fold $X$ can
  be $1$, $2$, $3$ or $4$.
\item The Picard rank $\rho$ of a toric semi-Fano 3--fold with nodal AC
  model takes all values between $2$ and $11$.
\item The genus $g$ of a toric semi-Fano 3--fold with nodal AC model takes all
  values in $\{11, \ldots ,25\} \cup \{28\}$.
\item The defect $\sigma$ of a terminal toric Fano  3--fold takes all
  values in $\{1, \ldots ,7\} \cup \{9\}$.
\item The number $e$ of exceptional $(-1,-1)$ curves of a toric
  semi-Fano 3--fold with nodal AC model takes all values in $\{1, \ldots, 9\}
  \cup\{12\}$.
  \item
  Every toric semi-Fano 3--fold $Y$ with nodal AC model is rigid, that is,
$$H^{1}(Y,\mathcal{T}_{Y})=(0).$$
  \item
  There are precisely 1009 deformation types of toric semi-Fano 3--fold with nodal AC model.
\end{enumerate}
\end{corollary}

\begin{proof}
  (i) follows either from Nill's thesis~\cite{nill:thesis} or from the
  Kreuzer--Skarke classification of reflexive polytopes in three
  dimensions and the characterisation of the terminal ones from
  \fullref{L:terminal:toric}(i).  (ii--vi) now follow from an examination
  of the $82$ possible terminal reflexive polytopes, for example, see the list
  of terminal toric Fano 3--folds on the Graded Rings database~\cite{grdb}.
  (vii) follows immediately from Ilten's thesis \mbox{\cite[Corollary~4.2.6]{ilten:thesis}}.
  For (viii) we first enumerate all projective small resolutions of the 82 terminal reflexive polytopes. 
  Next we identify projective small resolutions of  a given terminal polytope which 
  differ by a lattice automorphism. This yields the number of non-isomorphic 
  projective small resolutions for each polytope. The details of these calculations will appear in~\cite{toric:g2}.
  The total number of non-isomorphic projective small resolutions  turns out to be 1009:   
  since by (vii) all these varieties are rigid the number of deformation types is also equal to 1009.
\end{proof}

\begin{remark}[Toric semi-Fano 3--folds with nodal AC model and near-extremal Picard rank]
\label{R:picard:rank:toric:small}\hfill{}
\begin{enumerate}
\item
Let us a consider toric weak Fano 3--fold $Y$ with the minimal possible Picard rank $\rho=2$ (we assume $Y$ 
is not already Fano so $\rho\ge 2$).
By \fullref{r:picard:rk} the reflexive polytope corresponding to its 
AC model $X$ has exactly 6 lattice points. 
Consulting the classification we find that among the 82 terminal reflexive polytopes there is precisely one such polytope.
The corresponding terminal toric Fano
  3--fold $X \subset \CP^{4}$ is the projective cone over a non-singular
  quadric $Q \simeq \CP^{1}\times \CP^{1} \subset \CP^{3}$  and has two (isomorphic) projective
  small resolutions as described in  \ref{E:nodal:quadric}.
  Therefore up to isomorphism there is precisely one toric 
  semi-Fano 3--fold with $\rho=2$ and nodal AC model; as remarked previously it has 
  index~$3$. In particular, only this toric semi-Fano 3--fold appears
  in our earlier count of over $150$ semi-Fano 3--folds with $\rho=2$ and small AC morphism.

\item 
By \fullref{C:small:weak:toric:fano}(iii) any toric semi-Fano 3--fold $Y$
with nodal AC model has Picard rank at most $11$. By counting lattice points again
  the classification of terminal reflexive polytopes shows that if the
  Picard rank of $Y$ equals $11$ then its anticanonical model $X$ is the
  unique terminal toric Fano 3--fold $X_{20}$ of degree~$20$; $X_{20}$ has Picard
  rank $2$ and defect~$9$.  (We know $X_{20}$ could not have Picard rank $1$ because 
  there is no smooth rank $1$ Fano of degree $20$ which could degenerate to $X_{20}$.)
  $X_{20}$ corresponds to polytope $2355$ in the
  Sage list of $3$--dimensional reflexive polytopes; this polytope
  contains $12$ parallelograms and hence $X$ contains $12$ nodes.
  Using TOPCOM to count regular triangulations of the polytope we find
  that this polytope admits 3608 (out of all $2^{12}=4096$ possible small resolutions) 
  projective small resolutions and that these lie in  $125$ distinct isomorphism classes.
  Apart from these $125$ isomorphism classes of  toric semi-Fano
  3--folds with nodal AC model and Picard rank $11$, all other toric
  semi-Fano 3--folds with nodal AC model have Picard rank between $2$
  and $10$.  
\item Similarly, any toric semi-Fano 3--fold with nodal AC model and
  Picard rank equal to $10$ has anticanonical model the unique
  terminal toric Fano 3--fold $X_{22}$ of degree $22$; $X_{22}$~has Picard
  rank $1$ and defect $9$. It corresponds to polytope $1942$ in the
  Sage list of $3$--dimensional reflexive polytopes; this polytope
  contains $9$ parallelograms and hence $X$ contains $9$ nodes.
  Note that the number of nodes of $X$ is equal to its defect. 
  Using TOPCOM to count regular triangulations of the polytope we find that
  all $512$ small resolutions of this polytope are projective and 
  these consist of $84$ distinct isomorphism classes.  Building blocks
  constructed from this particular polytope were discussed in
  detail in Examples~\ref{exa:toricV22_a} and~\ref{exa:toricV22_b}. 
  We selected this particular 
  polytope because it has maximal defect $\sigma=9$ 
  and because the polytope is self-dual.
\end{enumerate}
\end{remark}

Similarly we can recognise more general toric semi-Fano 3--folds from
the geometry of the associated reflexive polytope.

\begin{lemma}[Toric semi-Fano 3--folds]\hfill{}
\label{L:toric:semi-small}
\begin{enumerate}
\item A toric weak Fano 3--fold $Y$ is semi-Fano if and only if the
  reflexive polytope corresponding to the (toric Fano) anticanonical
  model $X$ of $Y$ has no facets that contain lattice points strictly
  in their interior, that is, every lattice point on any facet lies on an
  edge of the facet. In this case by a slight abuse of terminology we
  will say that the reflexive polytope (or the singular toric Fano
  3--fold $X$) is semi-small.
\item Every semi-small toric Fano 3--fold $X$ admits at least one
  projective crepant resolution $Y$; $Y$ is a non-singular toric
  semi-Fano 3--fold.
\end{enumerate}
\end{lemma}

\begin{proof}
(i) is obvious; (ii) is a special case of \fullref{L:mpcp}.
\end{proof}

\begin{corollary}\quad
\begin{enumerate}
\item There are $799$ semi-small \,$3$--dimensional reflexive polytopes
  (excluding the $100=18+82$ corresponding to non-singular or terminal toric
  Fanos). These are precisely the polytopes for which every facet
  contains no interior lattice points, but that contain at least one
  boundary lattice point that is not a vertex of the polytope.
  $435$ of these polytopes contain at least one standard parallelogram.
\item The Picard rank $\rho$ of a semi-small toric Fano 3--fold $X$
  can be $1$, $2$, $3$ or $4$.
\item The Picard rank $\rho$ of a non-singular toric semi-Fano
  3--fold $Y$ can be any integer between $2$ and $15$.
\item The genus $g$ of a toric semi-Fano 3--fold $Y$ can be any
  integer between $7$ and~$29$.
\item The defect $\sigma$ of a semi-small toric Fano 3--fold $X$ is at
  least $1$ and at most $13$.
\item There are $526\,130$ isomorphism classes of non-singular toric semi-Fano 3--fold 
(including the $18+ 1009$ corresponding to smooth toric Fanos and toric Fanos with terminal AC model); 
$435\,459$ of these are rigid. 
\end{enumerate}
\end{corollary}

\begin{proof}
  (i) follows from the Kreuzer--Skarke explicit list of all $4319$
  reflexive polytopes and the characterisation given in
  \fullref{L:toric:semi-small}(i).  
  (ii--v) follow from computations based on
  the explicit list of semi-small reflexive polytopes.
The first part of (vi) follows by enumerating all projective crepant resolutions of the semi-small reflexive polytopes
as described in the proof of \fullref{C:small:weak:toric:fano}(viii).
It is not the case that every toric semi-Fano 3--fold is rigid; \fullref{ex:ky:nonample}(i) 
exhibits a toric semi-Fano 3--fold that is not rigid.
To determine which toric semi-Fano 3--folds are rigid first we compute 
$h^{0} = \dim{H^{0}(Y,\mathcal{T}_{Y})}$ by finding the number of Demazure roots 
of the associated fan, compare with Oda's book~\cite[Corollary~3.13]{oda}.
To compute $h^{1} = \dim{H^{1}(Y,\mathcal{T}_{Y})}$ we use the fact 
that 
\[
h^{1}-h^{0} = - \chi (Y,\mathcal{T}_{Y}) = 19 - \rho - g +h^{2,1}(Y)
\]
where $\rho$ and $g$ are the Picard rank and genus respectively of the semi-Fano~$Y$, 
compare with Mukai~\cite[Section~4]{mukai:quartics}. (In the toric case we always have $h^{2,1}=0$).
The detailed calculations will appear in~\cite{toric:g2}.
\end{proof}

In particular, there are at least $435\,459$ deformation types of toric
semi-Fano 3--folds (including the $1027$ corresponding to smooth toric Fanos
and toric semi-Fanos with terminal AC model). More effort would be needed to
determine how many deformation types are realised by the remaining non-rigid
toric semi-small Fano 3--folds.

\subsection*{Terminal Fano 3--folds via degenerations of non-singular Fano 3--folds}

Given any non-singular semi-Fano 3--fold $Y$ with small AC morphism, we can
associate a deformation class of non-singular Fano 3--folds as follows.  By
\fullref{R:ac:model}(ii), the anticanonical model $X$ of $Y$ is a
terminal Gorenstein Fano 3--fold which thanks to Namikawa's smoothing
result (\fullref{T:fano:smooth}) is smoothable by a flat
deformation to a family of non-singular Fano 3--folds $X_{t}$. The
anticanonical degrees and indexes of $Y$, $X$ and $X_{t}$ are
all the same and $\rho(X)=\rho(X_{t})$ but $\rho(Y) = \rho(X) +
\sigma$ where $\sigma$ is the defect of $X$.  For instance in the case
where $Y$ is a rank $2$ semi-Fano 3--fold with small AC morphism (as
considered earlier) this associates to $Y$ one of the $17$ deformation
classes of non-singular rank $1$ Fano 3--folds from the Iskovskih
classification. 

Semi-Fano 3--folds associated with degenerations of 
a cubic in $\CP^{4}$ to a nodal cubic are completely understood; see below for a 
summary of these results. 
On the other hand weak Fano 3--folds associated with say degenerations of a quartic 
in $\CP^{4}$ are still very far from understood in general; see below for some further 
discussion of nodal quartics and their small resolutions.

\subsubsection*{Weak del Pezzo 3--folds from nodal cubics}
\label{sec:nodal-cubics}
From our earlier remarks 
about the behaviour of the index and degree under smoothing and small resolution,  we see 
immediately that because a smooth cubic has index $2$ and degree $24$, any 
semi-Fano 3--fold arising as the small resolution of a nodal cubic also has index $2$ and degree $24$.
In particular, they are all weak del Pezzo 3--folds.
Finkelnberg--Werner~\cite{finkelnberg1989small}
understood how many nodes can occur on a degeneration of a smooth cubic, what defects occur and in each
case how many of the small resolutions are projective.  
Their results demonstrate clearly how one single deformation class of smooth del Pezzo 3--folds 
can give rise to a much larger number of deformation classes of weak del Pezzo 3--folds.

Finkelnberg--Werner show that the number of
nodes $k$ can take any value up to $10$ and the defect any value up to
$5$. \fullref{table:cubics} lists the possible number of nodes $e$,
the defect~$\sigma$, the number of projective small resolutions $s$,
the number of planes $P$ contained in the nodal cubic and $b^{3}(Y)$
denotes the third Betti number of any projective small resolution of
the nodal cubic (if any exists); the latter is computed using
\eqref{E:weak:fano:b3} and the fact that $b^{3}=10$ for a non-singular cubic
3--fold.

\newlength{\extrasep}
\setlength{\extrasep}{2mm}

\begin{table}[ht!]
\[
\begin{array}[t]{cccccc}\toprule
e & \sigma & s & P &\rho(Y) & b^{3}(Y)\\ \midrule
0 & 0 & 0 & 0 & 1& 10\\ \addlinespace[\extrasep]
1& 0 & 0 & 0 & - & -\\ \addlinespace[\extrasep]
2& 0 & 0 & 0 & -& - \\ \addlinespace[\extrasep]
3& 0 & 0 & 0 & -&- \\ \addlinespace[\extrasep]
\multirow{2}{*}{4} &0 & 0 & 0 & -& -\\ 
& 1 & 2 & 1 & 2 & 4 \\ \addlinespace[\extrasep]
5& 1 & 0 & 1 & -& -\\ \addlinespace[\extrasep]
\multirow{3}{*}{6} 
& 1 & 0 & 1 & - & -\\ 
& 1 & 2 & 0 & 2 & 0\\ 
& 2 & 6 & 2 & 3& 2\\ \addlinespace[\extrasep]
\multirow{2}{*}{7}
& 2 & 6 & 2 & 3& 0 \\ 
& 2 & 0 & 3 & - & -\\ \addlinespace[\extrasep]
8& 3 & 24 & 5 &  4 & 0\\ \addlinespace[\extrasep]
9& 4 & 102 & 9 & 5 & 0 \\ \addlinespace[\extrasep]
10& 5 & 332 & 15 & 6 & 0\\ \bottomrule
\end{array} \]
\caption{The possible nodal degenerations of cubic 3--folds; $e$ denotes the number of ODPs, 
$\sigma$ the defect, $s$ the number of projective small resolutions and $P$ the number of planes contained in the 
nodal cubic. $\rho(Y)$ and $b^{3}(Y)$ denote the Picard rank and third Betti number of any 
projective small resolution $Y$ of the nodal cubic (when one exists).}
\label{table:cubics}
\end{table}

\begin{remark}\hfill{}
\begin{enumerate}
\item
A nodal cubic 3--fold $X \subset \CP^{4}$ is nonrational if and only if it is smooth 
by Clemens--Griffiths~\cite[Theorem~13.12]{clemens:griffiths}.
Any projective small resolution $Y$ of a nodal cubic $X$ therefore has no torsion in $H^{3}(Y)$ 
(recall \fullref{R:torsion:weak:fano}) and hence gives rise to a building block $Z$ in the sense of \fullref{dfn:BLOCK} via the construction of \fullref{prop:block_from_weak}.
\item All the examples in \fullref{table:cubics} with $\rho(Y) >2$
  have nodal AC model and are not already included in any of the
  classes described earlier in the paper; to see this we only need
  note the following: a non-singular cubic and hence any nodal degeneration
  has Picard rank $\rho=1$ and anticanonical degree $24$, whereas the
  classification of toric terminal Fano 3--folds shows none of the
  three degree $24$ examples has Picard rank $1$.
\item \eqref{E:bound:nodes} applied to a degeneration of a non-singular
  cubic yields $e \le 5 + 20 - 1 = 24$, whereas in fact we have $e \le
  10$. So in this case the bound from \eqref{E:bound:nodes} is quite
  far from being sharp.
\item \eqref{E:weak:fano:b3} and non-negativity of $b^{3}(Y)$
  immediately implies $e-\sigma \le 5$; \fullref{table:cubics} shows
  that there are $5$ different possible combinations of $e$ and
  $\sigma$ realising $e-\sigma =5$ (which forces $b^{3}(Y)=0$).
\item From \fullref{table:cubics} we see that $3 \le e - \sigma \le
  5$ for any nodal cubic that admits a projective small resolution.
\item In \fullref{table:cubics} some (sometimes many) of the $s$
  small projective resolutions of a given nodal cubic may give rise to
  projective varieties that are abstractly isomorphic; if the nodal
  cubic $X$ admits a nontrivial discrete group of automorphisms then
  this group acts on the set of all small resolutions and different
  small resolutions in the same orbit are isomorphic. For example, the
  unique (up to projective equivalence) nodal cubic with $10$ nodes 
  called the Segre cubic has automorphism group the symmetric group
  $S_{6}$. In~\cite{finkelnberg1987small} Finkelnberg  showed that there are $13$ different orbits of
  $S_{6}$ acting on the set of all $2^{10}=1024$ small resolutions of
  the Segre cubic; $6$ of these orbits consist of projective small
  resolutions while $7$ contain only non-projective small resolutions.
  In particular,  we obtain $6$ different isomorphism classes of semi-Fano 3--fold with 
  index $2$ (weak del Pezzo 3--fold), degree $24$, Picard rank $6$ and nodal AC morphism.
  For other nodal cubics with
  close to the maximal number of nodes the number of non-isomorphic
  projective small resolutions does not seem to have been determined.
\end{enumerate}
\end{remark}

\subsubsection*{Semi-Fano 3--folds from nodal quartics}
Examples~\ref{exa:quartic_w_plane} to~\ref{exa:quartic_w_22} all give examples of defect $1$ semi-Fano 
3--folds arising from projective small resolutions of nodal quartics in $\CP^{4}$. 
\fullref{exa:burkhardt_quartic} is a defect $15$ weak Fano 3--fold associated 
with a nodal quartic in $\CP^{4}$  (with the maximal number of nodes $e=45$; 
moreover, $15$ is the maximal possible defect for a terminal quartic 3--fold: see below).
There currently does not seem to be a good understanding of semi-Fano 3--folds 
associated with nodal quartics when the defect is not either $1$ or close to the maximum $15$.
Even for the maximal defect $\sigma =15$ it does not seem that the number of projective small resolutions 
of the Burkhardt quartic has been determined. (Recall it has at least one projective small resolution 
and exactly $2^{45} \simeq 3.5 \times 10^{13}$ Moishezon but not necessarily projective small resolutions).

The following statement summarises some of the main known results about nodal quartics
and their defects and projective small resolutions.
\begin{theorem}
\label{t:nodal:quartics}
Let $X$ be a nodal quartic in $\CP^{4}$, let $e$ denote the number of nodes of $X$ and $\sigma(X)$ its defect.
\begin{enumerate}
\item
If $e<9$ then $\sigma=0$ and hence $X$ admits no projective small resolutions.
\item
If $e=9$ then $\sigma=0$ if and only if $X$ contains no planes $\Pi$. In particular, 
a general quartic with $e=9$ admits no projective small resolutions.
If $e=9$ and $X$ contains a plane $\Pi$ then $\sigma=1$ and blowing up $\Pi$ 
in $X$ yields a projective small resolution as in \fullref{exa:quartic_w_plane}.
\item
If $e<12$ and $X$ contains no planes then $\sigma=0$ and hence $X$ admits no projective small resolutions.
\item
If $e=12$ then $\sigma=0$ unless $X$ contains a quadric surface. 
A sufficiently general quartic containing an irreducible quadric $Q_{2}^{2}$ 
has precisely $12$ nodes all contained in $Q_{2}^{2}$ and has $\sigma=1$. 
Blowing up $Q_{2}^{2}$ yields a projective small resolution as in \fullref{exa:quartic_w_quadric}.
\item
$e\le 45$ with equality if and only if $X$ is projectively equivalent to the Burkhardt quartic
as in \fullref{exa:burkhardt_quartic}.
\item
$\sigma \le 15$ with equality if and only $X$ is projectively equivalent to the Burkhardt quartic.
Moreover, $\sigma \le 10$ if $X$ contains no planes.\end{enumerate}
\end{theorem}

\begin{proof}
(i) and (ii) are proved in Cheltsov~\cite[Theorems~2 and~5]{cheltsov:nodal:quartics}.
(iii) and (iv) are proved in Shramov~\cite[Theorem~1.3]{shramov}. (v): $e \le 45$ was proved in  
Varchenko~\cite{varchenko}; the case of equality was treated in de~Jong--Shepherd-Barron--Van de Ven~\cite{dejong}.
(vi) is proved in Kaloghiros~\cite[Theorem~1.1]{kaloghiros:defect}; see the erratum for a correction to the original 
claim of Theorem 1.1.(ii).
\end{proof}

\begin{remark}
The special class of nodal \emph{determinantal quartics} has been studied in some detail. 
A determinantal quartic is a hypersurface 
in $\CP^{4}$ given as the zero-locus of the determinant of a $4 \times 4$ matrix of linear forms in $[z_{0},\ldots ,z_{4}]$.
A determinantal quartic is never smooth but generically has only nodes; this makes determinantal quartics 
a good source of nodal quartics. A nodal determinantal quartic has $20 \le e \le 45$ and the generic 
one has $e=20$.  The Burkhardt quartic is determinantal. 
Every determinantal quartic is rational and hence any projective small resolution $Y$ has no torsion in $H^{3}(Y)$.
Pettersen's thesis~\cite{pettersen1998nodal} used a particular rationalisation to study nodal determinantal quartics. 
He gave classification results for nodal determinantal quartics with $e \ge 42$ and showed any such quartic admits 
at least one projective small resolution. Such resolutions are semi-Fano 3--folds with $H^{3}(Y)$ torsion-free 
and thus give rise to building blocks in the sense of \fullref{dfn:BLOCK}.
Pettersen~\cite[Section~6.2]{pettersen1998nodal} 
constructed determinantal quartics with $e=40$ and $\sigma=10$ which contain no plane; 
thus the defect bound from \fullref{t:nodal:quartics}(vi) is sharp.
\end{remark}

\bibliographystyle{gtart}
\bibliography{link}

\end{document}